\documentclass[reqno]{amsart}

\usepackage{hyperref}
\usepackage{amsmath}
\usepackage{amssymb}
\usepackage{amsfonts}
\usepackage{amsthm}
\usepackage{latexsym}
\usepackage{esint}
\usepackage{mathtools}
\usepackage{mathrsfs}
\usepackage{graphicx}
\usepackage{wrapfig}
\usepackage{bbm}
\usepackage{color}
\usepackage{bm}
\usepackage[top=1in, bottom=1.25in, left=1.10in, right=1.10in]{geometry}
\usepackage{todonotes}
\usepackage{stackengine}
\allowdisplaybreaks

\newcommand\ocirc[1]{\ensurestackMath{\stackon[1pt]{#1}{\mkern2mu\circ}}}

\DeclareMathOperator{\regkap}{\mathcal R^\varrho}

\newcommand{\Testzeta}{\mathscr{F}_{\eta}}
\newcommand{\Testetar}{\mathscr{F}_{\eta^\varrho}}

\newtheorem{theorem}{Theorem}[section]
\newtheorem{lemma}[theorem]{Lemma}

\newtheorem{corollary}[theorem]{Corollary}

\theoremstyle{definition}
\newtheorem{definition}[theorem]{Definition}
\newtheorem{example}[theorem]{Example}

\theoremstyle{remark}
\newtheorem{remark}[theorem]{Remark}

\DeclareMathOperator*{\tr}{tr}

\numberwithin{equation}{section}

\newcommand{\bx}{\mathbf{x}}
\newcommand{\by}{\mathbf{y}}
\newcommand{\bz}{\mathbf{z}}
\newcommand{\bq}{\mathbf{q}}
\newcommand{\bu}{\mathbf{u}}
\newcommand{\bfu}{\mathbf{u}}
\newcommand{\bfv}{\mathbf{v}}
\newcommand{\bff}{\mathbf{f}}
\newcommand{\vu}{\mathbf{u}}
\newcommand{\bv}{\mathbf{v}}
\newcommand{\bn}{\mathbf{n}}

\newcommand{\dy}{\, \mathrm{d}\mathbf{y}}
\newcommand{\dd}{\,\mathrm{d}}
\newcommand{\dq}{\, \mathrm{d} \mathbf{q}}
\newcommand{\dx}{\, \mathrm{d} \mathbf{x}}
\newcommand{\dxs}{\, \mathrm{d} \mathbf{x}\, \mathrm{d}\sigma}
\newcommand{\dt}{\, \mathrm{d}t}
\newcommand{\ds}{\, \mathrm{d}\sigma}

\newcommand{\dxt}{\,\mathrm{d}\mathbf{x}\, \mathrm{d}t}

\newcommand{\Div}{\mathrm{div}_{\mathbf{x}}}
\newcommand{\divx}{\mathrm{div}_{\mathbf{x}}}

\newcommand{\divqi}{\mathrm{div}_{\mathbf{q}_i}}

\newcommand{\nabx}{\nabla_{\mathbf{x}}}
\newcommand{\naby}{\nabla_{\mathbf{y}}}
\newcommand{\nabq}{\nabla_{\mathbf{q}}}

\newcommand{\nabqi}{\nabla_{\mathbf{q}_i}}
\newcommand{\nabqj}{\nabla_{\mathbf{q}_j}}
\newcommand{\Delx}{\Delta_{\mathbf{x}}}

\newcommand{\R}{\mathbb{R}}

\newcommand{\Oeta}{\Omega_{\eta(t)}}

\begin{document}

\title[An incompressible polymer fluid interacting with a Koiter shell]
{An incompressible polymer fluid interacting with a Koiter shell.}

\author{Dominic Breit}
\address{Department of Mathematics, Heriot-Watt University, Edinburgh, EH14 4AS, United Kingdom}
\email{d.breit@hw.ac.uk}

\thanks{The authors would like to thank S. Schwarzacher and E. S\"uli for valuable suggestions.}

\author{Prince Romeo Mensah}
\address{Department of Mathematics, Imperial College, London SW7 2AZ, United Kingdom
\&
Gran Sasso Science Institute, Viale F. Crispi, 7. 67100 L'Aquila, Italy}
\email{p.mensah@imperial.ac.uk; romeo.mensah@gssi.it}

\subjclass[2010]{76Nxx; 76N10; 35Q30 ; 35Q84; 82D60}

\date{\today}


\keywords{Incompressible Navier--Stokes--Fokker--Planck system, FENE model, Fluid-Structure interaction, Koiter shell}

\begin{abstract}
We study a mutually coupled mesoscopic-macroscopic-shell system of equations modeling a dilute incompressible polymer fluid which is evolving and interacting with a flexible shell of Koiter type. The polymer constitutes a solvent-solute mixture where the solvent is modelled on the macroscopic scale by the incompressible Navier--Stokes equation and the solute is modelled on the mesoscopic scale by a Fokker--Planck equation (Kolmogorov forward equation) for the probability density function of the bead-spring polymer chain configuration. This mixture interacts with a nonlinear elastic shell which serves as a moving boundary of the physical spatial domain of the polymer fluid.
We use the classical model by Koiter to describe the shell movement which yields a fully nonlinear fourth order hyperbolic equation. Our main result is the 
existence of a weak solution to the underlying system
 which exists until the Koiter energy degenerates or the flexible shell approaches a self-intersection.
\end{abstract}

\maketitle

\section{Introduction}
\label{sec:intro}
\noindent
On the one hand, fluid-structure interactions are common physical phenomena yet mathematically challenging problems with applications in aeroelasticity \cite{dowell2015modern},  biomechanics \cite{fluid2014bodnar} and hydrodynamics \cite{chakrabarti2002theory} amongst others. On the other hand, the huge industrial application of the interactions between polymer molecules and fluids such as in the production of paints, lubricants, plastics as well as in the processing of food stuff \cite{bird1987dynamics}, makes the analysis of polymeric fluids very important. Therefore, from a mathematical, physical and commercial point-of-view, the analysis of the mutual interaction of all three elements, i.e. fluid, structure and polymer molecules is crucial. 
\\
We consider in this work, the evolution of a dilute three-dimensional incompressible polymeric fluid in a spatial domain that is changing with respect to time. The displacement of the boundary is prescribed via the two-dimensional mid-section of the flexible Koiter shell
whose energy is a nonlinear function of the first and second fundamental
forms of the moving boundary. We prove the existence of a weak solution to the coupled fluid-kinetic system, given by the incompressible Navier--Stokes--Fokker--Planck sytem, which is interacting with an elastic Koiter shell. The existence time  is only restricted once the shell approaches a self-intersection.
\\
Existence of a solution to the Fokker--Planck equation for a given solenoidal velocity field incorporating the center-of-mass diffusion term has been established by El-Kareh and Leal \cite{el1989existence} independently of the Deborah number. 
The incompressible Navier--Stokes--Fokker--Planck system  for polymeric fluids including center-of-mass diffusion  has been studied considerably. See for example, the works by Barrett, Schwab \& S\"uli \cite{barrett2005existence}, Barrett \& S\"uli \cite{ barrett2007existence, barrett2008existence, barrettSuli2011existence, barrett2012existenceMMMAS, barrett2012existenceJDE},
as well as by 
Gwiazda, Luk\'{a}\v{c}ov\'{a}-Medvidov\'{a}, Mizerov\'{a} \&  \'{S}wierczewska-Gwiazda \cite{gwiazda2018existence}
and 
Luk\'{a}\v{c}ov\'{a}-Medvi\v{d}ov\'{a}, Mizerov\'{a}, Ne\v{c}asov\'{a} \& Renardy \cite{lukacova2017global}
 for the kinetic Peterlin model  with
a nonlinear spring law for an infinitely extensible spring. 
All these results derive global-in-time weak solutions for variations of the incompressible Navier--Stokes equation coupled with the Fokker--Planck equation.
On the other hand, a unique local-in-time strong solution for the center-of-mass system was first shown to exist by Renardy \cite{renardy1991existence}. Unfortunately, \cite{renardy1991existence} excludes the physically relevant FENE dumbbell models. The local theory was then revisited by Jourdain, Leli\`evre \& Le Bris \cite{jourdain2004existence} for the stochastic FENE model for the simple Couette flow
and by E, Li \& Zhang \cite{e2004well} who analysed the  incompressible Navier--Stokes equation coupled with a system of SDEs describing the configuration of the spring. The corresponding deterministic system, where instead the incompressible Navier--Stokes equations are coupled with the Fokker--Planck equation, was studied by Li, Zhang \& Zhang \cite{li2004local} and Zhang \& Zhang \cite{zhang2006local}. Constantin proved the existence of Lyapunov functionals and smooth solutions in \cite{constantin2005nonlinear} and then derived global-in-time strong solution for the 2-D system in \cite{constantin2007regularity} together with Fefferman, Titi \& Zarnescu. \\
The analysis is significantly harder without center-of-mass diffusion since the Fokker--Planck equation becomes a degenerate parabolic equation which behaves like an hyperbolic equation in the space-time variable.  A global weak solution result to the incompressible Navier--Stokes--Fokker--Planck system for the FENE dumbbell model without center-of-mass diffusion
was recently  achieved in the seminal paper \cite{masmoudi2013global} by Masmoudi.  The main difficulty is to pass to the limit in the drag term of the Fokker--Planck equation which does not have any obvious compactness properties.
Earlier global weak solution results without center-of-mass diffusion include the work by Lions \& Masmoudi \cite{lions2000global} for Oldroyd models,  Lions \& Masmoudi \cite{ lions2007global} who studied the corotational case, and Otto \& Tzavaras \cite{otto2008continuity} who studied weak solutions for the stationary system.
Masmoudi \cite{masmoudi2008well} also constructed a local-in-time strong solution to the incompressible Navier--Stokes--Fokker--Planck system for the FENE dumbbell model without center-of-mass diffusion
in \cite{masmoudi2008well}. Furthermore, the solution is global near equilibrium, see also  Kreml \& Pokorn\'y \cite{klainerman1981singular}. Further results on local strong solutions where proved by  Luo \& Yin \cite{luo2017global} and Breit \& Mensah \cite{breit2019local}.
\\
With respect to fluid-structure problems, the analysis of weak solutions to incompressible viscous fluids interacting with lower-dimensional linear elastodynamic equations has been studied by Chambolle, Desjardins, Esteban and Grandmont in \cite{chambolle2005existence}, by Grandmont \cite{grandmont2008existence}, 
Hundertmark-Zau\v{s}kov\'{a}, Luk\'{a}\v{c}ov\'{a}-Medvi\v{d}ov\'{a} \&
  Ne\v{c}asov\'{a}
\cite{hundertmark2016existence}, Lengeler \& R\ocirc{u}\v{z}i\v{c}ka \cite{lengeler2014weak} and by Muha and \v{C}ani\'{c} in \cite{muha2014existence, muha2013existence}, just to list a few. 
In particular, the existence of a weak solution for the  three-dimensional viscous incompressible fluid modelled by the Navier--Stokes equations which is interacting with a flexible elastic plate located on one
part of the fluid boundary was shown by Chambolle et al in \cite{chambolle2005existence}. This solution exists so long as the moving part of the structure does not touch the fixed part of the fluid boundary.
By using a singular limit argument, the existence of a weak solution to the incompressible Navier--Stokes equation coupled with a plate in flexion was constructed by Grandmont in \cite{grandmont2008existence} as the coefficient  modelling the viscoelasticity of the plate tends to zero.
In \cite{hundertmark2016existence}, Hundertmark-Zau\v{s}kov\'{a} et al studied the existence of a weak solution to a power-law viscosity fluid-structure interaction problem for shear-thickening flows. Again, the solution  exists
until a contact of the elastic boundary with a fixed boundary part is made. Lengeler \& R\ocirc{u}\v{z}i\v{c}ka also studied in \cite{lengeler2014weak},  the interaction of an incompressible Newtonian fluid, modelled by the Navier--Stoke equation, with a linear elastic shell of Koiter. Here, the middle surface of the shell serves as the mathematical boundary of the three-dimensional fluid domain. The weak solution is shown to exist so long as the magnitude of the shell's displacement stays below a bound that rules out self-intersection. In \cite{muha2013existence} however, Muha and \v{C}ani\'{c} use a semi-discrete, operator splitting numerical scheme to show the existence of weak solution of a fluid-structure coupled system governed by the two-dimensional incompressible Navier--Stokes equations, while the elastodynamics of the cylindrical wall is modelled by the one-dimensional cylindrical linear Koiter shell model. The solution exists as long as the cylinder radius is greater than zero.
A similar existence result as \cite{muha2013existence} was shown in \cite{muha2014existence} by the same authors where now, the elastodynamics of the cylinder wall is governed by the one-dimensional linear wave equation modelling the thin structural layer, and by the two-dimension equations of linear elasticity modelling the thick structural layer.
Further fluid-structure interaction results includes the work \cite{ignatova2017small} by Ignatova, Kukavica, Lasiecka \& Tuffaha where they construct a small data global solution for the motion of an elastic body inside an incompressible fluid.
Boulakia, Guerrero \& Takahashi
\cite{boulakia2019well} also considers the situation where the
elastic structure is immersed in the fluid and the whole system is confined into a general three-dimensional  bounded smooth domain.
Well-posedness and stability results for a fluid-structure interaction model with interior damping and delay in the structure is studied by Peralta \& Kunisch \cite{peralta2019analysis}.
As far as we know, the only result on the analysis of weak solutions to fluid-structure interaction, where the original Koiter model (to be described below in Section \ref{sec:koiterEnergy}) with a leading order nonlinear shell energy is considered, is the recent paper \cite{muha2019existence} by Muha \& Schwarzacher.
\\
Mathematical results concerning the interaction of a polymeric fluid with a flexible structure are, however, still missing in the literature. In this article, we aim to close this gap and initiate a corresponding analysis. In the following, we will describe the model in detail.

\subsection{Elastic shell}
\label{subsec:shell}
We are interested in the mathematical analysis of a polymer fluid evolving in a spatial domain with a moving shell. For this reason, we first describe this spatial geometry before we state the equations of motion. Following \cite{lee2013Introduction}, we let $\Omega\subset \mathbb{R}^3$ be an open, bounded, nonempty and connected reference domain  with an elastic shell $\omega \times (-\epsilon_0,\epsilon_0) \subset \mathbb{R}^3$  of thickness $2 \epsilon_0>0$ and a  \textit{middle surface}  $\omega$. 
Now  denote the first fundamental form (or metric tensor) and second fundamental form (or curvature tensor) of  $\partial \Omega$ induced by the ambient Euclidean space and by the surface measure $\dy$ of $\partial \Omega$ by $\mathbb{A}$ and $\mathbb{B}$ respectively. 
Assume that the movement of the shell $\partial \Omega$ is in the direction of the outer unit normal $\bm{\nu}$ (we shall give a precise construction of this normal vector later in Section \ref{sec:koiterEnergy}).
Now denote the normal bundle of $\partial \Omega $ by
\begin{align*}
N(\partial\Omega) = \big\{ (\by,\bz)\in \R^{3\times 3} \,   :\, \by \in \partial\Omega,\quad \bz \in N_\by(\partial\Omega) \big\},
\end{align*}
where $N_\by(\partial\Omega)$ is the $(3-2)$-dimensional normal space to $\partial\Omega$ at $\by$ consisting of all vectors orthogonal to the tangent space $T_\by(\partial\Omega)$ with respect to the Euclidean dot product. Simply put, $N(\partial \Omega)$ consists of all vectors normal to $\partial \Omega$ and by \cite[Theorem 6.23]{lee2013Introduction}, $N(\partial \Omega)$ is an embedded $3$-dimensional submanifold of $\R^{3\times 3}$. Consequently, $\partial \Omega$ has a \textit{tubular neighbourhood}, $S_L:= \big\{\bx \in \R^3 \, :\, \mathrm{dist}(\bx, \partial\Omega) <L\big\}$ for some $L>0$, see Fig. \ref{fig.tubular}. 
\begin{figure}[h!]
  \vspace{-5pt}
  \centering
    \includegraphics[width=0.65\textwidth]{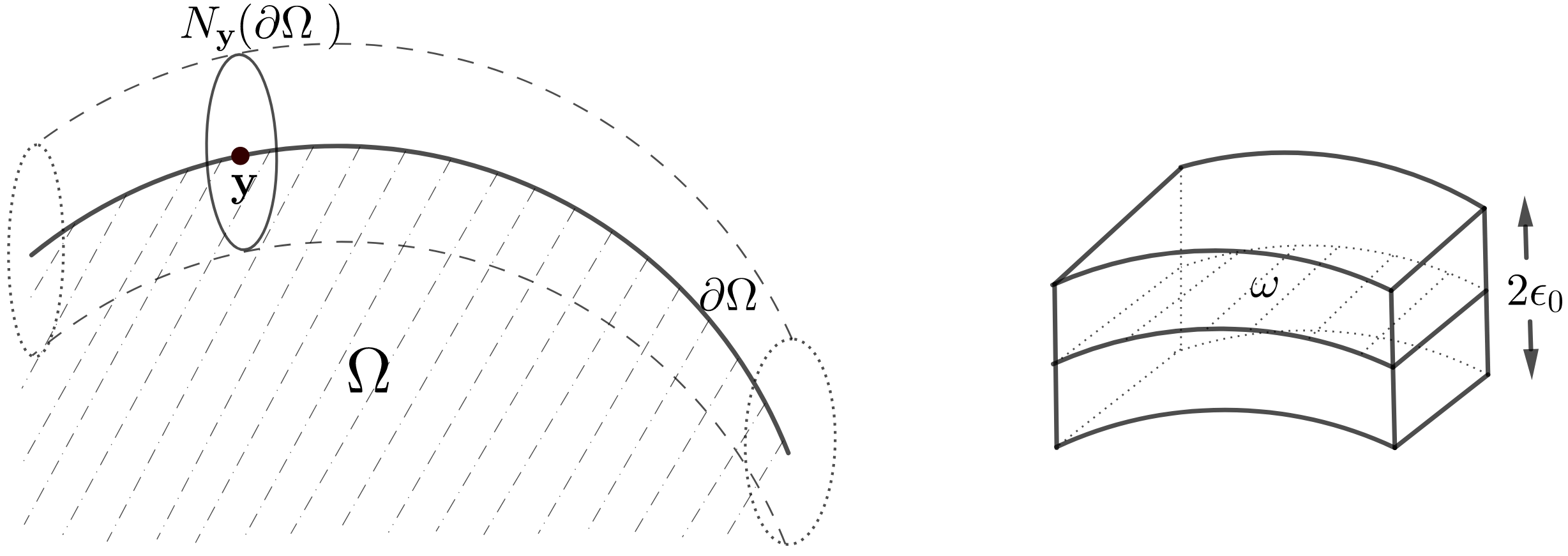}
      \vspace{-12pt}
  \caption{Left: A tubular neighbourhood of a shell $\partial \Omega$ is represented by the bended cylinder. Right: A macroscopic view of a tiny section of the shell $\partial \Omega$ with thickness $2\epsilon_0>0$. With an abuse of notation, we identify points $\by\in \partial\Omega$ on the $3$-d shell, with points on the middle surface $\by \in \omega$ which is a $2$-d submanifold.}
    \vspace{-5pt}
  \label{fig.tubular}
\end{figure}
Given the  outer unit normal $\bm{\nu}$ of $\partial \Omega$, one can construct a special affine mapping known as the  \textit{Hanzawa transform}, see \cite[Section 2]{lengeler2011globale}, which will be used below to relate the fixed domain $\Omega$ to a moving one. It is defined in terms of the mapping
\begin{align*}
\Lambda :\partial\Omega \times (-L,L)\rightarrow S_L, \quad  \Lambda(\by, s)=\by +s\bm{\nu}(\by).
\end{align*}
There is a maximal $L>0$ such that $\Lambda$
is a $C^3$-diffeomorphism with inverse
\begin{align}\label{eq:2212}
\Lambda^{-1} : S_L \rightarrow \partial\Omega \times (-L,L), \quad  \Lambda^{-1}(\bx)=(\by(\bx), s(\bx))
\end{align}
where $s(\bx) =(\bx -\by(\bx))\cdot \bm{\nu}(\by(\bx))$, cf. \cite[Theorem 6.24]{lee2013Introduction}.
A detailed construction of the Hanzawa transform can be found in \cite{lengeler2011globale} but for the sake of completeness, we summarize the construction below.\\
To begin with, for any specific instant of time $t\in \overline{I}:=\overline{(0,T)}$ where $T>0$, we consider a $C^3_\bx$-function $\eta(t,\cdot) : \partial \Omega \rightarrow (-L,L)$ and define the following open set
\begin{align*}
\Omega_{\eta(t)} := \Omega \setminus S_L \cup \big\{ \bx\in S_L : s(\bx) < \eta(t, \by(\bx)) \big\}.
\end{align*}
Let $\bm{\nu}_{\eta(t)}$  and $\dy_{\eta(t)}$ be the outer unit normal and the surface measure of $\partial \Omega_{\eta(t)}$ respectively. This function $\eta$ is further assumed to be continuous in time so that $\eta \in C(\overline{I} \times \partial\Omega)$ and we have that
\begin{align*}
\Vert \eta \Vert_{L^\infty(I \times \partial \Omega)}<L.
\end{align*} 
Moving on, we define a $C^3_\bx$-diffeomorphism $\bm{\Psi}_{\eta(t)} :\overline{\Omega} \rightarrow \overline{\Omega}_{\eta(t)}$ piecewise as
\begin{align}
&\bm{\Psi}_{\eta(t)} :\Omega \setminus S_L \rightarrow \overline{\Omega}_{\eta(t)}, \quad
&\bm{\Psi}_{\eta(t)}(\bx) &=\bx,\label{notdeform}
\\
&\bm{\Psi}_{\eta(t)} :\overline{\Omega} \cap S_L \rightarrow \overline{\Omega}_{\eta(t)}, \quad
&\bm{\Psi}_{\eta(t)}(\bx) &=\bx + \bm{\nu(\by(\bx))} \eta(t, \by(\bx)) \beta \big(s(\bx)/L \big) \nonumber
\\
&&&=\by(\bx) + \bm{\nu(\by(\bx))}\big[s(\bx) + \eta(t, \by(\bx)) \beta \big(s(\bx)/L \big) \big] \label{deform}
\end{align}
where $\beta \in C^\infty(\R)$ is a real-valued function which is zero in a neighbourhood of $-1$ and one in a neighbourhood of $0$. For the mapping $\bm{\Psi}_{\eta(t)}$ to have a continuously differentiable spatial inverse, we assume that $\vert \beta '(s) \vert < L/ \vert \eta(\by) \vert$ for all $s\in [-1,0]$ and all $\by \in \partial \Omega$. The boundary mapping is also a $C^3_\bx$-diffeomorphism defined as
\begin{align}
\label{deformBoundary}
\bm{\Phi}_{\eta(t)}:=\bm{\Psi}_{\eta(t)}\big\vert_{\partial \Omega} : \partial \Omega  \rightarrow \partial \Omega_{\eta(t)}, \quad
\bm{\Phi}_{\eta(t)}(\by) =\by + \bm{\nu(\by)} \eta(t,\by)
\end{align}
for every time $t\in \overline{I}$ with inverse $\bm{\Phi}_{\eta(t)}^{-1}(\bx) = \by(\bx)$.\\
The diffeomorphisms $\bm{\Phi}_{\eta}$ and $\bm{\Psi}_{\eta}$ constructed above, and thus the deformed shell $\overline{\Omega}_{\eta(t)}$, satisfies various continuity and embedding properties. A detailed analyses of these can be found in \cite{lengeler2011globale, lengeler2014weak}.\\
To summarize, if we denote the closure of the deformed spacetime cylinder $\cup_{t\in I} \{t\} \times\Omega_{\eta(t)} \subset\R^4$ by $\overline{I} \times\overline{\Omega}_{\eta(t)}$, then the mapping
\begin{align*}
\bm{\Psi}_{\eta} : \overline{I} \times\overline{\Omega} \rightarrow \overline{I} \times\overline{\Omega}_{\eta(t)}, \quad (t, \bx) \mapsto \big(t, \bm{\Psi}_{\eta(t)}(\bx) \big)
\end{align*}
preserves the portion of the original spacetime  cylinder $\overline{I} \times\overline{\Omega}$ that lies outside the tubular neigbourhood $S_L$ and deforms the residual portion of the original space-time  cylinder according to the mapping \eqref{deform}. The restriction of $\bm{\Psi}_{\eta}$ to the boundary is given by the mapping
\begin{align*}
\bm{\Phi}_{\eta} : \overline{I} \times \partial \Omega \rightarrow \overline{I} \times \partial\Omega_{\eta(t)}, \quad (t, \bx) \mapsto \big(t, \bm{\Phi}_{\eta(t)}(\bx) \big)
\end{align*}
according to the rule \eqref{deformBoundary}. 
\\
We now move on to give a precise description of the evolution of the shell and its associated energy below.
\subsection{Koiter shell energy and equation of motion}
\label{sec:koiterEnergy}
The polymer fluid we wish to model is assumed to interact with a Koiter shell $\omega \times (-\epsilon_0,\epsilon_0) \subset \mathbb{R}^3$.  Here,  $\omega \subset \R^2$ is the middle surface of the shell, recall Fig. \ref{fig.tubular}, and  $2 \epsilon_0>0$ is the thickness of $\partial \Omega$ and for simplicity, we take $\omega = \R^2\setminus \mathbb{Z}^2$ to be the flat torus. We emphasis that this periodic assumption on $\omega$ is not at all restrictive and everything we do subsequently can be replicated for a general $\omega$.
Following \cite{ ciarlet2001justification}, we suppose that $\partial\Omega$ can be parametrised by a smooth injective mapping $\bm{\varphi}:\omega\rightarrow \R^3$ such that for all points $\by=(y_1,y_2)\in \omega$, the pair of vectors  
$\partial_i \bm{\varphi}(\by)$, $i=1,2,$ are linearly independent where, $\partial_i :=\partial/\partial_{y_i}$. Simply put, $\bm{\varphi}$ is an injective map on the mid-section of the shell of the domain $\Omega$.
This vector pair $[\partial_1 \bm{\varphi}(\by), \partial_2 \bm{\varphi}(\by)]$ is the covariant basis of the tangent plane to the middle surface $\bm{\varphi}(\omega)$ of the reference configuration  at each point $\bm{\varphi}(\by)$ and
\begin{align*}
\bm{\nu}(\by)=\frac{\partial_1 \bm{\varphi}(\by) \times \partial_2 \bm{\varphi}(\by)}{\vert \partial_1 \bm{\varphi}(\by) \times \partial_2 \bm{\varphi}(\by) \vert}
\end{align*}
is a well-defined unit vector  normal to the surface $\bm{\varphi}(\omega)$ at $\bm{\varphi}(\by)$. The  area measure along the  surface $\bm{\varphi}(\omega)$ is $\mathrm{d}\by_{\bm{\nu}}:=\vert \partial_1 \bm{\varphi}(\by) \times \partial_2 \bm{\varphi}(\by) \vert \mathrm{d}\by$. We now assume that the shell (and in particular, its middle surface) only deforms along the normal direction with a displacement field $\eta \bm{\nu} : I \times \omega \rightarrow\R^3$ where $\eta : I \times \omega \rightarrow\R$ is considerably smooth. Then, we can parametrized the deformed boundary by the following coordinates
\begin{align*}
\bm{\varphi}_\eta(t, \by)=\bm{\varphi}(\by) + \eta(t, \by)\bm{\nu}(\by), \quad t\in I, \,\by \in \omega
\end{align*}
yielding the deformed middle surface 
$\bm{\varphi}_\eta(t, \omega)$. Now for
\begin{align*}
\partial_i \bm{\varphi}_\eta(t, \by)= \partial_i \bm{\varphi}( \by) +  \partial_i \eta(t, \by) \bm{\nu}(\by) +   \eta(t, \by) \partial_i\bm{\nu}(\by), \quad i =1,2,
\end{align*}
the covariant components of the first fundamental form of  the deformed middle surface  $ \bm{\varphi}_\eta(t, \omega)$ is given by
\begin{align*}
\partial_i \bm{\varphi}_\eta(t, \by) \cdot \partial_j \bm{\varphi}_\eta(t, \by) 
=\partial_i \bm{\varphi}(\by) \cdot \partial_j \bm{\varphi}( \by)+ G_{ij}(\eta)
\end{align*}
where
\begin{align*}
G_{ij}(\eta) &:= \partial_i \eta(t, \by)  \partial_j \eta(t, \by) + \eta(t, \by)\big[\partial_i \bm{\varphi}( \by) \cdot\partial_j  \bm{\nu}(\by) + \partial_j \bm{\varphi}( \by) \cdot\partial_i  \bm{\nu}(\by) \big] 
+ \eta^2(t, \by)\partial_i\bm{\nu}(\by) \cdot \partial_j  \bm{\nu}(\by)  
\end{align*}
are the covariant components of the `modified' change of metric tensor $\mathbb{G}(\eta)$. The normal (which is not a unit vector) to the deformed middle surface  $ \bm{\varphi}_\eta(t, \omega)$ at the point $ \bm{\varphi}_\eta(t, \by)$ is then given by
\begin{align*}
\bm{\nu}_\eta(t,\by)
&=\partial_1 \bm{\varphi}_\eta(t,\by) \times \partial_2 \bm{\varphi}_\eta(t,\by)
=
 \bm{\nu}(\by) \big\vert \partial_1 \bm{\varphi}(\by) \times \partial_2 \bm{\varphi}(\by) \big\vert
 +
  \partial_2 \eta(t,\by)\big(\partial_1 \bm{\varphi}(\by) \times \bm{\nu}(\by) 
 \\&+
 \eta(t,\by) \partial_1\bm{\nu}(\by) \times \bm{\nu}(\by) \big)
  +
  \partial_1 \eta(t,\by)\big( \bm{\nu}(\by) \times \partial_2 \bm{\varphi}(\by)  +
 \eta(t,\by) \bm{\nu}(\by) \times \partial_2\bm{\nu}(\by) \big)
 \\&
   +
\eta(t,\by)\big(\partial_1 \bm{\varphi}(\by) \times \partial_2  \bm{\nu}(\by)  +
\partial_1 \bm{\nu}(\by) \times \partial_2\bm{\varphi}(\by) \big)
+
\eta^2(t,\by)\big( \partial_1 \bm{\nu}(\by) \times \partial_2\bm{\nu}(\by) \big)
\end{align*}
and
\begin{align*}
R_{ij}^\sharp(\eta) :=\frac{\partial_{ij} \bm{\varphi}_\eta(t,\by)\cdot \bm{\nu}_\eta(t,\by)}{\vert \partial_1 \bm{\varphi}(\by) \times \partial_2 \bm{\varphi}(\by) \vert}
-
\partial_{ij} \bm{\varphi}(\by)\cdot \bm{\nu}(\by) , \quad i,j=1,2,
\end{align*}
are the covariant components of the change of curvature tensor $\mathbb{R}^\sharp(\eta)$. The elastic energy $K(\eta):=K(\eta, \eta)$ of the deformation is then given by 
\begin{equation}
\begin{aligned}
\label{koiterEnergy}
K(\eta)&= \frac{1}{2}\epsilon_0 \int_\omega \mathbb{C}:\mathbb{G}(\eta) \otimes \mathbb{G}(\eta ) \dy
+
\frac{1}{6}\epsilon_0^3 \int_\omega \mathbb{C}:\mathbb{R}^\sharp(\eta ) \otimes\mathbb{R}^\sharp(\eta ) \dy
\\&
:=\sum_{i,j,k,l=1}^2 \frac{1}{2}\epsilon_0 \int_\omega C^{ijkl}G_{kl}(\eta )G_{ij}(\eta ) \dy
+
\frac{1}{6}\epsilon_0^3 \int_\omega C^{ijkl}{R}^\sharp_{kl}(\eta ){R}^\sharp_{ij}(\eta ) \dy
\end{aligned}
\end{equation} 
where $\mathbb{C}=( C^{ijkl})_{i,j,k,l =1}^2$ is a fourth-order tensor whose entries are the contravariant components of the shell elasticity, see \cite[Page 162]{ciarlet2005Introduction}. We remark that for simplicity, we have normalized the measure $\dy$ in \eqref{koiterEnergy} which should have actually been the weighted measure $\dy_{\bm{\nu}}:=\vert \partial_1 \bm{\varphi}(\by) \times \partial_2 \bm{\varphi}(\by) \vert \mathrm{d}\by$ with the non-zero weight $\vert \partial_1 \bm{\varphi}(\by) \times \partial_2 \bm{\varphi}(\by) \vert$, see \cite{roquefort2001quelques}. 
Next, given the following geometric quantity
\begin{equation}
\begin{aligned}
\label{geomQuan}
\gamma(\eta) :=
1&+
\frac{\eta}{\vert \partial_1 \bm{\varphi}(\by) \times \partial_2 \bm{\varphi}(\by) \vert}
\Big[
\bm{\nu}(\by) \cdot\big(  \partial_1 \bm{\varphi}(\by)\times \partial_2 \bm{\nu}(\by)+\partial_1\bm{\nu}(\by) \times \partial_2 \bm{\varphi}(\by) \big) 
\Big]
\\
&+
\frac{\eta^2}{\vert \partial_1 \bm{\varphi}(\by) \times \partial_2 \bm{\varphi}(\by) \vert}
\bm{\nu}(\by)\cdot \big(\partial_1\bm{\nu} (\by)\times \partial_2\bm{\nu}(\by) \big),
\end{aligned}
\end{equation}
one deduces the $W^{2,2}(\omega)$-coercivity of the Koiter energy \eqref{koiterEnergy} as long as $\gamma(\eta)\neq0$. This is the case if $\|\eta\|_{L^\infty(\omega)}\leq \tilde L$
for some $\tilde L>0$ depending on the geometry of $\Omega$ 
 Further details can be found in \cite[Lemma 4.3
and Remark 4.4]{muha2019existence}. Without loss of generality, we assume that $L\leq \tilde L$, where $L$ is the threshold for self-intersection introduced in Section \ref{subsec:shell}. Finally, we remark that
the Koiter energy is continuous on $W^{2,p}(\omega)$ for all $p>2$ due to the Sobolev embedding
$W^{2,p}(\omega)\hookrightarrow W^{1,\infty}(\omega)$.
\\
If the mass density of $\omega$ is $\epsilon_0\rho_S$ where $\rho_S>0$ is a constant, and we simply denote the $L^2$-gradient of $K$ by $K'$ (which is to the interpreted in the sense that $K'(\eta) \zeta=2K(\eta, \zeta)$ for any $\zeta$ in the dual space of $\eta$), then the evolution of the shell is modelled by (later on, we assume for simplicity that $\epsilon_0\rho_S=1$)
\begin{align}
\label{shellEq}
\epsilon_0\rho_S\partial_t^2
\eta + K'(\eta) =g + \mathbf{F}\cdot \bm{\nu}
\end{align}
in $I \times \omega$ subject to the following initial and boundary conditions
\begin{align}
&\eta(0,\cdot) =\eta_0, \quad \partial_t \eta(0,\cdot)  =\eta_1
&\quad \text{in } \omega, \label{etaTzero}
\\
& \bu(t, \bm{\varphi}(\by)+\eta(t,\by) \bm{\nu}(\by)) =\partial_t \eta(t,\by) \bm{\nu}(\by)
&\quad \text{on } I \times \omega\label{etaNablaEta}
\end{align}
where $\eta_0, \eta_1 :\omega \rightarrow \mathbb{R}$ are given functions and where in \eqref{shellEq}, the function $g: I \times \omega \rightarrow \mathbb{R}$ is a given force density  and 
\begin{align}\label{bigF}
\mathbf{F}(t,\by):= \Big(-2\mu\, \mathbb{D}_\by \bu(t,\by)
-
[\mathbb{T}(\psi)](t,\by)
+ p(t, \by) \mathbb{I}
 \Big)\bm{\nu}_{\eta(t)} \circ \bm{\bm{\varphi}}_{\eta(t)} \vert \mathrm{det} D  \bm{\bm{\varphi}}_{\eta(t)} \vert .
\end{align}
Also, the tensor
\begin{align}
 \mathbb{D}_\by\bu =\frac{1}{2}(\naby \bu + \naby^T\bu)
\end{align} 
is the symmetric gradient of the fluid's velocity field. We will however introduce the tensor $\mathbb{T}(\psi)$ in the next section below.
\subsection{Polymer fluid}
\label{subsec:poly}
A common mathematical model to describe the behaviour of complex fluids are the FENE-type models. For these models, the polymer molecules are idealized as a chain of beads and springs with prescribed finitely extensible nonlinear elastic (FENE) type spring potentials. For a finite but arbitrary number $K>1$, $K+1$ beads are connected by $K$ springs to form a polymer chain. This polymer is represented by a vector of $K$-finite vectors
$\bq=(\bq^T_1, \ldots, \bq^T_K)^T\in B$, where $B =\bigotimes_{i=1}^K B_i\subset\R^{3K}$ is the Cartesian product of convex open sets $B_i \subset \mathbb{R}^3$ such that $\bq_i \in B_i$ if and only if  $-\bq_i \in B_i$.
\begin{figure}[h!]
  \vspace{-8pt}
  \centering
    \includegraphics[width=0.7\textwidth]{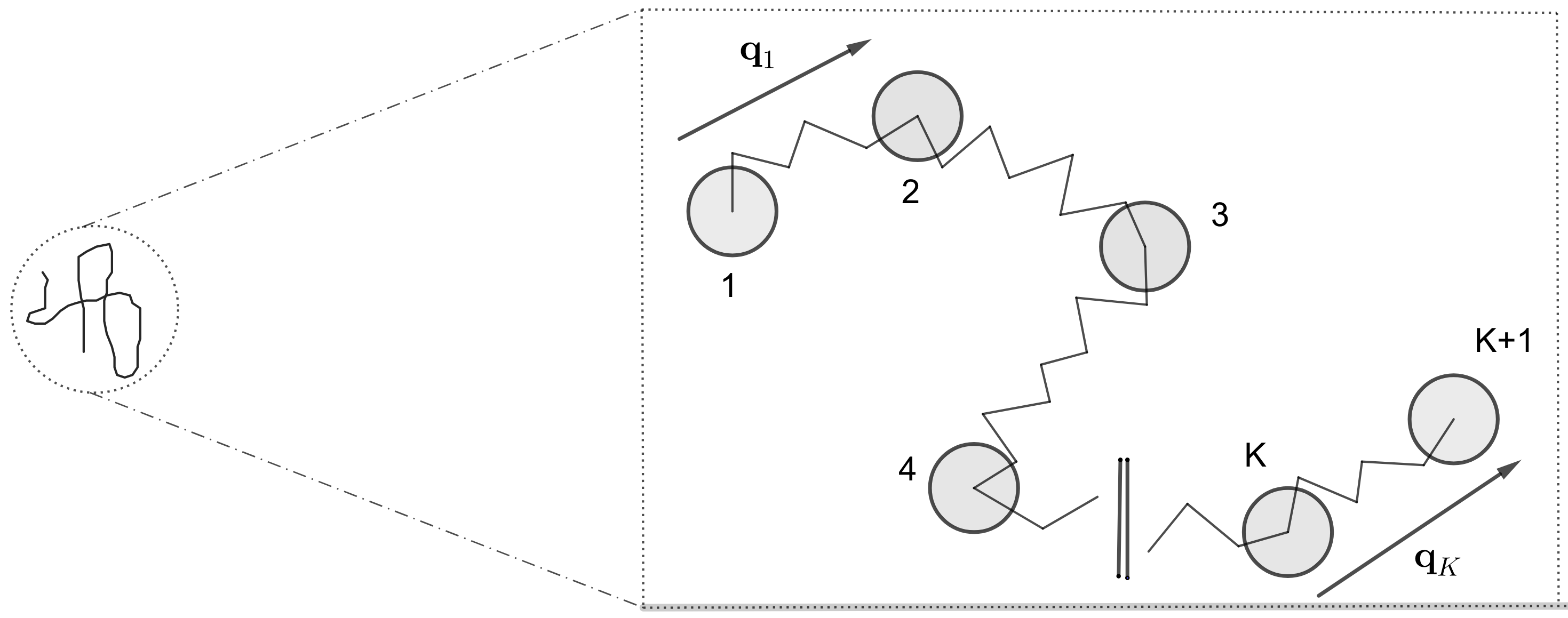}
      \vspace{-13pt}
  \caption{We model a polymer  (circled on the left) as a bead-spring chain consisting of $K+1$ beads connected by $K$ springs  (boxed on the right).}
    \vspace{-12pt}
  \label{fig.polyChain}
\end{figure}
On the mesoscopic level, we describe the evolutionary changes in the distribution of the bead-spring chain configuration  by the Fokker--Planck equation for the polymer density function $\psi =\psi(t, \mathbf{x},\mathbf{q})$ (depending on time $t\geq0$, spatial position $\bx\in\R^3$ and the prolongation vector $\bq\in B$ of the polymer chain).
On the macroscopic level, we consider a viscous fluid described by the incompressible
Navier--Stokes equations for the fluid velocity $\bu=\bu(t,\bx)$ and pressure $p=p(t,\bx)$. The beads, which model the monomers that are joined by springs to form a polymer chain, unsettle the flow field around the chain once immersed in the fluid. 
These mesoscopic effects of the polymer molecules on the fluid motion are described by an elastic stress tensor $\mathbb{T}$. It is meant to describe the random movements of polymer chains/springs and can be modelled by prescribing spring potentials $ U_i$, $i=1, \ldots, K$ for each of the $K$ springs. 
Here, for each $i=1, \ldots, K$, the potential $U_i$  is continuous on an interval $I_i\subset[0,\infty)$ containing the point $0$ where $I_i$ is the image of $B_i$ under the mapping $\bq_i\in B_i \mapsto \frac{1}{2}\vert \mathbf{q}_i \vert^2$. To be precise, we consider $U_i \in C^{0,1}_{\mathrm{loc}}(I_i;[0,\infty))$ for each $i=1, \ldots, K$. Typically, these potentials will be such that $U_i(0)=0$ and also, be monotonically increasing and unbounded on the interval $I_i$ , for each $i=1, \ldots, K$. 
The elastic spring force $\mathbf{F}_i:B_i \subset \mathbb{R}^3 \rightarrow \mathbb{R}^3$ of the $i$th spring and the associated \textit{Maxwellian} $M_i$ are defined by
\begin{align}
\label{elasticSpringForce}
\mathbf{F}_i (\mathbf{q}_i) = \nabla_{\bq_i} U_i\Big(\frac{1}{2}\vert \mathbf{q}_i \vert^2 \Big) = U'_i\Big(\frac{1}{2}\vert \mathbf{q}_i \vert^2 \Big) \mathbf{q}_i, \qquad i=1, \ldots, K
\end{align}
and
\begin{align}
\label{maxwellianPartial}
M_i(\mathbf{q}_i) = \frac{e^{-U_i \left(\frac{1}{2}\vert \mathbf{q}_i \vert^2 \right) }}{\int_{B_i}e^{-U_i \left(\frac{1}{2}\vert \mathbf{q}_i \vert^2 \right) }\,\mathrm{d}\bq_i} \qquad i=1, \ldots, K
\end{align}
respectively such that $\int_{B_i} M_i(\mathbf{q}_i)\,\mathrm{d}\bq_i=1$ for each $i=1, \ldots, K$. The (total) Maxwellian is then given by
\begin{align}
\label{maxwellian}
M(\bq) := \prod_{i=1}^KM_i(\mathbf{q}_i) , \qquad
\bq=(\bq^T_1, \ldots, \bq^T_K)^T\in B =\bigotimes_{i=1}^K B_i
\end{align}
and also satisfies $\int_{B} M(\mathbf{q})\,\mathrm{d}\bq=1$.
We observe from \eqref{elasticSpringForce}--\eqref{maxwellian} that
\begin{align}
\label{fandM}
M(\bq) \nabla_{\bq_i}\frac{1}{M(\bq)} =  \frac{-1}{M(\bq)}\nabla_{\bq_i} M(\bq) =\mathbf{F}_i (\mathbf{q}_i) 
\end{align}
for each $i=1, \ldots, K$. 
\\
Before we continue, we now give some examples of the precise force-laws \eqref{elasticSpringForce} used in the literature, see \cite[Table 11.5-1]{bird1987dynamics}. For simplicity of the presentation, we only describe the `dumbbell' models corresponding to $K=1$.
\begin{example}[\textbf{Hookean dumbbell model}]
This model has a prescribed linear spring force and (equivalently) linear spring potential given by $\mathbf{F}(\bq)=\bq$ where $\bq\in B=\R^d$, $d=2,3$ and $U(s)=s$, $s\in \R_{\geq0}$ respectively. Therefore, the beads coalesce at the origin $\bm{0}\in B$ when the spring force is zero. The main drawback of this model is that it admits arbitrarily large spring extension making it physically unrealistic.
\end{example}
\begin{example}[\textbf{`Linear-locked' Tanner's dumbbell model}]
This model is a realistic variant of the Hookean dumbbell model with the same linear force law and linear spring potential except that now,  $\bq\in B=B(\bm{0},\sqrt{b}) \subset \R^d$, $d=2,3$ where  $B=B(\bm{0},\sqrt{b})$ is a bounded open set centred at $\bm{0}\in \R^d$ of radius  $\sqrt{b}$, with $b>0$. The prescribed radius $\sqrt{b}>0$ thus denotes the extent to which the springs can be stretched.
\end{example}
%
\begin{example}[\textbf{FENE [finitely extensible nonlinear elastic] dumbbell model}]
This model also corresponds to the case $K=1$ but with a nonlinear spring force given by $\mathbf{F}(\bq)=(1- \vert \bq \vert^2/b)^{-1}\bq$ where $\bq\in B=B(\bm{0},\sqrt{b}) \subset \R^d$, $d=2,3$. 
Equivalently, the nonlinear spring potential is given $U(s)=-\frac{b}{2}\log(1-\frac{2s}{b})$, $s\in [0,\frac{b}{2})$. 
\end{example}
With respect to regularity, we assume that the Maxwellian satisfies the following conditions:
\begin{align}
&M\in C(\overline{B}) \cap C^{0,1}_{\mathrm{loc}}(B) \cap W^{1,1}_0(B), \quad M\geq 0, \quad M^{-1} \in C_{\mathrm{loc}}(B).\label{boundMiqi}
\end{align}
Now, for a given probability density function  $\psi =\psi(t, \mathbf{x},\mathbf{q})$  of a polymer, we let
\begin{align}
\label{eta}
\Xi(t,x) = \int_B \psi(t, \bx, \bq) \dq, \qquad (t, \bx) \in I \times \Omega_{\eta(t)}
\end{align}
be the \textit{polymer number density}. The elastic stress tensor $\mathbb{T}$ is then given by
\begin{align}
\label{extraStreeTensor}
\mathbb{T}(\psi) 
=
k\sum_{i=1}^K \mathbb{T}_i(\psi) - k(K+1)\Xi \,\mathbb{I}
-\eth \,\Xi^2 \mathbb{I}
\end{align}
where $\mathbb{I}$ is the identity matrix, $k >0$  and $\eth\geq 0$ are constants and for each $i=1, \ldots, K$,
\begin{align}
\label{extraStreeTensor1}
\mathbb{T}_i(\psi) 
=
 \int_B \psi(t, \mathbf{x},\mathbf{q})U'_i\Big(\frac{1}{2}\vert \mathbf{q}_i \vert^2 \Big) \mathbf{q}_i \otimes\mathbf{q}_i \dq
= 
 \int_B \psi(t, \mathbf{x},\mathbf{q})\, \mathbf{F}_i(\mathbf{q}_i) \otimes\mathbf{q}_i \dq,
\end{align}
elucidates how  the polymers - described by the force law for the $i$th spring - are transmitted through the fluid. 
\\
In addition to the  elastic stress tensor $\mathbb T$, we consider an external volume force  $\mathbf{f}:(t, \mathbf{x} )\in I \times \Oeta  \mapsto  \mathbf{f}(t, \mathbf{x}) \in \mathbb{R}^3$ in the fluid motion. This force may account for the influence of gravity and/or electric force as well as artificial forces produced, for example, by an ultracentrifuge. \\
The coupled system is now described by the incompressible Navier--Stokes--Fokker--Planck system in the moving domain $I\times\Omega_{\eta(t)}$ for a given function $\eta:I\times\partial\Omega\rightarrow (-L,L)$.
We wish to find the fluid's velocity field $\mathbf{u}:(t, \mathbf{x})\in I \times \Oeta \mapsto  \mathbf{u}(t, \mathbf{x}) \in \mathbb{R}^3$, the pressure $p:(t, \mathbf{x})\in I \times \Oeta \mapsto  p(t, \mathbf{x}) \in \mathbb{R}$ and the probability density function $\psi:(t, \mathbf{x}, \mathbf{q})\in I \times \Oeta \times B \mapsto \psi(t, \mathbf{x}, \mathbf{q}) \in [0,\infty)$
such that the equations
\begin{align}
\divx \bu=0,
\label{contEq}
\\
\partial_t \bu  + (\mathbf{u}\cdot \nabx)\mathbf{u} + \nabx   p
= 
\mu\Delta_\bx   \vu +\divx   \mathbb{T}(\psi) + \mathbf{f},
\label{momEq}
\\
\partial_t \psi + (\mathbf{u}\cdot \nabx) \psi
=
\varepsilon \Delx \psi
- \sum_{i=1}^K
 \divqi  \big( (\nabx   \mathbf{u}) \bq_i\psi \big) 
 +
\frac{1}{4\lambda} \sum_{i=1}^K \sum_{j=1}^K A_{ij}\,  \divqi  \bigg( M \nabqj  \frac{\psi}{M} 
\bigg)
\label{fokkerPlank}
\end{align}
are satisfied weakly in $I \times \Oeta \times B$ subject to the following initial ($t=0$)  condition and boundary ($\by \in \partial \Oeta $ or $\bq_i\in\partial \overline{B}_i$) conditions 
\begin{align}
&\bigg[\frac{1}{4\lambda}  \sum_{j=1}^K A_{ij}\, M \nabqj  \frac{\psi}{M} -(\nabx   \mathbf{u}) \bq_i\psi
 \bigg] \cdot \bn_i =0
&\quad \text{on }I \times \Oeta \times \partial \overline{B}_i, \quad\text{for } i=1, \ldots, K,
\label{fokkerPlankBoundary}
\\
&\varepsilon \naby \psi \cdot \bm{\nu}_{\eta(t)} =0
&\quad \text{on }  I \times \partial \Oeta \times B, \label{fokkerPlankBoundaryNeumann}
\\
&\mathbf{u}(0, \cdot) = \mathbf{u}_0
&\quad \text{in }  \Omega_{\eta_0},
\label{initialDensityVelo}
\\
&\psi(0, \cdot, \cdot) =\psi_0 \geq 0
& \quad \text{in }\Omega_{\eta_0} \times B.
\label{fokkerPlankIintial}
\end{align}
The parameter $\mu>0$ is the viscosity coefficient, $\varepsilon>0$  is the center-of-mass diffusion coefficient, $\lambda > 0$ is the \textit{Deborah number} $\mathrm{De}$, the $A_{ij}$'s are the components of the symmetric positive definite \textit{Rouse matrix} $(A_{ij})_{i,j=1}^K$ whose smallest eigenvalue is $A_0>0$ and $\bn_i$ is a unit outward normal vector to $\partial B_i$.
\\
If we now return to  \eqref{eta} for a moment, we observe that by formally integrating \eqref{fokkerPlank} over the open set $B$ and using the boundary condition \eqref{fokkerPlankBoundary}, then $\Xi$ satisfies the following viscous transport equation
\begin{align}
\label{transportXi}
\partial_t \Xi+ (\bu \cdot \nabx)\Xi= \varepsilon\Delx \Xi
\end{align}
weakly in $I \times \Oeta $ subject to the following initial and boundary conditions
\begin{align}
&\varepsilon \naby \Xi \cdot \bm{\nu}_{\eta(t)} =0
&\quad \text{on }  I \times \partial \Oeta, \label{XiBoundaryNeumann}
\\
&\Xi_0=\int_B\psi_0( \cdot, \bq)\dq \geq 0
& \quad \text{in }\Omega_{\eta_0}.
\label{XiIintial}
\end{align}
The structure of the tensor $\mathbb{T}$, given by \eqref{extraStreeTensor}, means that the analysis of \eqref{transportXi}--\eqref{XiIintial} is essential to the analysis of the extra stress tensor $\mathbb{T}$.
\\
If we now define $\widehat{\psi}:=\psi/M$, then 
the full extra stress tensor \eqref{extraStreeTensor} may be rewritten as
\begin{align}
\label{extraStreeTensorx}
\mathbb{T}(M\widehat{\psi}) 
=
\mathbb{T}(\psi) 
=
k\sum_{i=1}^K  \int_BM(\bq) \nabqi \widehat{\psi}(t, \mathbf{x},\mathbf{q})\otimes\mathbf{q}_i \dq
 -
  k\,\Xi \,\mathbb{I}
-\eth \,\Xi^2 \mathbb{I}.
\end{align}
The Fokker--Planck equation
\eqref{fokkerPlank} then become
\begin{align}
\partial_t (M\widehat{\psi}) + (\mathbf{u}\cdot \nabx) M\widehat{\psi}
=
\varepsilon \Delx (M\widehat{\psi})
- \sum_{i=1}^K
 \divqi  \big( M(\nabx   \mathbf{u}) \bq_i\widehat{\psi} \big) 
 +
\frac{1}{4\lambda} \sum_{i=1}^K \sum_{j=1}^K A_{ij}\,  \divqi  \big( M \nabqj  \widehat{\psi} 
\big)
\label{fokkerPlankx}
\end{align}
subject to the initial and boundary conditions
\begin{align}
&\bigg[\frac{1}{4\lambda}  \sum_{j=1}^K A_{ij}\, M \nabqj  \widehat{\psi} - M(\nabx   \mathbf{u}) \bq_i\widehat{\psi} 
 \bigg] \cdot \bn_i =0
&\quad \text{on }I \times \Oeta \times \partial \overline{B}_i, \quad\text{for } i=1, \ldots, K,
\label{fokkerPlankBoundaryx}
\\
&\varepsilon \naby\widehat{\psi} \cdot \bm{\nu}_{\eta(t)} =0
&\quad \text{on }  I \times \partial \Oeta \times B, \label{fokkerPlankBoundaryNeumannx}
\\
&\widehat{\psi}(0, \cdot, \cdot) =\widehat{\psi}_0 \geq 0
& \quad \text{in }\Omega_{\eta_0} \times B.
\label{fokkerPlankIintialx}
\end{align}

\subsection{Plan of the paper}
In the following, we give the outline of the rest of this paper. As stated in the abstract above, we aim to show the existence of a weak solution to the coupled system \eqref{contEq}--\eqref{fokkerPlankIintial} and \eqref{shellEq}--\eqref{etaNablaEta} where the solution exists globally in time until the shell approaches a self-intersection or the $W^{2,2}$-coercivity of the Koiter energy given in \eqref{koiterEnergy} degenerates.
Therefore, after collecting some preliminary notations and concepts
in Section \ref{subsec:Not} and Section \ref{subsec:Function}, 
we make precise, the exact notion of a solution to our system in Section \ref{subsec:solut}. We also state our main theorem, Theorem \ref{thm:main} in Section \ref{subsecMainResult}. We then move to Section \ref{sec:energEsr},  where we show how to formally derive a priori estimates and also, introduce the energy and relative entropy  of our system.
\\
Our principal strategy to solve the coupled system consists in regularising the shell (and the convective terms) and to decouple
the fluid-structure problem from the Fokker--Planck equation.
For this reason, in Section \ref{sec:SolveFokkerPlanck}, we solve the Fokker--Planck equation in a variable (but given) domain for a given (and smooth) velocity field. On the other hand, for a given probability density function, we  solve the fluid-structure problem by following the fixed-point arguments in \cite{lengeler2014weak}. This is done in Section \ref{sec:reg}. Finally, we pass to the limit in the regularisation layer in Section \ref{sec:main}. For this, we are required to rigorously prove the entropy estimates from Section \ref{sec:energEsr} (see Remark \eqref{rem:EnerEntr} for what we mean by an entropy estimate)
and to apply compactness methods to pass to the limit in the nonlinear terms of our system.

\section{Preliminaries and main result}
\label{sec:prelim}
In this section, we fix the notation, collect some preliminary material on function spaces and present the main result.
\subsection{Notations}
\label{subsec:Not}
The following quantities: $t\in \overline{I}$, representing  the time variable, $\mathbf{x}\in \Omega$, representing the spatial variable, and $\mathbf{q}\in B$, representing the elongation vector of a polymer molecule, will denote the independent variables we shall use throughout this work. The domain $B =\bigotimes_{i=1}^K B_i\subset\R^{3K}$ is the Cartesian product of $K$ convex and bounded open sets $B_i \subset \mathbb{R}^3$ for which $\bq_i \in B_i$ if and only if $-\bq_i \in B_i$. In particular, the origin $\bm{0} \in \R^3$ is contained in each set $B_i$ whose boundary we denote by $\partial B_i$. We further have
\begin{align*}
\partial B = \bigcup_{i=1}^K \partial \overline{B_i}, \quad \text{where}\quad   \partial \overline{B_i} := B_1\times \ldots \times B_{i-1} \times \partial B_i \times B_{i+1} \times \ldots \times B_K
\end{align*}
with $\bn_i$ being a unit outward normal vector to $\partial B_i$, $i=1,\ldots,K$.
For functions $F$ and $G$ and a variable $p$, we write $F \lesssim G$ and $F \lesssim_p G$ if there exists  a generic constant $c>0$ and another such constant $c(p)>0$ which now depends on $p$ such that $F \leq c\,G$ and $F \leq c(p) G$ holds respectively. The symbol $\vert \cdot \vert$ may be used in four different context. For a scaler function $f\in \mathbb{R}$, $\vert f\vert$ denotes the absolute value of $f$. For a vector $\bff\in \mathbb{R}^n$ where $n>1$ is an integer, $\vert \bff \vert$ denotes the Euclidean norm of $\bff$. For a square matrix $\mathbb{F}\in \mathbb{R}^{n\times n}$ where $n>1$ is an integer, $\vert \mathbb{F} \vert$ shall denote the Frobenius norm $\sqrt{\mathrm{trace}(\mathbb{F}^T\mathbb{F})}$. Finally, if $S\subset \R^n$ is  a (sub)set, then $\vert S \vert$ is the $n$-dimensional Lebesgue measure of $S$.
\\
Let $\mathcal{O}\subset \R^d$ be a measureable set. By $L^p(\mathcal{O})$ [respectively $L^p(\mathcal{O}; \mathbb{R}^3)$], $W^{s,p}(\mathcal{O})$ [respectively $W^{s,p}(\mathcal{O}; \mathbb{R}^3)$] and $D^{s,p}(\mathcal{O})$ [respectively $D^{s,p}(\mathcal{O}; \mathbb{R}^3)$] for $1\leq p\leq\infty$ and $s\in\mathbb N$, we denote respectively, the standard Lebesgue spaces, Sobolev spaces, homogeneous Sobolev spaces for scalar-valued [respectively $\mathbb{R}^3$-vector-valued] functions defined on $\mathcal O$. The dual space of $W^{s,p}(\mathcal{O})$ will be denoted by $W^{-s,p}(\mathcal{O})$.  By $L^p_{\divx}(\mathcal{O}; \mathbb{R}^3)$, we mean vector-valued measurable functions $\bv :\mathcal{O} \rightarrow \mathbb{R}^3$ such that $\divx \bv=0$ in the distributional sense and $\Vert \bv \Vert_{L^p(\mathcal{O})} <\infty$. The Sobolev space $W^{s,p}_0(\mathcal{O})$ is endowed with zero boundary condition (i.e. it is the closure of the smooth and compactly supported functions in $W^{s,p}(\mathcal{O})$). Also, for $s\in(0,1)$ and $1\leq p<\infty$,
we define the fractional Sobolev space $W^{s,p}(\mathcal O)$ as the set of all measurable functions $f:\Omega\rightarrow\R$ such that
\begin{align}
\label{frac}
\|f\|_{W^{s,p}(\mathcal O)}^p=\int_{\mathcal O}|f(\bx)|^p\dx+\int_{\mathcal O}\int_{\mathcal O}\frac{|f(\bx)-f(\mathbf{z})|^p}{|\bx-\mathbf{z}|^{d+ps}}\dx\,\dd\mathbf{z}<\infty.
\end{align}
In general, for a separable Banach space $(X,\|\cdot\|_X)$, we denote by $L^p(0,T;X)$, the space of Bochner-measurable functions $u:(0,T)\rightarrow X$ such that $\|u\|_X\in L^p(0,T)$. Similarly, we consider
the space $L^p(\mathcal{O};X)$ for a  measurable set $\mathcal{O}\subset \R^d$.
Also,  $C(\mathcal{O};X)$ is the set of continuous functions $u:\mathcal{O}\rightarrow X$. Finally, for any nonnegative $N\in C(\mathcal{O})$, where $\mathcal{O}\subset \R^d$ is a measurable set, and for a constant $p\geq1$, we denote by
\begin{align*}
&L^ p _N(\mathcal{O})  =  \big\{ f\in L^p_{\mathrm{loc}}(\mathcal{O}) \, : \, \Vert f \Vert_{L^p_N(\mathcal{O})}^p < \infty \big\}, 
\quad
&W^{1,p}_N(\mathcal{O}) =  \big\{ f\in W^{1,p}_{\mathrm{loc}}(\mathcal{O}) \, : \, \Vert f \Vert_{W^{1,p}_N(\mathcal{O})}^p < \infty \big\},
\end{align*}
the weighted $L^p $ and $W^{1,p}$ spaces over $\mathcal{O}$ with norms
\begin{align*}
\Vert f \Vert_{L^ p _N(\mathcal{O})}^p :=  \int_\mathcal{O} N(\mathbf{z} ) \vert  f(\mathbf{z} ) \vert^ p  \,\mathrm{d} \mathbf{z}
\quad \quad \text{and} \quad \quad
\Vert f \Vert_{W^{1,p}_N(\mathcal{O})}^p :=  \int_\mathcal{O} N(\mathbf{z})\big( \vert    f(\mathbf{z}) \vert^p+ \vert \nabla_{\mathbf{z}}f(\mathbf{z}) \vert^p\big)  \,\mathrm{d} \mathbf{z},
\end{align*}
respectively.

\subsection{Function spaces on variable domains}
\label{subsec:Function}
 The spatial domain $\Omega$ is a nonempty bounded subset of $\mathbb{R}^3$ with smooth boundary and an outer unit normal $\bm{\nu}$, $\partial \Omega$ is the shell of $\Omega \subset \R^3$.
We use $\by \in \partial \Omega$ to emphasis spatial boundary points with a corresponding surface measure $\dy$. To further clarify, we identify the shell as a \textit{usual boundary} by tracing out its $2$-dimensional mid-section $\omega\in \R^2$, see Fig \ref{fig.tubular}. In the following, $\overline{I}$ is the closure of $I=(0,T)$, a time interval where $T>0$ is a constant.  For $\eta \in C(\overline{I} \times \partial\Omega)$ satisfying $\Vert \eta \Vert_{L^\infty(I \times \partial\Omega)}<L$, we shall abuse notation and denote the deformed spacetime cylinder $\cup_{t\in I} \{t\} \times\Omega_{\eta(t)} \subset\R^4$ by either $\Omega_{\eta(t)}^I$ or $I\times \Omega_{\eta(t)}$. 
We are now in the position to define function spaces on a variable domain.
\begin{definition}{(Function spaces)}
We define for $1\leq p,r\leq\infty$, 
\begin{align*}
L^p(I;L^r(\Omega_{\eta(t)}))&:=\big\{v\in L^1(I\times\Omega_{\eta(t)}):\,\,v(t,\cdot)\in L^r(\Omega_{\eta(t)})\,\,\text{for a.e. }t,\,\,\|v(t,\cdot)\|_{L^r(\Omega_{\eta(t)})}\in L^p(I)\big\},\\
L^p(I;W^{1,r}(\Omega_{\eta(t)}))&:=\big\{v\in L^p(I;L^r(\Omega_{\eta(t)})):\,\,\nabx v\in L^p(I;L^r(\Omega_{\eta(t)}))\big\}.
\end{align*}
\end{definition}
We now give a concept of convergence in variable domains which is similar to \cite[Sec. 2.3]{breitSchw2018compressible}.
\begin{definition}[Convergence]
\label{def:varConv}
Let $(\eta_i) \subset C(\overline{I} \times \partial\Omega;[-\theta L, \theta L])$, $\theta \in (0,1)$ be a sequence such that $\eta_i \rightarrow \eta$ uniformly in $\overline{I} \times \omega$. Let $p_1, p_2,p_3\in [1,\infty]$ and let $M$ satisfy \eqref{boundMiqi} or be identically equal to one. Then;
\begin{itemize}
\item[(a)] we say that a sequence $g_i \in L^{p_1}(I;L^{p_2}(\Omega_{\eta_i(t)};L^{p_3}_M(B)))$ converges strongly to $g$ in\\ $L^{p_1}(I;L^{p_2}(\Omega_{\eta_i(t)};L^{p_3}_M(B)))$ with respect to $\eta_i$, denoted by $ g_i \rightarrow^\eta g$ in $L^{p_1}(I;L^{p_2}(\Omega_{\eta_i(t)};L^{p_3}_M(B))),$ if
\begin{align*}
\chi_{\Omega_{\eta_i(t)}}g_i \rightarrow \chi_{\Omega_{\eta(t)}}g \quad \text{in} \quad L^{p_1}(I;L^{p_2}(\R^3;L^{p_3}_M(B)));
\end{align*}
\item[(b)] for $p_1, p_2,p_3\in [1,\infty)$, we say that a sequence $g_i \in L^{p_1}(I;L^{p_2}(\Omega_{\eta_i(t)};L^{p_3}_M(B)))$ converges weakly to $g$ in $L^{p_1}(I;L^{p_2}(\Omega_{\eta(t)};L^{p_3}_M(B)))$ with respect to $\eta_i$, denoted by $ g \rightharpoonup^\eta g$ in $L^{p_1}(I;L^{p_2}(\Omega_{\eta_i(t)};L^{p_3}_M(B))),$ if
\begin{align*}
\chi_{\Omega_{\eta_i(t)}}g_i \rightharpoonup \chi_{\Omega_{\eta(t)}}g \quad \text{in} \quad L^{p_1}(I;L^{p_2}(\R^3;L^{p_3}_M(B)));
\end{align*}
\item[(c)] for $p_1=\infty$ and $ p_2,p_3\in [1,\infty)$, we say that a sequence $g_i \in L^{\infty}(I;L^{p_2}(\Omega_{\eta_i(t)};L^{p_3}_M(B)))$ converges weakly$^*$ to $g$ in $L^{\infty}(I;L^{p_2}(\Omega_{\eta(t)};L^{p_3}_M(B)))$ with respect to $\eta_i$, denoted by $ g_i \rightharpoonup^{*,\eta} g$ in $L^{\infty}(I;L^{p_2}(\Omega_{\eta_i(t)};L^{p_3}_M(B))),$ if
\begin{align*}
\chi_{\Omega_{\eta_i(t)}}g_i \rightharpoonup^* \chi_{\Omega_{\eta(t)}}g \quad \text{in} \quad L^{\infty}(I;L^{p_2}(\R^3;L^{p_3}_M(B))).
\end{align*}
\end{itemize}
\end{definition}
Definition \ref{def:varConv} can be extended  in a canonical way to Sobolev spaces.\\

Since we are dealing with boundary value problems, we need a concept of traces on variable domains.
The following lemma is a modification of \cite[Corollary 2.9]{lengeler2014weak}. See also, \cite{muha2014note}. We recall the  transform $\bm{\varphi}_\eta$ from Section \ref{sec:koiterEnergy}.
\begin{lemma}[Trace operator]\label{lem:2.28}
Let $1<p<3$ and $\eta\in W^{2,2}(\partial\Omega)$ with $\|\eta\|_{L^\infty(\partial\Omega)}<L$. Then
the linear mapping
$\tr_\eta:v\mapsto v\circ\bm{\varphi}_\eta|_{\partial\Omega}$ is well defined and continuous from $W^{1,p}(\Omega_\eta)$ to $W^{1-\frac1r,r}(\partial \Omega)$ for all $r\in (1,p)$ and well defined and continuous from  $W^{1,p}(\Omega_\eta)$ to $L^{q}(\partial\Omega)$ for all $1<q<\frac{2p}{3-p}$.
The continuity constants depend only on $\Omega,p,$ and $\|\eta\|_{W^{2,2}(\partial\Omega)}$. 
\end{lemma}

\subsection{Concept of a solution}
\label{subsec:solut}
In order to describe the notion of a solution that we wish to construct, we first define the following energy functionals:
\begin{equation}
\begin{aligned}
\label{energyFunctional}
\mathcal{E}(t) &:=
 \Vert \Xi(t, \bx)\Vert^2_{L^\infty(\Omega_{\eta(t)})} 
+
\int_{\Oeta}
\frac{1}{2} \vert \bu(t, \bx) \vert^2  \dx
+
\int_{\omega}
\frac{1}{2}
 \vert \partial_t \eta(t, \by) \vert^2  \dy
+ K(\eta(t))
\\&+
k \int_{\Oeta \times B} 
M(\bq)\, \mathcal{F} \big( \widehat{\psi}(t, \bx, \bq)\big)
\dq \dx,
\end{aligned}
\end{equation}
where 
\begin{align}
\label{relativeEntropyFunctional}
\mathcal{F}(s):= s\ln s+\mathrm{e}^{-1}\quad\text{ for } \quad s\geq 0
\end{align}
is the  \textit{entropy function} that generates the physical \textit{relative (with respect to the Maxwellian) entropy}
\begin{align}
\label{relativeEntropyDefined}
\mathcal{E}_{\mathcal{F}}^{\eta(t)}\big( \widehat{\psi} \vert M \big):=\int_{\Omega_{\eta(t)} \times B} M \mathcal{F} \big( \widehat{\psi} \big)\dq\dx
\end{align}
and let
\begin{equation}
\begin{aligned}
\label{initialEner}
\mathcal{E}(0) &:=
 \Vert \Xi_0( \bx)\Vert^2_{L^\infty(\Omega_{\eta(0)})} 
 +
\int_{\Omega_{\eta_0}}
\frac{1}{2} \vert \bu_0(\bx) \vert^2  \dx
 +
\int_{\omega}
 \frac{1}{2}
\vert  \eta_1(\by) \vert^2  \dy
+ K(\eta_0)
\\&+
k \int_{\Omega_{\eta_0} \times B} 
M\, \mathcal{F} \big(  \widehat{\psi}_0(\bx,\bq) \big)
\dq \dx
\end{aligned}
\end{equation}
be the initial energy, recall \eqref{etaTzero}.
With the above information in hand, we now proceed to make rigorous, what we mean by a solution.
\begin{definition}[Finite energy weak solution] \label{def:weakSolution}
Let $(\bff, g, \eta_0,  \widehat{\psi}_0, \bu_0, \eta_1)$ be a dataset such that
\begin{equation}
\begin{aligned}
\label{dataset}
&\bff \in L^2\big(I; L^2_{\mathrm{loc}}(\mathbb{R}^3; \mathbb{R}^{3})\big),\quad
g \in L^2\big(I; L^2(\omega)\big), \quad
\eta_0 \in W^{2,2}(\omega) \text{ with } \Vert \eta_0 \Vert_{L^\infty( \omega)} < L, 
\\
& 0\leq\widehat{\psi}_0\in L^2\big( \Omega_{\eta_0}; L^2_M(B) \big), \quad
\vu_0\in L^2_{\mathrm{\divx}}(\Omega_{\eta_0}; \mathbb{R}^{3}) \text{ is such that }\mathrm{tr}_{\eta_0} \bu_0 =\eta_1 \gamma(\eta_0), \quad
\eta_1 \in L^2(\omega).
\end{aligned}
\end{equation} 
In addition, we assume
\begin{align}
\label{XI0xxx}
\Xi_0\in L^\infty(\Omega_{\eta_{0}}) \quad \text{where} \quad \Xi_0=\int_B M(\bq)\widehat{\psi}_0(\cdot,\bq) \dq \quad\text{in}\quad \Omega_{\eta_{0}}.
\end{align}
We call the triple
$(\bu, \widehat{\psi}, \eta  )$
a \textit{finite energy weak solution} to the system \eqref{contEq}--\eqref{fokkerPlankIintial} and \eqref{shellEq}--\eqref{etaNablaEta} with data $(\bff, g, \eta_0,  \widehat{\psi}_0, \bu_0, \eta_1)$ provided that the following holds:
\begin{itemize}
\item[(a)] the velocity $\vu$ satisfies
\begin{align*}
 \vu \in L^\infty \big(I; L^2(\Omega_{\eta(t)} ;\mathbb{R}^3) \big)\cap  L^2 \big(I; W^{1,2}_{\divx}(\Omega_{\eta(t)};\mathbb{R}^3) \big) \quad \text{with} \quad 
\bu(t, \bm{\varphi}_\eta(t,\by)) =\partial_t \eta(t,\by) \bm{\nu}(\by)
\end{align*}
in the sense of traces and $\eta$ satisfies
\begin{align*}
\eta \in W^{1,\infty} \big(I; L^2(\omega) \big)\cap  L^\infty \big(I; W^{2,2}(\omega) \big) \quad \text{with} \quad \Vert \eta \Vert_{L^\infty(I \times \omega)} <L
\end{align*}
and for all  $(\phi, \bm{\varphi}) \in C^\infty(\overline{I}\times\omega) \times C^\infty(\overline{I}\times \R^3; \R^3)$ with $\phi(T,\cdot)=0$, $\bm{\varphi}(T,\cdot)=0$, $\divx \bm{\varphi}=0$ and $\mathrm{tr}_\eta\bm{\varphi}= \phi\bm{\nu}$, we have
\begin{align*}
\int_I  \frac{\mathrm{d}}{\dt}\bigg(\int_{\Oeta}\vu  \cdot \bm{\varphi}\dx
+\int_\omega \partial_t \eta \, \phi \dy
\bigg)\dt 
&=\int_I  \int_{\Oeta}\big(  \vu\cdot \partial_t  \bm{\varphi} + \vu \otimes \vu: \nabx \bm{\varphi} 
  \big) \dx\dt
\\&
-\int_I  \int_{\Oeta}\big(   
\mu \nabx \bu:\nabx \bm{\varphi}  + \mathbb{T}(\psi) :\nabx \bm{\varphi}-\bff\cdot\bm{\varphi} \big) \dx\dt\\
&+
\int_I \int_\omega \big(\partial_t \eta\, \partial_t\phi-
 +g\, \phi \big)\dy\dt-\int_I\langle K'(\eta), \phi\rangle\dt;
\end{align*}
\item[(b)] the probability density function $ \widehat{\psi}$ satisfies:
\begin{align*}
&\widehat{\psi}\geq 0 \text{ a.e. in }  I \times \Omega_{\eta(t)}\times B,
\\ 
&
\widehat{\psi} \in L^\infty \big( I \times \Omega_{\eta(t)} ; L^1_M(B) \big),
\\&\mathcal{F}(\widehat{\psi} ) \in L^\infty \big( I; L^1(\Omega_{\eta(t)}; L^1_M(B))\big),
\\&
\sqrt{\widehat{\psi}} \in  L^2 \big( I; L^2(\Omega_{\eta(t)} ; W^{1,2}_M(B))\big) \cap L^2 \big( I; D^{1,2}(\Omega_{\eta(t)} ; L^2_M(B))\big),
\\&
 \Xi(t,\bx) = \int_BM \widehat{\psi}(t, \bx, \bq) \dq \in L^\infty\big(I \times \Omega_{\eta(t)}\big) \cap L^2\big(I; W^{1,2}(\Omega_{\eta(t)}) \big);
\end{align*}
and for all  $\varphi \in C^\infty (\overline{I}\times \R^3 \times \overline{B} )$, we have
\[
\begin{split}
\int_I  \frac{\mathrm{d}}{\dt} \int_{\Omega_{\eta(t)} \times B}M \widehat{\psi} \, \varphi \dq \dx \dt &= \int_{I \times \Omega_{\eta(t)} \times B}\big(M \widehat{\psi} \,\partial_t \varphi 
+
M\bu  \widehat{\psi} \cdot \nabx \varphi -
\varepsilon M\nabx \widehat{\psi} \cdot \nabx \varphi \big) \dq \dx \dt
\\&
+\sum_{i=1}^K \int_{I \times \Omega_{\eta(t)} \times B}
 \bigg( M(\nabx   \bu ) \mathbf{q}_i\widehat{\psi}-
\sum_{j=1}^K \frac{A_{ij}}{4\lambda}   M \nabqj  \widehat{\psi} \bigg) \cdot \nabqi\varphi \dq \dx \dt;
\end{split}
\]
\item[(c)] for all $t\in I$, we have
\begin{equation}
\begin{aligned}
\label{energyEst}
\mathcal{E}(t)
&+\mu\int_0^t
 \int_{\Omega_{\eta(\sigma)}}\vert \nabx \bu \vert^2 \dx\ds
 +
\varepsilon
  \int_0^t\int_{\Omega_{\eta(\sigma)}}\vert \nabx \Xi \vert^2  \dx\ds
  \\&+
 4k\,\varepsilon \int_0^t
 \int_{\Omega_{\eta(\sigma)} \times B}
 M\Big\vert \nabx \sqrt{ \widehat{\psi} } \Big\vert^2
 \dq \dx\ds
 +
 \frac{kA_0}{\lambda} \int_0^t
 \int_{\Omega_{\eta(\sigma)} \times B}
 M\Big\vert \nabq \sqrt{ \widehat{\psi} } \Big\vert^2
 \dq \dx\ds\\
& \lesssim
 \mathcal{E}(0)
 + \frac{1}{2}\int_0^t\
\Vert \mathbf{f}\Vert^2_{L^2(\Omega_{\eta(\sigma)})}\ds 
+\frac{1}{2}
\int_0^t\
\Vert g  \Vert^2_{L^2(\omega)}\ds.
\end{aligned}
\end{equation}
\end{itemize}
\end{definition}
\begin{remark}
\label{rem:EnerEntr}
In the sequel, for $(t, \mathbf{x}, \mathbf{q})\in I \times \Oeta \times B$,  we shall refer to the following summand
\begin{align}
\label{fisherInformation}
 4k\,\varepsilon \int_0^t
 \int_{ \Omega_{\eta(\sigma)} \times B}
 M\Big\vert \nabx \sqrt{ \widehat{\psi}} \Big\vert^2
 \dq \dx\ds
+
 \frac{kA_0}{\lambda} \int_0^t
 \int_{\Omega_{\eta(\sigma)} \times B}
 M\Big\vert \nabq \sqrt{ \widehat{\psi} } \Big\vert^2
 \dq \dx\ds
\end{align}
as the  \textit{Fisher information}. Any estimate for the energy functional \eqref{energyFunctional} will be referred to as an \textit{energy estimate}. This will include the case where formally speaking, $\psi\equiv 0$. On the other hand, an estimate involving the relative entropy \eqref{relativeEntropyDefined} and any other function of $\psi$ such as the associated Fisher information \eqref{fisherInformation} and the polymer number density \eqref{eta} without any contribution from the solution $\bu$ of the fluid equation will be referred to as an \textit{entropy estimate}. 
\end{remark}
\subsection{Main result}
\label{subsecMainResult}
The main result of this paper is the following:
\begin{theorem}
\label{thm:main}
Let $(\bff, g, \vu_0, \eta_0, \eta_1, \widehat{\psi}_0)$ be a dataset satisfying \eqref{dataset}. Then there exists  a finite energy weak solution  $(\bu,\widehat{\psi}, \eta )$  of  \eqref{contEq}--\eqref{fokkerPlankIintial} and \eqref{shellEq}--\eqref{etaNablaEta} on the interval $I=(0,T)$ in the sense of Definition \ref{def:weakSolution}. The number $T$ is restricted only if $\lim_{t\rightarrow T}\|\eta(t)\|_{L^\infty(\omega)}=L$.
\end{theorem}
\begin{remark}\label{rem:2212} 
 We recall from the definition of $L$ in \eqref{eq:2212} and \eqref{geomQuan} that the Koiter energy does not degenerate and a self-intersection of the moving domain is excluded as long as $\|\eta(t)\|_{L^\infty(\omega)}$ stays strictly below $L$. If, however, we have
$\lim_{t\rightarrow T}\|\eta(t)\|_{L^\infty(\omega)}=L$ it may happen that the
 Koiter energy does
degenerate or that we do have a self-intersection of the shell at time $T$. In this case our existence scheme breaks.
\end{remark}

\section{Formal estimates for the energy and relative entropy}
\label{sec:energEsr}
In this section, we formally derive energy and entropy estimates assuming that we have a sufficiently regular solution  $(\bu,\widehat{\psi}, \eta )$  of  \eqref{contEq}--\eqref{fokkerPlankIintial} and \eqref{shellEq}--\eqref{etaNablaEta}.
Before we begin, we recall the Reynolds transport theorem which states that any vector (or scalar) $\bv:=\bv(t,\bx)$ of class $C^1_{t, \bx}$ satisfies
\begin{equation}
\begin{aligned}
\label{reynold1}
\frac{\dd}{\dt}\int_{\Oeta} \bv \dx 
= \int_{\Oeta} \partial_t \bv \dx
+
\int_{\Oeta} \divx(\bv\otimes \bu) \dx 
= \int_{\Oeta} \partial_t \bv \dx
+\int_{\partial\Oeta} (\bu\cdot \bm{\nu}_{\eta(t)})\bv \dy
\end{aligned}
\end{equation}
in $\Oeta$
where $\bu:=\bu(t,\bx)$ is the velocity of the fluid (If $\bv$ is a scalar $v$, we replace $\bv\otimes \bu$ with $v\bu $). For a solenoidal field $\bu$, the first equality explains the fact that the boundary of $\Oeta$, denoted by $\partial\Oeta$, moves with $\bu$. Note that such a movement is described by the material derivative $\partial_t +(\bu\cdot \nabx)$. The second equation is just Gauss' theorem. 
\\
Following the arguments in \cite[Section 2]{barrett2012finite} (for the simpler case of $K=1)$, let us recall $\mathcal{F}(s):= s \ln s+ \mathrm{e}^{-1}$ for $s>0$  so that $\mathcal{F}'(s)=1+\ln s$ and $\mathcal{F}''(s)=\frac{1}{s}$ are well-defined. We further recall from Section \ref{subsec:poly} that
\begin{equation}
\begin{aligned}
\label{maxwellianIdentity}
M \nabqi \frac{1}{M} = -\frac{1}{M} \nabqi M = \nabqi U\Big(\frac{1}{2}\vert \mathbf{q}_i \vert^2 \Big)  = U'\Big(\frac{1}{2}\vert \mathbf{q}_i \vert^2 \Big) \bq_i.
\end{aligned}
\end{equation}
Now, since our flow is incompressible,  by using Reynolds transport theorem \eqref{reynold1},
\begin{equation}
\begin{aligned}
\label{partialEtaEqu}
\int_{\Oeta \times B}\Big(\partial_t \psi +(\vu \cdot \nabx) \psi\Big) \, \mathcal{F}\,' \bigg( \frac{\psi}{M} \bigg) \dq \dx
&=
\int_{\Oeta \times B}\Big(\partial_t  +\vu \cdot \nabx\Big) \bigg[M\, \mathcal{F} \bigg( \frac{\psi}{M} \bigg)\bigg] \dq \dx
\\&=
\frac{\mathrm{d}}{\mathrm{d}t}
\int_{\Oeta \times B}M\, \mathcal{F} \bigg( \frac{\psi}{M} \bigg) \dq \dx.
\end{aligned}
\end{equation}
Next, since the identities
\begin{equation}
\begin{aligned}
\nonumber
\varepsilon \int_{\Oeta}
 \nabx \psi \cdot \nabx \mathcal{F}\,' \bigg( \frac{\psi}{M} \bigg)
 \dx
 &=
\varepsilon \int_{\Oeta }
 \nabx \psi \cdot \frac{M}{\psi}\nabx  \frac{\psi}{M} 
\dx
=\varepsilon
 \int_{\Oeta}
M \sqrt{\frac{M}{\psi}}\nabx \frac{\psi}{M}\cdot \sqrt{\frac{M}{\psi}}\nabx  \frac{\psi}{M} 
 \dx
 \\&  =
4\,\varepsilon
 \int_{\Oeta}
 M\Bigg\vert \nabx \sqrt{ \frac{\psi}{M} } \Bigg\vert^2
 \dx
\end{aligned}
\end{equation}
and
\begin{equation}
\begin{aligned}
\nonumber
 \varepsilon\int_{\Oeta}
&\divx\bigg[ \nabx \psi\, \mathcal{F}\,' \bigg( \frac{\psi}{M} \bigg)\bigg] \dx
=
  \int_{\partial\Oeta }
 \mathcal{F}\,' \bigg( \frac{\psi}{M} \bigg)\varepsilon \nabx \psi \cdot \bm{\nu}_{\eta(t)}
\dy_{\eta(t)}=0
\end{aligned}
\end{equation}
hold by virtue of the divergence theorem and \eqref{fokkerPlankBoundaryNeumann}, for the center-of-mass diffusion term, we have
\begin{equation}
\begin{aligned}
\varepsilon
\int_{\Oeta \times B}
 \Delx \psi\, \mathcal{F}\,' \bigg( \frac{\psi}{M} \bigg)
 \dq \dx
 &=
 -
 4\,\varepsilon
 \int_{\Oeta \times B}
 M\Bigg\vert \nabx \sqrt{ \frac{\psi}{M} } \Bigg\vert^2
 \dq \dx.
 \end{aligned}
\end{equation}
Now, for each $i=1, \ldots, K$, we can use the fact that $M=0$ on $\partial B$ (see \eqref{boundMiqi}) and the relation $s \nabq \mathcal{F}\,'(s) = \nabq s$ to obtain
\begin{align}
\nonumber
&\int_{\Oeta \times B}
 \divqi  \big( (\nabx   \mathbf{u}) \mathbf{q}_i\psi \big)  \, \mathcal{F}\,' \bigg( \frac{\psi}{M} \bigg)
 \dq \dx
 =
 -
 \int_{\Oeta \times B}
  (\nabx   \mathbf{u}) \mathbf{q}_iM\frac{\psi}{M} \cdot \nabqi  \mathcal{F}\,' \bigg( \frac{\psi}{M} \bigg)
 \dq \dx
\\ \nonumber
& = 
 -
 \int_{\Oeta \times B}
 (\nabx   \mathbf{u}) \mathbf{q}_i\psi \cdot\bigg[   \frac{M \nabqi  \psi - \psi \nabqi  M}{M^2} \bigg]\frac{M}{\psi}\,
 \dq \dx
 \\& =
 \int_{\Oeta \times B}
\divqi \big[ (\nabx   \mathbf{u}) \mathbf{q}_i \big]\psi
 \dq \dx
  +
 \int_{\Oeta \times B}
 (\nabx   \mathbf{u}) \mathbf{q}_i\psi \cdot  \frac{ 1 }{M}\nabqi  M\,
 \dq \dx  \label{referT0}.
\end{align}
However, one can check that the identity $\divqi \big[ (\nabx   \mathbf{u}) \mathbf{q}_i \big] = \divx \bu=0$ holds so that in combination with \eqref{maxwellianIdentity}, we can conclude that
\begin{align}
\sum_{i=1}^K
\int_{\Oeta \times B}
 &\divqi  \big( (\nabx   \mathbf{u}) \mathbf{q}_i\psi \big)   \mathcal{F}\,' \bigg( \frac{\psi}{M} \bigg)
 \dq \dx
 =
  -
  \sum_{i=1}^K
  \int_{\Oeta \times B}
\psi \, U'\Big(\frac{1}{2}\vert \mathbf{q}_i \vert^2 \Big) \bq_i \otimes \bq_i :\nabx   \mathbf{u}
 \dq \dx 
 \nonumber
 \\&
  =
 -
  \sum_{i=1}^K
  \int_{\Oeta}
\mathbb{T}_i(\psi) :\nabx   \mathbf{u}
  \dx
=
   -
\frac{1}{k}
  \int_{\Oeta}
\mathbb{T}(\psi) :\nabx   \mathbf{u}
  \dx  \label{referT}
\end{align}
holds by the use of \eqref{extraStreeTensor} and the observation that $\Xi^p\, \mathbb{I} : \nabx \bu =\Xi^p\, \divx \bu =0$ for $p=1,2$.
Finally, we can use that $M=0$ on $\partial B$ to obtain
\begin{equation}
\begin{aligned}
\label{aijEqu}
\sum_{i=1}^K \sum_{j=1}^K
\frac{A_{ij}}{4\lambda} \int_{\Oeta \times B}
&\divqi  \bigg( M \nabqj  \frac{\psi}{M} 
\bigg)\, \mathcal{F}\,' \bigg( \frac{\psi}{M} \bigg)
 \dq \dx
 \\&=
 -
\sum_{i=1}^K
\sum_{j=1}^K
\frac{A_{ij}}{4\lambda}  \int_{\Oeta \times B}
  M \nabqj  \frac{\psi}{M}  \cdot \frac{M}{\psi}\nabqi  \frac{\psi}{M} 
 \dq \dx
  \\& =
 -
 \sum_{i=1}^K
 \sum_{j=1}^K
\frac{A_{ij}}{\lambda} \,
 \int_{\Oeta \times B}
 M \nabqj \sqrt{ \frac{\psi}{M} }
 \cdot
   \nabqi \sqrt{ \frac{\psi}{M} }
 \dq \dx
\end{aligned}
\end{equation}
If we now collect \eqref{partialEtaEqu}--\eqref{aijEqu} and consider the smallest eigenvalue $A_0>0$ of $(A_{ij})_{i,j=1}^K$, then by integrating over the time interval $[0,t]$, we obtain from \eqref{fokkerPlank} that
\begin{equation}
\begin{aligned}
\label{fokkerEnerxyz}
\int_{\Oeta \times B}
&kM\, \mathcal{F} \bigg( \frac{\psi(t,\cdot) }{M}\bigg) \dq \dx
 +
 4k\,\varepsilon
 \int_0^t
 \int_{\Omega_{\eta(\sigma)} \times B}
 M\Bigg\vert \nabx \sqrt{ \frac{\psi}{M} } \Bigg\vert^2
 \dq \dx \ds
 \\&
 +
\frac{kA_0}{\lambda}  \int_0^t
 \int_{\Omega_{\eta(\sigma)} \times B}
 M\Bigg\vert \nabq \sqrt{ \frac{\psi}{M} } \Bigg\vert^2
 \dq \dx \ds
  \\&
\leq 
\int_{\Omega_{\eta_0} \times B}
kM\, \mathcal{F} \bigg( \frac{\psi_0 }{M}\bigg) \dq \dx
+
 \int_0^t
 \int_{\Omega_{\eta(\sigma)}}
\mathbb{T}(\psi) :\nabx \vu
  \dx \ds
\end{aligned}
\end{equation}
holds for all $t\in I$.
\\
If we also test \eqref{transportXi} with $\Xi$ and use Reynolds' transport theorem \eqref{reynold1} and the boundary condition \eqref{XiBoundaryNeumann}, we obtain
\begin{equation}
\begin{aligned}
\label{intPsiEner}
\int_{\Oeta}
 \vert \Xi(t,\cdot)\vert^2 
\dx
 +
\varepsilon\int_0^t
  \int_{\Omega_{\eta(\sigma)}}\vert \nabx \Xi \vert^2  \dx\ds
  =
  \int_{\Omega_{\eta_0}}
 \vert \Xi_0\vert^2 
\dx.
\end{aligned}
\end{equation}
On the other hand, if we test  \eqref{momEq}  with the velocity $\bu$, then we obtain 
\begin{equation}
\begin{aligned}
\label{fluidEner}
\int_{\Oeta}
\frac{1}{2} \vert \bu(t,\cdot) \vert^2 
\dx
&+\mu
 \int_0^t\int_{\Omega_{\eta(\sigma)} }
 \vert \nabx \bu \vert^2  \dx\ds
+
 \int_0^t\int_{\Omega_{\eta(\sigma)} }
\mathbb{T}(\psi) :\nabx   \mathbf{u}
  \dx\ds
\\&=\int_{\Omega_{\eta_0} }
 \frac{1}{2} \vert \bu_0 \vert^2 
\dx
+
 \int_0^t\int_{\Omega_{\eta(\sigma)} }
\bff \cdot\bu
 \, \dx\ds
 \\&
 +\int_0^t 
  \int_{\partial \Omega_{\eta(\sigma)} } 
\big( 2\mu \,\mathbb{D}_\by\bu 
 + \mathbb{T}(\psi)
 - p\mathbb I
\big) \bm{\nu}_{\eta(\sigma)}\cdot \bu
 \, \dy_{\eta(\sigma)}\ds.
\end{aligned}
\end{equation}
Finally, testing \eqref{shellEq} with  $\partial_t \eta$ yields
\begin{equation}
\begin{aligned}
\label{boundEner}
\int_\omega  \frac{1}{2} \vert\partial_t
\eta(t,\cdot) \vert^2 \dy 
&+  K(\eta(t))
= 
\int_\omega \frac{1}{2} \vert
\eta_1 \vert^2 \dy +  K(\eta_0)
\\&
+
\int_0^t\int_\omega g\, \partial_\sigma \eta \dy \ds +\int_0^t \int_\omega \mathbf{F}\cdot \bm{\nu}\, \partial_\sigma \eta \dy\ds.
\end{aligned}
\end{equation} 
If we now sum up \eqref{fokkerEnerxyz}--\eqref{boundEner}, then we obtain
\begin{equation}
\begin{aligned}
\bigg[
&\int_{\Omega_{\eta(\sigma)}}
\frac{1}{2} \vert \bu(\sigma,\cdot) \vert^2 
\dx
+
\int_{\Omega_{\eta(\sigma)}}
 \vert \Xi(\sigma,\cdot)\vert^2 
\dx
+
\int_\omega  \frac{1}{2} \vert\partial_\sigma
\eta(\sigma,\cdot) \vert^2 \dy
+K(\eta(\sigma))
\\&+
\int_{\Omega_{\eta(\sigma)}\times B}kM\mathcal{F} \bigg( \frac{\psi(\sigma,\cdot) }{M}\bigg) \dq \dx
 \bigg]_{\sigma=0}^{\sigma=t}
+
\mu
\int_0^t
\int_{\Omega_{\eta(\sigma)}}\vert \nabx \bu \vert^2  \dx\ds
  +
\varepsilon
\int_0^t
\int_{\Omega_{\eta(\sigma)}}\vert \nabx \Xi \vert^2  \dx\ds
 \\&+
 4k\,\varepsilon
\int_0^t
\int_{\Omega_{\eta(\sigma)} \times B}
 M\Bigg\vert \nabx \sqrt{ \frac{\psi}{M} } \Bigg\vert^2
 \dq \dx\ds
 +
 \frac{kA_0}{\lambda} \,
 \int_0^t
\int_{\Omega_{\eta(\sigma)} \times B}
 M\Bigg\vert \nabq \sqrt{ \frac{\psi}{M} } \Bigg\vert^2
 \dq \dx\ds
\leq 
I(t),
\end{aligned}
\end{equation}
where
\begin{equation}
\begin{aligned}
I(t)
&:= 
\int_0^t
\int_{\Omega_{\eta(\sigma)}}
\mathbf{f}\cdot\bu
 \dx\ds
+
\int_0^t
\int_\omega g\, \partial_\sigma \eta \dy \ds
\leq
\int_0^t
\int_{\Omega_{\eta(\sigma)}}
\frac{\vartheta}{2}
\vert \bu \vert^2 \dx\ds
\\&+
\int_0^t
\int_\omega \ \frac{\vartheta}{2} \vert \partial_\sigma \eta \vert^2 \dy\ds
+
\int_0^t
\int_{\Omega_{\eta(\sigma)}}
c_\vartheta
 \vert \bff \vert^2 \dx\ds
+
\int_0^t
\int_\omega c_\vartheta\vert g \vert^2 \dy\ds
\end{aligned}
\end{equation}
for any $\vartheta>0$.
The energy estimate  thus follow for all $t\in I$
by the use of \eqref{etaTzero}.

\subsection{Towards making the entropy estimate rigorous}
Since our ultimate goal is the construction of a non-negative probability function $\psi\geq 0$, even on the formal level, the derivation of the energy estimate above is still problematic. For example, $\mathcal{F}'(s)$ and $\mathcal{F}''(s)$ are singular at $s=0$ and thus, the a priori testing of the Fokker--Planck with the first derivative of the relative entropy functional is delicate. This problem can however be fixed by using a convex regularization $\mathcal{F}_\delta$ to approximate $\mathcal{F}$ and then passing to the limit $\delta\rightarrow0$. Following the arguments in \cite{bulicek2013existence}, for a fixed $\delta>0$, let us define $\mathcal{F}_\delta(s):= (s+\delta)\ln(s+\delta)+\mathrm{e}^{-1}$ for $s\geq0$  so that $\mathcal{F}_\delta'(s)=1+\ln(s+\delta)$ and $\mathcal{F}_\delta''(s)=\frac{1}{s+\delta}$ are well-defined even at $s=0$. 
\\
Now, because of the highly coupled nature of the macroscopic fluid system and the mesoscopic Fokker--Planck equation,  we also need to introduce a preliminary smoothing in order to ultimately maintain energy balance. In this respect, we follow \cite{bulicek2013existence} and consider a smooth nonnegative function $\Gamma$ supported on the interval $(-2,2)$ such that $\Gamma(s)=1$ for all $s\in [-1,1]$. Now for $\ell\in \mathbb{N}$, we set $\Gamma_\ell(s) := \Gamma(s/\ell)$ and define $T_\ell$ and $\Lambda_\ell$ by
\begin{align}
\label{tandlambdacutoff}
T_\ell(s) := \int_0^s \Gamma_\ell(r)\, \mathrm{d}r, \qquad \Lambda_\ell(s) := s\Gamma_\ell(s).
\end{align}
If we  further define 
\begin{align}
\label{conGamma}
T_{\delta, \ell}(s) :=\int_0^s\frac{\Lambda_\ell(r)}{r+\delta}\mathrm{d}r 
=
\int_0^s\frac{r\Gamma_\ell(r)}{r+\delta}\mathrm{d}r
\quad
\text{ with }\quad T_{\delta, \ell} \rightarrow T_\ell\quad \text{ in }\quad C([0,\infty)),
\end{align}
as $\delta \rightarrow 0$, then be setting $\widehat{\psi}: =\frac{\psi}{M}$, we observe that
\begin{align*}
\Lambda_\ell \big(\widehat{\psi} \big) \nabqi  \mathcal{F}'_\delta \big(\widehat{\psi} \big) = \nabqi T_{\delta,\ell} \big(\widehat{\psi} \big).
\end{align*}
Therefore, by using the fact that $M=0$ on $\partial B$ (see \eqref{boundMiqi}) and the identity $\divqi \big[ (\nabx   \mathbf{u}) \mathbf{q}_i \big] = \divx \bu=0$ together with \eqref{maxwellianIdentity}, we obtain
\begin{align}
\nonumber
\int_{\Oeta \times B}
 \divqi  \big( M(\nabx   \mathbf{u}) \mathbf{q}_i \Lambda_\ell \big( \widehat{\psi}\big) \big)  \, &\mathcal{F}'_\delta  \big(\widehat{\psi} \big)
 \dq \dx
 =
 -
 \int_{\Oeta \times B}
  M(\nabx   \mathbf{u}) \mathbf{q}_i\cdot \nabqi T_{\delta,\ell}\big(\widehat{\psi}\big)
 \dq \dx
 \\&
 =
-  
  \int_{\Oeta \times B}
M\, T_{\delta,\ell}\big(\widehat{\psi}\big) \, U'\Big(\frac{1}{2}\vert \mathbf{q}_i \vert^2 \Big) \bq_i \otimes \bq_i :\nabx   \mathbf{u}
 \dq \dx 
\label{referT0}.
\end{align}
for the corresponding `approximate' drag term in \eqref{fokkerPlankx}. By treating the rest of the terms in the modified Fokker--Planck equation \eqref{fokkerPlankx} similarly to the previous subsection, we obtain
\begin{equation}
\begin{aligned}
\label{fokkerEner0}
\frac{\mathrm{d}}{\mathrm{d}t}
\int_{\Oeta \times B}
&kM\, \mathcal{F}_\delta \big( \widehat{\psi} \big) \dq \dx
 +
k\,\varepsilon\int_{\Oeta \times B}
M\big\vert\nabx \widehat{\psi} \big\vert^2\big( \widehat{\psi} +\delta \big)^{-1} \dq \dx
 \\&
 +
 \sum_{i=1}^K
\sum_{j=1}^K
\frac{kA_{ij}}{4\lambda}  \int_{\Oeta \times B}
  M \nabqj  \widehat{\psi}  \cdot\nabqi \widehat{\psi} 
\big( \widehat{\psi} + \delta \big)^{-1}
 \dq \dx
\\&=
 k\,  \sum_{i=1}^K\int_{\Oeta \times B}
M\, T_{\delta,\ell}\big( \widehat{\psi} \big) \, U'\Big(\frac{1}{2}\vert \mathbf{q}_i \vert^2 \Big) \bq_i \otimes \bq_i :\nabx   \mathbf{u}
 \dq \dx .
\end{aligned}
\end{equation}
If we now consider the smallest eigenvalue $A_0>0$ of $(A_{ij})_{i,j=1}^K$, then by integrating \eqref{fokkerEner0} over the time interval $[0,t]$, we obtain
\begin{equation}
\begin{aligned}
\label{fokkerEnerA}
&\int_{\Oeta \times B}
kM\, \mathcal{F}_\delta \big( \widehat{\psi}(t,\cdot)\big) \dq \dx
 +
k\,\varepsilon \int_0^t\int_{\Omega_{\eta(\sigma)} \times B}
M\big\vert\nabx \widehat{\psi} \big\vert^2 \big(\widehat{\psi}+\delta \big)^{-1} \dq \dx\ds
 \\&
 +
\frac{kA_{0}}{4\lambda}\int_0^t  \int_{\Omega_{\eta(\sigma)}\times B}
  M \big\vert \nabq  \widehat{\psi}  \big\vert^2
\big(\widehat{\psi}+\delta \big)^{-1}
 \dq \dx\ds
 \leq
 \int_{\Omega_{\eta_0} \times B}
 kM\, \mathcal{F}_\delta \big( \widehat{\psi}_0 \big) \dq \dx
\\&+
 k\,\int_0^t\int_{\Omega_{\eta(\sigma)} \times B}
 \sum_{i=1}^K
M\, T_{\delta,\ell}\big(\widehat{\psi} \big) \, U'\Big(\frac{1}{2}\vert \mathbf{q}_i \vert^2 \Big) \bq_i \otimes \bq_i :\nabx   \mathbf{u}
 \dq \dx\ds.
\end{aligned}
\end{equation}
By recalling \eqref{extraStreeTensor} and the observation that $\Xi^p\, \mathbb{I} : \nabx \bu =\Xi^p\, \divx \bu =0$ for $p=1,2$, we can use \eqref{conGamma} to obtain
\begin{equation}
\begin{aligned}
\label{formalConvdelta}
\limsup_{\delta\rightarrow 0} k\,&\int_0^t\int_{\Omega_{\eta(\sigma)} \times B}
\sum_{i=1}^K
M\, T_{\delta,\ell}\big(\widehat{\psi} \big) \, U'\Big(\frac{1}{2}\vert \mathbf{q}_i \vert^2 \Big) \bq_i \otimes \bq_i :\nabx   \mathbf{u}
 \dq \dx\ds 
 \\&=
  k\,\int_0^t\int_{\Omega_{\eta(\sigma)} \times B}
 \sum_{i=1}^K
M\, T_\ell\big(\widehat{\psi}\big) \, U'\Big(\frac{1}{2}\vert \mathbf{q}_i \vert^2 \Big) \bq_i \otimes \bq_i :\nabx   \mathbf{u}
 \dq \dx\ds
  \\&=
  \int_0^t\int_{\Omega_{\eta(\sigma)} }
\mathbb{T}_\ell(M\widehat{\psi})  :\nabx   \mathbf{u}
  \dx\ds
\end{aligned}
\end{equation}
where for $\widehat{\psi}=\psi/M$,
\begin{align}
\label{extraStreeTensorxApprox}
\mathbb{T}_\ell(M\widehat{\psi}) 
=
k\sum_{i=1}^K  \int_BM \, \nabqi T_\ell(\widehat{\psi})\otimes\mathbf{q}_i \dq
 -
  k\,\Xi \,\mathbb{I}
-\eth \,\Xi^2 \mathbb{I}
\end{align}
is a truncated extra stress tensor in the fluid system.
The convergence in \eqref{formalConvdelta} together with the monotone convergence theorem and the identities
\begin{equation}
 \frac{M}{\widehat{\psi}}\big\vert\nabx  \widehat{\psi} \big\vert^2
  =
4
 M\big\vert \nabx \sqrt{ \widehat{\psi} } \big\vert^2, 
 \quad \quad
  \frac{M}{\widehat{\psi}}\big\vert\nabq  \widehat{\psi} \big\vert^2
  =
4
 M\big\vert \nabq \sqrt{\widehat{\psi} } \big\vert^2
\end{equation}
allows us to finally obtain in the inequality
\begin{equation}
\begin{aligned}
\label{fokkerEner}
&\int_{\Oeta \times B}
kM\, \mathcal{F}\big( \widehat{\psi}(t,\cdot) \big) \dq \dx
 +
4k\,\varepsilon \int_0^t\int_{\Omega_{\eta(\sigma)} \times B}
M\Big\vert \nabx \sqrt{ \widehat{\psi} } \Big\vert^2 \dq \dx\ds
 \\&
 +
\frac{kA_{0}}{\lambda}\int_0^t  \int_{\Omega_{\eta(\sigma)}\times B}
  M\Big\vert \nabq \sqrt{ \widehat{\psi} } \Big\vert^2
 \dq \dx\ds
 \leq
 \int_{\Omega_{\eta_0} \times B}
 kM\, \mathcal{F} \big( \widehat{\psi}_0 \big) \dq \dx
\\&+
 \int_0^t\int_{\Omega_{\eta(\sigma)} }
\mathbb{T}_\ell(M\widehat{\psi})  :\nabx   \mathbf{u}
  \dx\ds.
\end{aligned}
\end{equation}
by passing to limit $\delta \rightarrow 0$ in \eqref{fokkerEnerA}. The last term in \eqref{fokkerEner} above then balances with the corresponding term in the macroscopic system for the  solvent.

\section{Solving the Fokker--Planck equation}
\label{sec:SolveFokkerPlanck}
In this section we are going to to solve to Fokker-Planck equation for a given function $\xi: I\times\omega \rightarrow (-L,L)$ and a given vector field $\bv:I\times\Omega_\xi\rightarrow\R^3$ such that 
\begin{equation}
\begin{aligned}
\label{datasetReg0}
\xi \in  C^3(\overline{I} \times  \omega ), \qquad \bv\in L^\infty \big(I; W^{1,\infty}_{\divx}( \Omega_{\xi} ;\mathbb{R}^3) \big).
\end{aligned}
\end{equation}
Now subject to the following initial and boundary conditions
\begin{align}
&\bigg[\frac{1}{4\lambda}  \sum_{j=1}^K A_{ij}\, M \nabqj  \widehat{\psi} - M(\nabx   \bv) \bq_i\widehat{\psi}
 \bigg] \cdot \bn_i =0
&\quad \text{on }I \times \Omega_{\xi} \times \partial B_i, \quad\text{for } i=1, \ldots, K,
\label{fokkerPlankBoundaryGivenFluid0}
\\
&\varepsilon \naby \widehat{\psi} \cdot \bm{\nu}_{\xi} =\widehat{\psi}(\bv-(\partial_t\xi\bm{\nu})\circ\bm{\varphi}^{-1}_\xi)\cdot\bm{\nu}_\xi 
&\quad \text{on }  I \times \partial \Omega_{\xi(t)} \times B, \label{fokkerPlankBoundaryNeumannGivenFluid0}
\\
&\widehat{\psi}(0, \cdot, \cdot) =\widehat{\psi}_0 \geq 0
& \quad \text{in }\Omega_{\xi(0)} \times B,
\label{fokkerPlankIintialGivenFluid0}
\end{align}we aim to construct a weak solution to the Fokker--Planck equation 
\begin{align}
\label{fokkerPlankGivenVelo}
\partial_t (M\widehat{\psi}) + (\bv\cdot \nabx) M\widehat{\psi}
=
\varepsilon \Delx (M\widehat{\psi})
- \sum_{i=1}^K
 \divqi  \big( M(\nabx   \bv) \bq_i\widehat{\psi} \big) 
 +
\frac{1}{4\lambda} \sum_{i=1}^K \sum_{j=1}^K A_{ij}\,  \divqi  \big( M \nabqj  \widehat{\psi} 
\big)
\end{align}
in $I \times \Omega_{\xi(t)} \times B$. Here  $\varepsilon>0$  is the center-of-mass diffusion coefficient, $\lambda > 0$ is the \textit{Deborah number} $\mathrm{De}$ and the $A_{ij}$'s are the components of the symmetric positive definite \textit{Rouse matrix} $(A_{ij})_{i,j=1}^K$ whose smallest eigenvalue is $A_0>0$. The Maxwellian $M$ satifies \eqref{maxwellianPartial}--\eqref{boundMiqi} and $\widehat{\psi}=\widehat{\psi}(t,\bx,\bq)$ is the unknown. \\
We now make precise, what we mean by $\widehat{\psi}$ being a solution of \eqref{fokkerPlankBoundaryGivenFluid0}--\eqref{fokkerPlankGivenVelo}.
\begin{definition} \label{def:RegweakSolution}
Take $(\bv, \xi, \widehat{\psi}_0)$ as a dataset such that $(\bv, \xi )$ satisfies \eqref{datasetReg0} and
\begin{equation}
\begin{aligned}
\label{datasetReg}
\widehat{\psi}_0 \in L^2\big( \Omega_{\xi(0)}; L^2_M(B) \big), \quad \widehat{\psi}_0\geq 0.
\end{aligned}
\end{equation}
In addition, we assume
\begin{align}
\label{XI0}
\Xi_0\in L^\infty(\Omega_{\xi(0)}) \quad \text{where} \quad \Xi_0=\int_B M(\bq)\widehat{\psi}_0(\cdot,\bq) \dq \quad\text{in}\quad \Omega_{\xi(0)}.
\end{align}
We call $\widehat{\psi}$ a finite energy weak solution of  \eqref{fokkerPlankBoundaryGivenFluid0}--\eqref{fokkerPlankGivenVelo} with data $(\bv, \xi, \widehat{\psi}_0)$ provided that:
\begin{itemize}
\item[(a)] the probability density function $\widehat{\psi}$ satisfies
\begin{align*}
&\widehat{\psi}\geq 0 \text{ a.e. in }  I \times \Omega_{\xi(t)}\times B,
\\ 
&
\widehat{\psi} \in L^\infty \big( I \times \Omega_{\xi(t)} ; L^1_M(B) \big),
\\&\mathcal{F}(\widehat{\psi} ) \in L^\infty \big( I; L^1(\Omega_{\xi(t)}; L^1_M(B))\big),
\\&
\sqrt{\widehat{\psi}} \in  L^2 \big( I; L^2(\Omega_{\xi(t)} ; W^{1,2}_M(B))\big) \cap L^2 \big( I; D^{1,2}(\Omega_{\xi(t)} ; L^2_M(B))\big),
\\&
 \Xi(t,\bx) = \int_BM \widehat{\psi}(t, \bx, \bq) \dq \in L^\infty\big(I \times \Omega_{\xi(t)}\big) \cap L^2\big(I; W^{1,2}(\Omega_{\xi(t)}) \big);
\end{align*}
\item[(b)] for all  $\varphi \in C^\infty (\overline{I}\times \R^3 \times \overline{B})$, we have
\begin{equation}
\begin{aligned}
\int_I  \frac{\mathrm{d}}{\dt} \int_{\Omega_{\xi(t)} \times B}M \widehat{\psi} \, \varphi &\dq \dx \dt = \int_{I \times \Omega_{\xi(t)} \times B}\big(M \widehat{\psi} \,\partial_t \varphi 
+
M\bv  \widehat{\psi} \cdot \nabx \varphi -
\varepsilon M\nabx \widehat{\psi} \cdot \nabx \varphi \big) \dq \dx \dt
\\&
+\sum_{i=1}^K \int_{I \times \Omega_{\xi(t)} \times B}
 \bigg( M(\nabx   \bv ) \mathbf{q}_i\widehat{\psi}-
\sum_{j=1}^K \frac{A_{ij}}{4\lambda}   M \nabqj  \widehat{\psi} \bigg) \cdot \nabqi\varphi \dq \dx \dt;
\label{weakFPlimit0}
\end{aligned}
\end{equation}
\item[(c)] we have the estimate
\begin{equation}
\begin{aligned}
\label{energyEstKappaNeg1}
&
\Vert \Xi(t,\cdot) \Vert^2_{L^\infty(\Omega_{\xi(t)})}
 +
k \int_{\Omega_{\xi(t)} \times B} 
M\, \mathcal{F} \big( \widehat{\psi}(t, \cdot, \cdot) \big)
\dq \dx
+
\varepsilon \int_0^t
 \int_{\Omega_{\xi(\sigma)}}\vert \nabx \Xi \vert^2 \dx\ds
\\&
+
 4k\,\varepsilon \int_0^t
 \int_{ \Omega_{\xi(\sigma)} \times B}
 M\Big\vert \nabx \sqrt{ \widehat{\psi}} \Big\vert^2
 \dq \dx\ds
+
 \frac{kA_0}{\lambda} \int_0^t
 \int_{\Omega_{\xi(\sigma)} \times B}
 M\Big\vert \nabq \sqrt{ \widehat{\psi} } \Big\vert^2
 \dq \dx\ds
 \\&
 \lesssim 
\Vert \Xi_0  \Vert^2_{L^\infty(\Omega_{\xi(0)})}
+
k \int_{\Omega_{\xi(0)}\times B} 
M\, \mathcal{F} \big( \widehat{\psi}_0 \big)
\dq \dx  
+
\int_0^t  \int_{\Omega_{\xi(\sigma)} }
\mathbb{T}(M\widehat{\psi}) :\nabx \bv
  \dx\ds 
\end{aligned}
\end{equation}
for all $t\in I$.
\end{itemize}
\end{definition}
The main result of this section is the following:
\begin{theorem}
\label{thm:kappaReg} 
Assume that $(\bv, \xi, \widehat{\psi}_0)$ satisfies \eqref{datasetReg0} and \eqref{datasetReg}--\eqref{XI0}.
Then there exists a finite energy weak solution of  \eqref{fokkerPlankBoundaryGivenFluid0}--\eqref{fokkerPlankGivenVelo} in the sense of Definition \ref{def:RegweakSolution}.
\end{theorem}

We now devote the rest of this section to the proof of Theorem \ref{thm:kappaReg}. As is usually the case, we begin by establishing an approximate solution of the Fokker--Planck equation \eqref{fokkerPlankGivenVelo}.

\subsection{A Galerkin scheme}
We let $(\bv, \xi,  \widehat{\psi}_0)$ be a given dataset satisfying \eqref{datasetReg0} and \eqref{datasetReg}--\eqref{XI0}.
Our solution $\widehat{\psi}$ to \eqref{fokkerPlankGivenVelo} will be derived as a limit $n\rightarrow \infty$, $m\rightarrow \infty$, $\ell \rightarrow \infty$ (in that order) of a regular solution $\widehat{\psi}^{\ell,m,n}$ where $n\in\mathbb{N}$ is a Galerkin parameter and  $m\in \mathbb{N}$ and $\ell \in \mathbb{N}$ are parameters indexing a family of approximate Maxwellians and extra stress tensors respectively.
\begin{remark}
\label{rem:noL}
Before we begin, we remark that henceforth, for a given pair $(\bv, \xi)$ and fixed $\ell\in \mathbb{N}$, we simply write $\widehat{\psi}^{m,n}$ in place of $(\widehat{\psi}^{\ell,m,n})_{\ell\in \mathbb{N}}$ until  when we have to pass to the limit $\ell \rightarrow \infty$. We also recall the definition of $T_\ell$ and $\Lambda_\ell$ in \eqref{tandlambdacutoff}. 
\end{remark}
Now to begin, we first consider the following approximate form of the Maxwellian in order to avoid potential blow-up. Following \cite[Section 2.1]{bulicek2013existence}, we approximate $M=M(\bq)$ by $M^m=M^m(\bq)$ given by
\begin{align}
M^m := \overline{M}^m +\frac{1}{m}, \quad m\in \mathbb{N}
\end{align}
where $\overline{M}^m=\overline{M}^m(\bq)$ is such that
\begin{align}\label{eq:Mm}
\lim_{m\rightarrow \infty}\bigg\{\big\Vert \overline{M}^m -M \big\Vert_{C(\overline{B}) \cap W^{1,1}_0(B)}
+
\big\Vert (\overline{M}^m)^{-1} -M^{-1} \big\Vert_{C(\mathcal{B}_\bq) }
\bigg\}=0
\end{align}
for any compact set $\mathcal{B}_\bq \subset B$. Next, for each $m\in \mathbb{N}$, we consider a countable family 
$(\phi^m_i)_{i\in \mathbb{N}}$ of eigenfunctions in $W^{1,2}(B)$ that are orthogonal in $W^{1,2}_{M^m}(B)$ and orthonormal in $L^{2}_{M^m}(B)$.
We also consider a smooth basis $( \varpi_j)_{j\in\mathbb{N}}$ of the space $W^{1,2}(\R^3)$ along 
restriction to $\Omega_\xi$ so that $( \varpi_j  )_{j\in\mathbb{N}}$ is  a basis  of the space $W^{1,2}(\Omega_{\xi(t)} )$ for any $t\in I$ (this, in fact, follows form the existence of an extension operator to $\R^3$ since $\xi(t)$ is smooth). 
For fixed $m,n\in \mathbb{N}$, our aim now is to look for a function
\begin{align}
\label{spannedBasis}
\widehat{\psi}^{m,n}(t,\bx,\bq)=\sum_{r=1}^n\sum_{l=1}^n f^{m,n}_{rl}(t)\phi^m_r(\bq) \varpi_l(\bx),
\end{align}
where the $f^{m,n}_{rl}$'s are the actual unknowns, such that for all $r = 1, \dots, n$ and $l=1,\dots,n$,
\begin{equation}
\begin{aligned}
\frac{\mathrm{d}}{\mathrm{d}t}&
\int_{\Omega_{\xi(t)} \times B } M^m\widehat{\psi}^{m,n} \phi^m_r \varpi_l\dq \dx
\\&
=
\int_{\Omega_{\xi(t)}  \times B }    \bv\,M^m  \widehat{\psi}^{m,n} \cdot \nabx( \phi^m_r \varpi_l) \dq\dx
-
 \varepsilon 
\int_{\Omega_{\xi(t)}  \times B } M^m \nabx \widehat{\psi}^{m,n}\cdot \nabx( \phi^m_r \varpi_l) \dq\dx
\\&
+
\sum_{i=1}^K
 \int_{\Omega_{\xi(t)}  \times B }   M(\nabx   \bv) \bq_i \Lambda_\ell(\widehat{\psi}^{m,n})   \cdot \nabqj( \phi^m_r \varpi_l) \dq\dx 
\\&
-
\sum_{i=1}^K \sum_{j=1}^K \frac{A_{ij}}{4\lambda} \int_{\Omega_{\xi(t)}  \times B } M^m \nabqi \widehat{\psi}^{m,n}
  \cdot \nabqj( \phi^m_r \varpi_l) \dq\dx
\label{eq:FPn}
\end{aligned}
\end{equation}
holds subject to the initial condition
\begin{equation}
\begin{aligned}
\widehat{\psi}^{m,n}(0, \cdot, \cdot) &=\widehat{\psi}_0^{m,n}:= 
\sum_{r=1}^n\sum_{l=1}^n\bigg(
\int_{\Omega_{\xi(0)} \times B } T_\ell(\widehat{\psi}^{m}_0) \phi^m_r \varpi_l \dq \dx \bigg)\phi^m_r(\bq) \varpi_l (\bx)
,
\\&
\widehat{\psi}^m_0
:=
\widehat{\psi}_0 \frac{M}{M^m}
 \quad \text{in} \quad\Omega_{\xi(0)} \times B.
\end{aligned}
\end{equation}
Equation \eqref{eq:FPn} is a system of ODEs with continuous coefficients which can be solved locally in time.  We start by showing uniform a priori estimates for $\widehat{\psi}^{m,n}$ which will imply global solvability of \eqref{eq:FPn}.
To obtain these estimates, we will take $\widehat{\psi}^{m,n}$ as test function in \eqref{eq:FPn}. Using Reynolds' transport theorem \eqref{reynold1} we obtain  
\begin{align*}
\frac{\mathrm{d}}{\mathrm{d}t}
\int_{\Omega_{\xi(t)} \times B } &M^m\vert\widehat{\psi}^{m,n}\vert^2 \dq \dx
-
\int_{\Omega_{\xi(t)} \times B }M^m \widehat{\psi}^{m,n} \partial_t\widehat{\psi}^{m,n} \dq \dx
\\&=
\frac{1}{2}
\frac{\mathrm{d}}{\mathrm{d}t} \int_{\Omega_{\xi(t)} \times B }M^m\vert\widehat{\psi}^{m,n}\vert^2  \dq \dx 
+\frac{1}{2}
\int_{\partial\Omega_{\xi(t)} \times B} M^m\vert\widehat{\psi}^{m,n}\vert^2 (\partial_t\xi \bm{\nu})\circ \bm{\varphi}_\xi^{-1}\cdot\bm{\nu}_\xi\dq \dy,
\end{align*}
whereas the divergence theorem yields
\begin{align*}
\int_{\Omega_{\xi(t)}  \times B }  \bv\,M^m  \widehat{\psi}^{m,n} \cdot \nabx \widehat{\psi}^{m,n}  &\dq\dx
=
\frac{1}{2}
\int_{\partial\Omega_{\xi(t)} \times B} M^m\vert\widehat{\psi}^{m,n}\vert^2 (\bv \cdot \bm{\nu}_{\xi})\dq \dy
\end{align*}
since $\Div\bv=0$.
By treating the rest of the terms in \eqref{eq:FPn} in a straightforward manner
and integrating 
over the interval $(0,t)$, we obtain
\begin{align}
\frac{1}{2}\int_{\Omega_{\xi(\sigma)} \times B}M^m\vert\widehat{\psi}^{m,n}(\sigma)\vert^2\dq \dx \bigg]_{\sigma=0}^{\sigma=t}     
&+
\sum_{i=1}^K \sum_{j=1}^K \frac{A_{ij}}{4\lambda}
\int_0^t \int_{\Omega_{\xi(\sigma)} \times B} M^m \nabqi \widehat{\psi}^{m,n} \cdot \nabqj \widehat{\psi}^{m,n}  \dq  \dx \, \mathrm{d}\sigma\nonumber
\\ 
+
\varepsilon
\int_0^t
\int_{\Omega_{\xi(\sigma)} \times B} M^m \vert \nabx \widehat{\psi}^{m,n} \vert^2 \dq  \dx \, \mathrm{d}\sigma
&=
\sum_{i=1}^K
\int_0^t
\int_{\Omega_{\xi(\sigma)} \times B}M^m (\nabx   \bv) \bq_i \Lambda_\ell(\widehat{\psi}^{m,n})   \cdot \nabqi \widehat{\psi}^{m,n}\dq \dx \, \mathrm{d}\sigma\nonumber
\\&
+
\frac{1}{2}
\int_0^t
\int_{\partial\Omega_{\xi(\sigma)} \times B} M^m\vert\widehat{\psi}^{m,n}\vert^2 (\bv \cdot \bm{\nu}_{\xi})\dq \dy \, \mathrm{d}\sigma\nonumber \\
&-\frac{1}{2}
\int_0^t
\int_{\partial\Omega_{\xi(\sigma)} \times B}  M^m\vert\widehat{\psi}^{m,n}\vert^2 (\partial_\sigma\xi \bm{\nu})\circ \bm{\varphi}_\xi^{-1}\cdot\bm{\nu}_\xi\dq \dy \ds \nonumber\\
&= :(I)^{m,n}+(II)^{m,n}+(III)^{m,n}
\label{distriFormPsiApprZ0}
\end{align}
for all $t\in I$. As a consequence of Young's inequality, the definitions of $\Lambda_\ell$ and $M^m$,  boundedness of $ \nabx  \bv $ and the boundedness of $\vert \bq_i\vert^2$ which holds for each $i=1, \ldots, K$, we obtain the estimate
\begin{align*}
(I)^{m,n}
&\leq
\delta \int_0^t
\int_{\Omega_{\xi(\sigma)} \times B}  M^m
\vert
\nabq \widehat{\psi}^{m,n} \vert^2 \dq  \dx \, \mathrm{d}\sigma
+
c_\delta
\int_0^t
\int_{\Omega_{\xi(\sigma)} \times B} M^m  \vert \widehat{\psi}^{m,n} \vert^2 \dq \dx\, \mathrm{d}\sigma
\end{align*}
uniformly in $m,n\in \mathbb{N}$ for any $\delta>0$ recalling the definition of $\Lambda_\ell$.
 In order to estimate $(II)^{m,n}$, we use the trace embedding $W^{1/2,2}(\Omega_{\xi(t)})\hookrightarrow L^2(\partial \Omega_{\xi(t)})$. The constant in this embedding depends on the $W^{1,\infty}_{\bx}$-norm of $\xi$. 
Employing also an interpolation argument and using the regularity of $\bfv$ again, we obtain 
\begin{align*}
(II)^{m,n}
&\lesssim
\int_0^t\int_BM^m
\|\widehat{\psi}^{m,n}\|^2_{L^2(\partial\Omega_{\xi(\sigma)})}  \dq\,\dd\sigma\lesssim \int_0^t\int_BM^m
\|\widehat{\psi}^{m,n}\|^2_{W^{1/2,2}(\Omega_{\xi(\sigma)})}  \dq\,\dd\sigma
 \\
&\leq 
\delta \int_0^t
\int_{\Omega_{\xi(\sigma)} \times B}  M^m
\vert
\nabx \widehat{\psi}^{m,n} \vert^2 \dq  \dx \, \mathrm{d}\sigma
+
c_\delta
\int_0^t
\int_{\Omega_{\xi(\sigma)} \times B} M^m  \vert \widehat{\psi}^{m,n} \vert^2 \dq \dx\, \mathrm{d}\sigma
\end{align*}
uniformly in $m,n\in \mathbb{N}$ for any $\delta>0$. By employing the regularity of $\xi$ in place of $\bv$, we can derive the same estimate for $(III)^{m,n}$. 
If we now consider the smallest eigenvalue $A_0>0$ of $(A_{ij})_{i,j=1}^K$, absorb the small $\delta$-terms into the corresponding terms on the left-hand side of \eqref{distriFormPsiApprZ0} and apply Gronwall's lemma, then we obtain
\begin{align}\label{suptx3Z}
\begin{aligned}
\int_{\Omega_{\xi(t)} \times B} & M^m
\vert
 \widehat{\psi}^{m,n} \vert^2 \dq \dx
+
\int_0^t
\int_{\Omega_{\xi(\sigma)} \times B}  M^m
\vert
\nabx \widehat{\psi}^{m,n} \vert^2 \dq  \dx\,\mathrm{d}\sigma
\\
&+
\int_0^t
\int_{\Omega_{\xi(\sigma)} \times B}  M^m
\vert
\nabq \widehat{\psi}^{m,n} \vert^2 \dq  \dx \,\mathrm{d}\sigma
\lesssim
\int_{\Omega_{\xi(0)} \times B}  M^m
\vert
 \widehat{\psi}^{m,n}_0 \vert^2\dq \dx
\end{aligned}
\end{align}
uniformly in $m,n\in \mathbb{N}$ for all $t\in I$
with a constant that depends only on $A_0$, $\varepsilon>0$, $\bv$ and $\xi$. Note that \eqref{suptx3Z} implies that there is a global-in-time solution to \eqref{eq:FPn}. Furthermore, we can use the fact that $\frac{1}{m}\leq M^m \leq c$, where $c>0$ is a constant independent of $m$ to conclude from \eqref{suptx3Z} that 
\begin{align}\label{suptx3Zx}
\begin{aligned}
\sup_{t\in I}\int_{\Omega_{\xi(t)} \times B} & 
\vert
 \widehat{\psi}^{m,n} \vert^2 \dq \dx
+
\int_I
\int_{\Omega_{\xi(t)} \times B}  
\vert
\nabx \widehat{\psi}^{m,n} \vert^2 \dq  \dx \dt
\\
&+
\int_I
\int_{\Omega_{\xi(t)} \times B}  
\vert
\nabq \widehat{\psi}^{m,n} \vert^2 \dq  \dx \dt
\lesssim_m 1
\end{aligned}
\end{align}
where the constant additionally depends only on $A_0$, $\varepsilon>0$, $\bv$ and $\xi$.

\subsection{Regularity and identification of the first limits}
To begin this subsection, we first note that since the estimate \eqref{suptx3Zx} holds uniformly in $n\in \mathbb{N}$,  it follows that there exist 
\begin{align}
\label{psiLimReg1}
\widehat{\psi}^m \in L^\infty(I; L^2( \Omega_{\xi(t)} ; L^2( B)))
\cap
L^2 \big( I; D^{1,2}( \Omega_{\xi(t)}  ; L^2(B))\big)
\cap
L^2 \big( I; L^2( \Omega_{\xi(t)}  ; W^{1,2}(B))\big) 
\end{align}
such that $\widehat{\psi}^{m,n}$ converges weakly-$(*)$ to $\widehat{\psi}^m$ as $n\rightarrow \infty$. The limit $\widehat{\psi}^m$
satisfies both \eqref{suptx3Z} and \eqref{suptx3Zx}.
Furthermore, 
since $\frac{3}{5} = 3\left( \frac{1}{2} - \frac{3}{10}\right)$, it follows from the Gagliardo--Nirenberg interpolation inequality that for all $t\in I$, the estimate
\begin{align*}
\Vert \widehat{\psi}^m \Vert_{L^\frac{10}{3}(I \times \Omega_{\xi(t)};L^2(B) )}
\lesssim
\Vert \widehat{\psi}^m \Vert_{L^\infty(I;L^2(\Omega_{\xi(t)};L^2(B)))}^{2/5}
\Vert \widehat{\psi}^m \Vert_{L^2(I;D^{1,2}(\Omega_{\xi(t)};L^2(B))) }^{3/5}
\end{align*}
holds for a constant depending only on $\xi$. Thus, given \eqref{psiLimReg1}, it follows that
\begin{align}
\label{psiLimReg2}
\widehat{\psi}^m \in L^\frac{10}{3}\big(I \times \Omega_{\xi(t)};L^2(B)\big) .
\end{align}
Although the original equation \eqref{fokkerPlankGivenVelo} is linear in $\psi$, the regularised equation (which we approximate by \eqref{eq:FPn}) is not due to the modified drag-term.
Hence, we need compactness to pass to the limit. On account of the moving shell, the standard argument
based on Aubin--Lions lemma does not apply. Also, we cannot localise the equation on the Galerkin level (this will be done later in Sections \ref{subsec:2nd} and \ref{subsec:comp}). Hence, we modify the compactness approach from 
\cite{lengeler2014weak} for our purposes
and prove that
\begin{align}\label{eq:312N}
\begin{aligned}
\int_{I\times \Omega_{\xi(t)} \times B}&M^m|\widehat\psi^{m,n}|^2\dq\dx\dt
&\longrightarrow \int_{I\times \Omega_{\xi(t)} \times B} &M^m|\widehat\psi^{m}|^2\dq\dx\dt
\end{aligned}
\end{align}
as $n\rightarrow\infty$.
In order to do so, we consider for $\varphi\in W^{1,2}(\R^3\times B)$ with $\|\varphi\|_{W^{1,2}_{(\bx,\bq)}}\leq 1$, the functions
\begin{align*}
c_{\varphi,n}(t)&=\int_{\Omega_{\xi(t)} \times B} M^m\widehat\psi^{m,n}(t)\,\Pi_n^{\bq}\Pi_n^{\bx}\varphi\dq\dx,\\
c_{\varphi}(t)&=\int_{\Omega_{\xi(t)} \times B} M^m\widehat\psi^{m}(t)\,\varphi\dq\dx,
\end{align*}
where $\Pi_n^{\bq}$  and $\Pi_n^{\bx}$ are $L^2$-projections onto $(\phi_i)_{i=1}^n$ and $({\varpi}_i)_{i=1}^n$ respectively. We aim to show
\begin{align}\label{eq:1601N}
\big\|c_{\varphi,n}\big\|_{C^{0,1/\chi'}(\overline I)}\lesssim 1
\end{align}
for some $\chi>1$ uniformly in $n$ and $\varphi$. This follows by inserting $\Pi_n^{\bq}\Pi_n^{\bx}\varphi$ into \eqref{eq:FPn} and integrating in time. The result is then a consequence of the a priori estimate from \eqref{suptx3Zx} and the continuity of $\Pi_n^{\bq}$ and $\Pi_n^{\bx}$.
Taking a subsequence and applying Arcel\'a-Ascoli's theorem, we can conclude that $c_{\varphi,n}$ converges uniformly in $\overline I$.
Using \eqref{eq:1601N} we will prove that the function
\begin{align}\label{eq:2009}
t\mapsto\sup_{\|\varphi\|_{W^{1,2}_{(\bx,\bq)}}\leq 1}\big(c_{\varphi,n}(t)-c_\varphi(t))
\end{align}
converges uniformly in $\overline I$.
First of all we extend $\widehat{\psi}^{m,n}(t,\cdot,\bq)$ and $\widehat{\psi}^{m}(t,\cdot,\bq)$ from $\Omega_{\xi(t)}$ to $S_L$
and use the compact embedding $L^{2}(S_L\times B)\hookrightarrow W^{-1,2}(S_L\times B)$
to conclude that (after taking a diagonal sequence)
\begin{align*}
\widehat\psi^{m,n}(t)\rightarrow \widehat\psi^{m}(t)\quad\text{in}\quad W^{-1,2}(\Omega_{\xi(t)}\times B)
\end{align*}
for all $t\in I_0$, where $I_0\subset I$ is a dense subset.
Consequently, we have for each such $t$
\begin{align*}
|c_{\varphi,n}(t)-c_{\varphi}(t)|&\leq \bigg|\int_{\Omega_{\xi(t)} \times B} M^m(\widehat\psi^{m,n}(t) - \widehat\psi^{m}(t))\,\Pi_n^{\bq}\Pi_n^{\bx}\varphi\dq\dx\bigg|
\\&
+\bigg|\int_{\Omega_{\xi(t)} \times B} M^m\widehat\psi^{m}(t)\,\big(\Pi_n^{\bq}\Pi_n^{\bx}\varphi-\varphi\big)\dq\dx\bigg|
\\&
\leq \,c\,\Big(\|\widehat\psi^{m,n}(t) - \widehat\psi^{m}(t)\|_{W^{-1,2}_{(\bx,\bq)}}+\| \widehat\psi^{m}(t)\|_{L^{2}_{(\bx,\bq)}}\|\Pi_n^{\bq}\Pi_n^{\bx}\varphi-\varphi\|_{L^{2}_{(\bx,\bq)}}\Big)\\
&\longrightarrow 0
\end{align*}
uniformly in $\varphi$. Combining this with \eqref{eq:1601N} and boundedness of $M^m$ easily yields \eqref{eq:2009}.
On account of \eqref{suptx3Zx} and \eqref{eq:2009}, we have
\begin{align*}
\int_I \big(c_{\widehat\psi^{m,n},n}(t)-c_{\widehat\psi^{m,n}}(t)\big)\dt
&=
\int_I \|\widehat\psi^{m,n}\|_{W^{1,2}_{(\bx,\bq)}}\frac{\big(c_{\widehat\psi^{m,n},n}(t)-c_{\widehat\psi^{m,n}}(t)\big)}{\|\widehat\psi^{m,n}\|_{W^{1,2}_{(\bx,\bq)}}}\dt
\\
&\leq\bigg(\int_I \|\widehat\psi^{m,n}\|_{W^{1,2}_{(\bx,\bq)}}^2\dt\bigg)^{\frac{1}{2}}\bigg(\int_{I}\sup_{\|\varphi\|_{W^{1,2}_{(\bx,\bq)}}\leq 1}\big(c_{\varphi,n}(t)-c_{\varphi}(t)\big)^2\dt\bigg)^{\frac{1}{2}}
\\
&\longrightarrow0
\end{align*}
as $n\rightarrow\infty$ using also \eqref{suptx3Zx}. 
By combining this with the weak convergence of $\widehat{\psi}^{m,n}$ from \eqref{psiLimReg1}, we obtain \eqref{eq:312N}.
Since $M^m$ is strictly positive and bounded, one can now easily pass to the limit in \eqref{eq:FPn} and conclude that the limit $\widehat{\psi}^m$ satisfies 
\begin{equation}
\begin{aligned}
\int_I  &\frac{\mathrm{d}}{\dt} \int_{\Omega_{\xi(t)} \times B}M^m\widehat{\psi}^m \, \varphi \dq \dx \dt = \int_{I \times \Omega_{\xi(t)} \times B}M^m\widehat{\psi}^m \partial_t \varphi 
\dq \dx \dt
\\&+
\int_{I \times \Omega_{\xi(t)} \times B}
\big(\bv \,M^m   \widehat{\psi}^m 
-
\varepsilon M^m\nabx \widehat{\psi}^m \big) \cdot \nabx \varphi  \dq \dx \dt
\\&
+\sum_{i=1}^K \int_{I \times \Omega_{\xi(t)} \times B}
 \bigg(  M (\nabx   \bv) \bq_i\Lambda_\ell(\widehat{\psi}^m)  
 -
\sum_{j=1}^K \frac{A_{ij}}{4\lambda}   M^m \nabqj  \widehat{\psi}^m \bigg) \cdot \nabqi\varphi \dq \dx \dt
\label{weakFPlimit}
\end{aligned}
\end{equation}
for all  $\varphi \in C^\infty (\overline{I}\times \R^3 \times \overline{B})$. 
Noticing the duality
\begin{align}\label{eq:dual}
\big(L^2(I;W^{1,2}(\Omega_{\xi(t)}\times B))\big)'\cong L^2(I;W^{-1,2}(\Omega_{\xi(t)} \times B)),
\end{align}
cf. \cite[page 8]{MR3575853},
we can rewrite \eqref{weakFPlimit} as
\begin{equation}
\begin{aligned}
 \Big\langle&\partial_t (M^m\widehat{\psi}^m),\varphi 
\Big\rangle_t
=
\int_{\Omega_{\xi(t)} \times B}
\big(\bv \,M^m   \widehat{\psi}^m 
-
\varepsilon M^m\nabx \widehat{\psi}^m \big) \cdot \nabx \varphi  \dq \dx
\\&
+\sum_{i=1}^K \int_{\Omega_{\xi(t)} \times B}
 \bigg(  M (\nabx   \bv) \bq_i\Lambda_\ell(\widehat{\psi}^m)  
 -
\sum_{j=1}^K \frac{A_{ij}}{4\lambda}   M^m \nabqj  \widehat{\psi}^m \bigg) \cdot \nabqi\varphi \dq \dx \\
&- \int_{\partial\Omega_{\xi(t)} \times B} (\partial_t\xi \bm{\nu})\circ\bm{\varphi}_\xi^{-1}\cdot\bm{\nu}_\xi M^m \widehat{\psi}^m \varphi\dq\dy
\label{weakFPlimit*}
\end{aligned}
\end{equation}
for $t\in I$ and all $\varphi\in W^{1,2}(\Omega_{\xi(t)}\times B)$. Here $\langle\cdot,\cdot\rangle_t$
denotes the duality pairing between $W^{-1,2}(\Omega_{\xi(t)}\times B)$ and $W^{1,2}(\Omega_{\xi(t)}\times B)$. For various estimates, we need the formula
\begin{align}\label{eq:1009}
\begin{aligned}
 \int_0^\tau\Big\langle\partial_t (M^m\widehat{\psi}^m),f'(\widehat{\psi}^m)
\Big\rangle_t\dt&=\int_{\Omega_{\xi(\tau)}\times B}M^mf(\widehat{\psi}^m(\tau))\dx\dq-\int_{\Omega_{\xi(0)}\times B}M^mf(\widehat{\psi}^m(0))\dx\dq\\
&- \int_0^{\tau}\int_{\partial\Omega_{\xi(t)} \times B}M^mf(\widehat{\psi}^m) (\partial_t\xi \bm{\nu})\circ\bm{\varphi}_\xi^{-1}\cdot\bm{\nu}_\xi \dq\dy\dt
\end{aligned}
\end{align}
for a.e. $\tau\in I$, provided $f$ is a Lipschitz-function. In order to prove \eqref{eq:1009}, we can define a temporal regularisation operator $\mathscr S_\varrho$ on $L^2(I;W^{1,2}(\Omega_{\xi(t)}\times B))$ by
\begin{align*}
\mathscr S_\varrho\phi=(\mathscr E_\xi(t)\phi)_\varrho\Big|_{I\times\Omega_\xi\times B}
\end{align*}
where $\mathscr E_\xi(t):W^{1,2}(\Omega_{\xi(t)})\rightarrow W^{1,2}(\R^3)$
is the standard extension operator (note that the boundary of $\Omega_{\xi(t)}$ is smooth)
and $(\cdot)_\varrho$ denotes a temporal regularisation which is symmetric and commutes with derivatives.
We clearly have
\begin{align}\label{reg:1}
\mathscr S_\varrho\phi\rightarrow\phi\quad \text{in}\quad L^2(I;W^{1,2}(\Omega_{\xi(t)}\times B))
\end{align}
for $\phi\in L^2(I;W^{1,2}(\Omega_{\xi(t)}\times B))$.
Thus, for $\varphi\in C^\infty(\overline I\times\R^3)$ with $\varphi(0,\cdot)=\varphi(T,\cdot)=0$, we have
\begin{align*}
 \int_I\Big\langle\partial_t \mathscr S_\varrho(M^m\widehat{\psi}^m),\varphi
\Big\rangle_t\dt
&= -\int_I\int_{\Omega_{\xi(t)}\times B} \mathscr S_\varrho(M^m\widehat{\psi}^m)\,\partial_t\varphi\dq\dx\dt\\
&- \int_I\int_{\partial\Omega_{\xi(t)} \times B}(\partial_t\xi \bm{\nu}) \circ\bm{\varphi}_\xi^{-1}\cdot\bm{\nu}_\xi \mathscr S_\varrho(M^m\widehat{\psi}^m)\,\varphi\dq\dy\dt\\
&= -\int_I\int_{\Omega_{\xi(t)}\times B} M^m\widehat{\psi}^m\,\partial_t\mathscr S_\varrho\varphi\dq\dx\dt\\
&- \int_I\int_{\partial\Omega_{\xi(t)} \times B} (\partial_t\xi \bm{\nu})\circ\bm{\varphi}_\xi^{-1}\cdot\bm{\nu}_\xi \mathscr S_\varrho(M^m\widehat{\psi}^m)\,\varphi\dq\dy\dt.
\end{align*}
Using
$\mathscr S_\varrho\varphi$ as a test-function in
\eqref{weakFPlimit}, the latter formula also make sense for $\varphi\in L^2(I;W^{1,2}(\Omega_{\xi(t)}\times B))$. Moreover, we obtain
\begin{align}\label{reg:2}
\partial_t\mathscr S_\varrho(M^m\widehat{\psi}^m)\rightharpoonup\partial_t(M^m \widehat{\psi}^m)\quad \text{in}\quad L^2(I;W^{-1,2}(\Omega_{\xi(t)} \times B))
\end{align}
as a consequence of \eqref{reg:1}. Combining \eqref{reg:1} and \eqref{reg:2} finally proves \eqref{eq:1009}. \\
With \eqref{eq:1009} at hand we can justify taking $(\widehat{\psi}^m)_{-}:=\min\{0,\widehat{\psi}^m\}$ as a test function in \eqref{weakFPlimit*}. Indeed, choosing $f(s)=s_{-}:=\min\{0,s\}$ we obtain in analogy with \eqref{suptx3Zx},
\begin{align}\label{suptx3Zxxx}
\begin{aligned}
\sup_{t\in I}\int_{\Omega_{\xi(t)} \times B} & 
\vert
 (\widehat{\psi}^m)_- \vert^2 \dq \dx
+
\int_I
\int_{\Omega_{\xi(t)} \times B}  
\vert
\nabx (\widehat{\psi}^m)_- \vert^2 \dq  \dx \dt
\\
&+
\int_I
\int_{\Omega_{\xi(t)} \times B}  
\vert
\nabq (\widehat{\psi}^m)_- \vert^2 \dq  \dx \dt
\lesssim_{m} 0
\end{aligned}
\end{align}
where the constant depends only on $A_0$, $\varepsilon>0$, $\bv$ and $\xi$. As opposed to \eqref{suptx3Zx}, the right hand side of \eqref{suptx3Zxxx} is zero because $(\widehat{\psi}^m)_{-}\big\vert_{t=0}=0$ since $T_\ell(\widehat{\psi}^{m}_0) \geq0$. We can therefore conclude that 
$(\widehat{\psi}^m)_- \equiv 0 \text{ in } I \times \Omega_{\xi(t)} \times B$ and thus
\begin{align}
\widehat{\psi}^m \geq 0 \quad \text{a.e. in} \quad I \times \Omega_{\xi(t)} \times B;
\end{align}
the minimum principle. 
\subsection{An auxiliary problem : the dumped continuity equation}
\label{subsec:aux}
For the ultimate purpose of estimating the Kramer's expansion of the extra stress tensor $\mathbb{T}(\psi)$, we also need to consider the solution $\Xi^m$ of a dumped continuity equation for the weighted average of $\widehat{\psi}^m$.
In particular, by setting $\Xi^m:=\int_BM^m \widehat{\psi}^m \dq$ where $m\in\mathbb{N}$ is fixed and choosing the test-function in \eqref{weakFPlimit} to be independent of $\bq$, we obtain
\begin{equation}
\begin{aligned}\label{etaLimit}
\int_I  &\frac{\mathrm{d}}{\dt} \int_{\Omega_{\xi(t)}}\Xi^m \, \varphi \dx \dt = \int_{I \times \Omega_{\xi(t)}}\big( \Xi^m \,\partial_t \varphi 
+
\bv \,\Xi^m \cdot \nabx \varphi -
\varepsilon \,\nabx \Xi^m \cdot \nabx \varphi \big) \dx \dt
\end{aligned}
\end{equation}
for all  $\varphi \in C^\infty (\overline{I}\times \R^3)$. We supplement \eqref{etaLimit} with the initial condition
\begin{align}
\label{ximo}
\Xi^m(0,\cdot) = \Xi^m_0:=\int_B M^m T_\ell(\widehat{\psi}^m_0)\dq \leq \int_B M\widehat{\psi}_0\dq =\Xi_0.
\end{align}
As shown in \cite[Theorem 3.1]{breitSchw2018compressible}, there is a unique solution to \eqref{etaLimit} in the class  
\begin{equation}
\begin{aligned}
\label{polyNumDenKappaEst00}
 L^\infty \big( I;L^2(\Omega_{\xi(t)})\big)
 \cap L^2\big(I;W^{1,2}(\Omega_{\xi(t)}) \big)
\end{aligned}
\end{equation}
and it satisfies the estimate
\begin{equation}
\begin{aligned}
\label{xiRegLimita}
&\big\Vert \Xi^m \big\Vert_{L^\infty(I;L^2(\Omega_{\xi(t)} ))}^2
+
\big\Vert \nabx \Xi^m \big\Vert_{L^2(I \times \Omega_{\xi(t)} )}^2
  \lesssim 
 \big\Vert \Xi_0  \big\Vert_{L^2( \Omega_{\xi(0)} )}^2
\end{aligned}
\end{equation}
with a constant depending only on the center-of-mass diffusion coefficient $\varepsilon>0$, $\xi(t)$ and $\bv$. It is also shown in \cite[Theorem 3.1]{breitSchw2018compressible} that the unique solution to \eqref{etaLimit} satisfies a renormalised formulation which reads as follows.
 For $\theta\in C^2([0,\infty);[0,\infty))$ with $\theta'(s)= 0$
 for large values of $s$ and $\theta(0)=0$, we have
\begin{align}
\label{eq:renormz}
\begin{aligned}
\int_I\frac{\dd}{\dt}&\int_{\Omega_{\xi(t)}} \theta(\Xi^m)\,\psi\dxt-\int_{I\times \Omega_{\xi(t)}}\theta(\Xi^m)\,\partial_t\psi\dxt
\\
=&-\int_{I\times \Omega_{\xi(t)}}\big(\varrho\theta'(\Xi^m)-\theta(\Xi^m)\big)\Div\bv\,\psi\dxt
+\int_{I\times \Omega_{\xi(t)}}\theta(\Xi^m) \bv\cdot\nabx\psi\dxt
\\ &-\int_{I\times\Omega_{\xi(t)}}\varepsilon\nabx \theta(\Xi^m)\cdot\nabx\psi\dxt-\int_{I\times \Omega_{\xi(t)}}\varepsilon\theta''(\Xi^m)|\nabx\Xi^m|^2\psi\dxt
\end{aligned}
\end{align}
for all $\psi\in C^\infty(\overline{I}\times\R^3)$. In fact, the assumption that the function $\theta$ in \eqref{eq:renormz} has linear growth can be replaced by a quadratic growth assumption.  Due to \eqref{xiRegLimita} and a smooth approximation argument, we can allow smooth non-negative functions with $\theta(0)=0$ which have linear growth and bounded derivatives. Choosing $\psi$ as a smooth approximation
of $\mathbb I_{(0,t)}$ and noticing that in our case $\Div\bv=0$, we obtain
\begin{align*}
\int_{\Omega_{\xi(t)} } \theta(\Xi^m(t))\dx\leq \int_{\Omega_{\xi(0)} } \theta(\Xi^m_0)\dx
\end{align*}
for all $t\in I$ provided that $\theta$ additionally satisfies $\theta''\geq0$.
We would like to choose $\theta(s)=s^p$ for large $p$, which is currently not possible. However, we can approximate it by a sequence of convex $C^2$-functions $\theta_{N}$ with
$\theta_N(s)=s^p$ for $s\leq N$  and $\theta_N\leq \,cp^2\theta$ for all $N\gg1$. This yields
\begin{align*}
\int_{\Omega_{\xi(t)}} \theta_N(\Xi^m(t))\dx\leq \int_{\Omega_{\xi(0)} } \theta(\Xi^m_0)\dx\leq\,c\,p^2\int_{\Omega_{\xi(0)} } |\Xi^m_0|^p\dx.
\end{align*}
Fatou's lemma  yields $\Xi^m(t)\in L^p(\Omega_{\xi(t)})$ and
\begin{align*}
\int_{\Omega_{\xi(t)}} |\Xi^m(t)|^p\dx\leq  \,c\,p^2\int_{\Omega_{\xi(0)}} |\Xi^m_0|^p\dx.
\end{align*}
Taking the $p$th root and passing with $p\rightarrow\infty$ implies
\begin{align*}
\|\Xi^m(t)\|_{L^\infty(\Omega(t))}\leq \, \|\Xi^m_0\|_{L^\infty(\Omega_{\xi(0)})}.
\end{align*}
for all $t\in I$.
Recalling \eqref{xiRegLimita}, we conclude that
\begin{equation}
\begin{aligned}
\label{xiRegLimit}
&\big\Vert \Xi^m \big\Vert_{L^\infty(I\times \Omega_{\xi(t)} )}^2
+
\big\Vert \nabx \Xi^m \big\Vert_{L^2(I \times \Omega_{\xi(t)} )}^2
  \lesssim 
 \big\Vert \Xi_0  \big\Vert_{L^\infty( \Omega_{\xi(0)} )}^2.
\end{aligned}
\end{equation}
Standard arguments making use of the linearity of \eqref{etaLimit} yield a limit function $\Xi$, as $m\rightarrow\infty$, also satisfying \eqref{etaLimit} and \eqref{xiRegLimit}.
We now show that the leading term in the extra stress tensor is essentially bounded too.
\begin{corollary}
\label{cor:polyNumDenKappa}
For each $i=1, \ldots, K$, let
\begin{equation}
\begin{aligned}
\label{extraStreeTensorx}
\mathbb{T}_i^\ell(M\widehat{\psi}^m) 
&:=
\int_BM(\bq) \nabqi T_\ell\big[ \widehat{\psi}^m(t, \mathbf{x},\mathbf{q}) \big]\otimes\mathbf{q}_i \dq
\end{aligned}
\end{equation}
Then for each $i=1, \ldots, K$,
\begin{equation}
\begin{aligned}
\label{elasticTensorKappaEst}
\big\vert \mathbb{T}_i^\ell(M\widehat{\psi}^m) \big\vert
\lesssim \ell 
\end{aligned}
\end{equation}
uniformly in $m\in \mathbb{N}$.
\end{corollary}
For the proof of the Corollary above, integrate by parts and use the property of the cut-off $T_\ell$, see \cite[(2.65)]{bulicek2013existence}.

\subsection{Estimates for the relative entropy and Fisher information}
Next, we show an estimate for the relative entropy defined in \eqref{relativeEntropyDefined} and the Fisher information \eqref{fisherInformation} which in our context is given by the second moment of the gradient of the square-root of our solution.
Note that a form of  energy estimate can be derived by passing to the limit in \eqref{suptx3Z}. However, this gives no information on the relative entropy which is crucial once we couple the Fokker--Planck equation with the fluid system. \\
Now, let us recall the entropy functional \eqref{relativeEntropyFunctional} and its convex regularization $\mathcal{F}_\delta(\cdot)$ introduced at the start of Section \ref{sec:energEsr}. If we use  $\mathcal{F}'_\delta(\widehat{\psi}^m)$ as a test-function in \eqref{weakFPlimit*} (justified through \eqref{eq:1009}), then similar to \eqref{fokkerEnerA}, we obtain
\begin{align}
&
\int_{\Omega_{\xi(t)}  \times B}
kM^m\, \mathcal{F}_\delta \big(\widehat{\psi}^m\big) \dq \dx
 +
 k\,\varepsilon
 \int_0^t\int_{\Omega_{\xi(\sigma)} \times B}
M^m\big\vert\nabx  \widehat{\psi}^m\big\vert^2\big(\widehat{\psi}^m+\delta \big)^{-1}
 \dq \dx \ds\nonumber
 \\&+
\frac{kA_0}{4\lambda} \,
 \int_0^t\int_{\Omega_{\xi(\sigma)}  \times B}
 M^m \big\vert\nabq  \widehat{\psi}^m \big\vert^2\big(\widehat{\psi}^m +\delta \big)^{-1}
 \dq \dx\ds\label{fokkerEnerXkappa}
 \\&
\leq
\int_{\Omega_{\xi(0)}  \times B}
kM^m\, \mathcal{F}_\delta \big(\widehat{\psi}^m_0\big) \dq \dx+
 k\,\int_0^t\int_{\Omega_{\xi(\sigma)}  \times B}
 \sum_{i=1}^K
M\,T_{\delta,\ell}\big(\widehat{\psi}^m \big) \, U'\Big(\frac{1}{2}\vert \mathbf{q}_i \vert^2 \Big) \bq_i \otimes \bq_i : \nabx\bv
 \dq \dx\ds
\nonumber
\end{align}
uniformly in $m, \ell \in \mathbb{N}$ and $\delta>0$. By passing to the limit $\delta\rightarrow 0$ (keeping in mind \eqref{conGamma}), we obtain just as in \eqref{fokkerEner},
\begin{equation}
\begin{aligned}
\label{fokkerEner1}
&\int_{\Omega_{\xi(t)} \times B}
kM^m\, \mathcal{F}\big( \widehat{\psi}^m(t,\cdot) \big) \dq \dx
 +
4k\,\varepsilon \int_0^t\int_{\Omega_{\xi(\sigma)} \times B}
M^m\Big\vert \nabx \sqrt{ \widehat{\psi}^m } \Big\vert^2 \dq \dx\ds
 \\&
 +
\frac{kA_{0}}{\lambda}\int_0^t  \int_{\Omega_{\xi(\sigma)}\times B}
  M^m\Big\vert \nabq \sqrt{ \widehat{\psi}^m } \Big\vert^2 
 \dq \dx\ds
 \leq
 \int_{\Omega_{\xi(0)} \times B}
 kM^m\, \mathcal{F}\big( \widehat{\psi}^m_0 \big)  \dq \dx
\\&+
 \int_0^t\int_{\Omega_{\xi(\sigma)} }
\mathbb{T}_\ell\big(M\widehat{\psi}^m \big) :\nabx   \bv
  \dx\ds.
\end{aligned}
\end{equation}
for all $t\in \overline{I}$ where
\begin{align}
\label{extraStreeTensorx1}
\mathbb{T}_\ell(M\widehat{\psi}^m) 
:=
k\sum_{i=1}^K  \int_BM(\bq) \nabqi T_\ell\big[ \widehat{\psi}^m(t, \mathbf{x},\mathbf{q}) \big]\otimes\mathbf{q}_i \dq
 -
  k\,\Xi^m \,\mathbb{I}
-\eth \,(\Xi^m)^2 \mathbb{I}.
\end{align}
\subsection{Second layer approximate weak solution}
\label{subsec:2nd}
If we now recall  the fact that $\bv$ satisfies \eqref{datasetReg0}, we can use \eqref{xiRegLimit} and Corollary \ref{cor:polyNumDenKappa} to conclude that the right-hand side of \eqref{fokkerEner1} is bounded uniformly in $m\in \mathbb{N}$. We show now how to identify the limit, as $m\rightarrow \infty$, of the last term in \eqref{fokkerEner1}.
To do this, we first note that there exist a constant $s_0>0$ such that
\begin{align*}
s \ln(\mathrm{e}+s)\lesssim M\mathcal{F}(s/M) \quad \text{if} \quad s\geq s_0 \quad \text{and}\quad M<c.
\end{align*}
Now, we let $\psi^m := M^m \widehat{\psi}^m$. Since $M^m<c$ uniformly in $m\in \mathbb{N}$,  for all $t\in I$, we have for some $s_0>0$,
\begin{align*}
\int_{\R^3\times B}&\chi_{\Omega_{\xi(t)}}\psi^m \ln(\mathrm{e}+ \psi^m)\dq\dx
=
\int_{\psi^m< s_0} \chi_{\Omega_{\xi(t)}} \psi^m \ln(\mathrm{e}+ \psi^m)\dq\dx
\\&+
\int_{\psi^m\geq s_0} \chi_{\Omega_{\xi(t)}} \psi^m \ln(\mathrm{e}+ \psi^m)\dq\dx
\lesssim
1
+
\int_{\R^3 \times B} \chi_{\Omega_{\xi(t)}}
M^m\, \mathcal{F}\big( \widehat{\psi}^m\big)\dq \dx
\end{align*}
with a constant depending only on the Lebesgue measure of $\Omega_{\xi(t)} \times B$ and $s_0$.
By using \eqref{fokkerEner1}, we can conclude that
\begin{align}
\label{logEPlus1}
\sup_{t\in I}\int_{\R^3 \times B}& \chi_{\Omega_{\xi(t)}} \psi^m \ln(\mathrm{e}+ \psi^m)\dq\dx
\lesssim 1
\end{align}
holds uniformly in $m\in\mathbb{N}$ and thus, $\psi^m$ is equi-integrable. In order to
obtain compactness for $\widehat{\psi}^m$, we have to localise the equation to avoid problems with the moving boundary. 
We consider  a sequence $\mathcal Q_k=J^k\times \mathcal B_{\bx}^k\times\mathcal B_{\bq}^k$ (the $\mathcal B_{\bx}^k$s and $\mathcal B_{\bq}^k$s are open balls and the $J_k$s open intervals), $k\in \mathbb{N}$ such that $\bigcup_k \mathcal Q_k =I\times\Omega_{\xi}\times B$ $J_k\Subset I$,  $\mathcal B_{\bx}^k\Subset \bigcup_{t\in J_k}\Omega_{\xi(t)}$ and $\mathcal B_{\bq}^k\Subset B$ for all $k\in\mathbb N$. Now fixed $k\in\mathbb N$ and use $\varphi\in C^\infty_c(\mathcal Q_k)$
as a test-function in \eqref{weakFPlimit} to conclude
\begin{align*}
\partial_t\widehat{\psi}^m\in L^2(J_k;W^{-1,2}(\mathcal B_{\bx}^k\times \mathcal B_{\bq}^k))
\end{align*}
uniformly in $m$ using the uniform bounds from \eqref{fokkerEner1}. Since
$M^m$ is strictly positive in $\mathcal B_{\bq}^k$ with a bound depending on $k$ but independent of $m,$ we also have
\begin{align*}
\widehat\psi^m\in L^2(J_k;W^{1,2}(\mathcal B_{\bx}^k\times \mathcal B_{\bq}^k))
\end{align*}
uniformly. Consequently, we obtain
\begin{align*}
\widehat{\psi}^m\rightarrow \widehat\psi\quad \text{in}\quad L^2(\mathcal Q_k)
\end{align*}
for a subsequence as well as
\begin{align}\label{3101x}
\chi_{\Omega_{\xi(t)}}\widehat{\psi}^m  \rightarrow \chi_{\Omega_{\xi(t)}}\widehat{\psi} \quad \text{a.e. in }\quad I\times\R^3\times B
\end{align}
by taking a diagonal sequence.
Due to uniform convergence of $M^m$ to $M$, cf. \eqref{eq:Mm}, we also obtain
\begin{align}\label{3101z}
\chi_{\Omega_{\xi(t)}}\psi^m \rightarrow  \chi_{\Omega_{\xi(t)}}\psi \quad \text{a.e. in }\quad I\times\R^3\times B.
\end{align}
By combining \eqref{3101z} with \eqref{logEPlus1},  we can conclude from Vitali's convergence theorem that
\begin{align*}
\psi^m \rightarrow  \psi \quad \text{ in }\quad L^1(I \times \Omega_{\xi(t)}\times B)
\end{align*}
and by interpolation with \eqref{xiRegLimit},
\begin{align}
\label{psimq}
\psi^m \rightarrow \psi \quad \text{ in }\quad L^q(I \times \Omega_{\xi(t)};L^1( B)) \quad \text{for all} \quad q\in[1,\infty).
\end{align}
This is enough to pass to the limit $m\rightarrow \infty$ in \eqref{fokkerEner1} so that together with results from Section \ref{subsec:aux}, we obtain
\begin{equation}
\begin{aligned}
\label{fokkerEnerXLimit}
&
\Vert \Xi(t,\cdot) \Vert^2_{L^\infty(\Omega_{\xi(t)})}
 +
k \int_{\Omega_{\xi(t)} \times B} 
M\, \mathcal{F} \big( \widehat{\psi}(t, \cdot, \cdot) \big)
\dq \dx
+
\varepsilon \int_0^t
 \int_{\Omega_{\xi(\sigma)}}\vert \nabx \Xi \vert^2 \dx\ds
\\&
+
 4k\,\varepsilon \int_0^t
 \int_{ \Omega_{\xi(\sigma)} \times B}
 M\Big\vert \nabx \sqrt{ \widehat{\psi}} \Big\vert^2
 \dq \dx\ds
+
 \frac{kA_0}{\lambda} \int_0^t
 \int_{\Omega_{\xi(\sigma)} \times B}
 M\Big\vert \nabq \sqrt{ \widehat{\psi} } \Big\vert^2
 \dq \dx\ds
 \\&
 \lesssim 
\Vert \Xi_0  \Vert^2_{L^\infty(\Omega_{\xi(0)})}
+
k \int_{\Omega_{\xi(0)}\times B} 
M\, \mathcal{F} \big( \widehat{\psi}_0 \big)
\dq \dx  
+
\int_0^t \ \int_{\Omega_{\xi(\sigma)} }
\mathbb{T}_\ell(M\widehat{\psi}) :\nabx \bv
  \dx\ds .
\end{aligned}
\end{equation}
for all $t\in I$.  To finally identify the distributional solution solved by the limit $\widehat{\psi}:=\widehat{\psi}^\ell$, we first note that for all $j=1,\ldots,K$, the following convergence 
\begin{align}
M^m\nabx \widehat{\psi}^m \rightharpoonup M\nabx \widehat{\psi} &\quad \text{in}\quad L^2\big(I \times \Omega_{\xi(t)}; L^1(B; \R^{3}) \big) \label{mNabxPsi}
\\
M^m\nabqj \widehat{\psi}^m \rightharpoonup M\nabqj \widehat{\psi} &\quad \text{in}\quad L^2\big(I \times \Omega_{\xi(t)}; L^1(B; \R^{3K}) \big) \label{mNabjPsi}
\end{align}
holds as $m\rightarrow \infty$. To avoid repetition, we refer the reader to \cite[(2.100)]{bulicek2013existence}.
Using \eqref{mNabxPsi}--\eqref{mNabjPsi}, the uniform convergence of $M^m$ as well as \eqref{psimq}, we can pass to the limit in \eqref{weakFPlimit} to obtain
\begin{equation}
\begin{aligned}
\int_I  &\frac{\mathrm{d}}{\dt} \int_{\Omega_{\xi(t)} \times B}M\widehat{\psi} \, \varphi \dq \dx \dt = \int_{I \times \Omega_{\xi(t)} \times B}M\widehat{\psi} \partial_t \varphi 
\dq \dx \dt
\\&+
\int_{I \times \Omega_{\xi(t)} \times B}
\big(\bv \,M  \widehat{\psi} 
-
\varepsilon M\nabx \widehat{\psi} \big) \cdot \nabx \varphi  \dq \dx \dt
\\&
+\sum_{i=1}^K \int_{I \times \Omega_{\xi(t)} \times B}
 \bigg(  M (\nabx   \bv) \bq_i\Lambda_\ell(\widehat{\psi})  
 -
\sum_{j=1}^K \frac{A_{ij}}{4\lambda}   M \nabqj  \widehat{\psi} \bigg) \cdot \nabqi\varphi \dq \dx \dt
\label{weakFPlimit1}
\end{aligned}
\end{equation}
for all  $\varphi \in C^\infty (\overline{I}\times \R^3 \times \overline{B})$.
\subsection{Conclusion}
By recalling Remark \ref{rem:noL}, we note that $\widehat{\psi}:=\widehat{\psi}^\ell$ where $\ell\in \mathbb{N}$ is fixed and $\psi=M\widehat{\psi}:=M\widehat{\psi}^\ell$.  Since \eqref{fokkerEnerXLimit} holds independently of $\ell\in \mathbb{N}$, just as in \eqref{logEPlus1}, we obtain
\begin{align}
\label{logEPlus2} 
\sup_{t\in I}\int_{\Omega_{\xi(t)}\times B}&\psi^\ell \ln(\mathrm{e}+ \psi^\ell)\dq\dx
\lesssim 1
\end{align}
uniformly in $\ell\in\mathbb{N}$ from which we obtain
\begin{align}
\label{psiellq}
\psi^\ell \rightarrow \psi \quad \text{ in }\quad L^q(I \times \Omega_{\xi(t)};L^1( B)) \quad \text{for all} \quad q\in[1,\infty)
\end{align}
exactly as in \eqref{psimq}. Using \eqref{psiellq} and the definition of $\mathbb{T}_\ell$ given by \eqref{extraStreeTensorx1}, we can in particular, pass to the limit $\ell\rightarrow \infty$ in the last term in \eqref{fokkerEnerXLimit}. Subsequently, we obtain
\eqref{energyEstKappaNeg1}
for all $t\in I$. Similar to \eqref{weakFPlimit1}, we also obtain \eqref{weakFPlimit0}
for all  $\varphi \in C^\infty (\overline{I}\times \R^3 \times \overline{B})$ by using \eqref{psiellq}. 
This completes the proof of Theorem \ref{thm:kappaReg}.

\section{The regularized system}
\label{sec:reg}
Let $\varrho>0$ be a fixed regularization kernel. Our ultimate goal in this section is to construct on $I \times \Omega_{\mathcal{R}^\varrho\eta^\varrho} \times B$, a weak solution to a version of our fluid-structure-kinetic system transported by a regularized material derivative
\begin{align}
\label{regMatDer}
\partial_t +\mathcal{R}^\varrho\bu^\varrho\cdot\nabx
\end{align}
similar to \cite{lengeler2014weak}. 
Here, $(\mathcal{R}^\varrho)_{\varrho>0}$ is a family of regularization operators
defined as the composition of a temporal regularisation on $I$ (which is symmetric and commutes with derivatives) and a mollification by a smooth kernel. For the later one to make sense we extend
the corresponding functions by zero to the whole space. 
In \eqref{regMatDer} $\eta^\varrho$  and $\bu^\varrho$ are corresponding approximate solutions of the shell equation \eqref{shellEq} and momentum equation \eqref{momEq} respectively. We now recall that unlike \cite{lengeler2014weak} where a linearized Koiter elastic energy is considered, we are working with the more realistic fully nonlinear energy. To be able to proceed therefore, additionally, we regularize the shell equation by the higher order linear term $\varrho\mathcal L'(\eta)$ where $\mathcal L(\eta)=\int_{\omega}|\nabla_\by^5\eta|^2\dy$ (which is to be interpreted in the sense that $\int_\omega \varrho\mathcal L'(\eta) \phi\dy =\int_\omega \varrho\nabla_\by^5\eta:\nabla_\by^5 \phi\dy $ for all $\phi\in W^{5,2}(\omega)$). Once the above construction is done, we can pass to the limit $\varrho\rightarrow 0$ to complete the proof of our main result, Theorem \ref{thm:main}. This will be done in the next section.
\\
Besides regularizing the shell equation by $\varrho\mathcal L'(\eta)$ as earlier explained, we require further, two main tools to achieve our goal for this section: a Galerkin procedure and a fixed-point argument. Unfortunately, we are required to have a fully linear system and require an additional regularization procedure in order to apply our choice of fixed-point argument. These two obstacles and their remedies leads to the following steps in achieving our goal:

The following steps give the line-by-line reasoning as to why we require \eqref{regMatDer}.
\begin{enumerate}
\item First of all, we note that one source of nonlinearity in our system comesfrom the convective term in the momentum equation \eqref{momEq}. Furthermore, the spatial domain on which a solution to our full system is defined depends on the solution itself which is problematic. We remedy these problems using the following.
\begin{enumerate}
\item We linearise the transport term of the momentum equation by replacing the material derivation $\partial_t +\bu\cdot\nabx$ with  $\partial_t +\bv\cdot\nabx$ where $\bv$ is a given velocity field. In addition, we replace the spatial domain with $\Omega_{\xi}$ where $\xi$ is a given function of time.
\item Because of the highly coupled nature of our fluid-structure-kinetic system, we are also required to replace the transport term of the Fokker--Planck equation with $\partial_t +\bv\cdot\nabx$. This makes physically sense since one can expect the transport of the solute to be affected by a change in transport of the solvent in the solution. Finally, we linearise the shell equation by replacing $K'(\eta)$ by $K'(\xi)$, which will be considered as part of the right-hand side.
\end{enumerate}
\item At this point, in theory, we should be able to apply a Gelerkin method to show that a solution to our modified-transport problem exists. Unfortunately, we are constrained by the low regularity of the boundary.  For this reason, we construct instead, a solution to our system on $\Omega_{\mathfrak{r}^\varrho\xi^\varrho}$ with material derivation  $\partial_t +\mathcal{R}^\varrho\bv^\varrho\cdot\nabx$ where $\varrho>0$ is a fixed regularization kernel. Here $\mathfrak r^\varrho$
is regularisation operator acting on the periodic functions defined on $\omega$ composed again with a a temporal regularisation on $I$.
\item Finally, for $\varrho>0$ fixed, we can use a Schauder-type fixed-point argument to the mapping 
$(\xi^\varrho,\bv^\varrho) \mapsto (\eta^\varrho, \bu^\varrho)$
where $\eta^\varrho$ and $\bu^\varrho$ are solutions to the decoupled regularised shell equation with data $(\xi^\varrho,\bv^\varrho)$ to be defined below in Section \ref{sec:decoupled}.
\end{enumerate}
Henceforth, we simply write $(\xi,\bv)$ in place of $(\xi^\varrho,\bv^\varrho)_{\varrho>0} $ (and the same for $(\eta, \bu)$) until the next section when we pass to the limit $\varrho \rightarrow 0$. However, to emphasize that our regularization is parametrized by $\varrho>0$, we maintain the notation $\mathcal{R}^\varrho$ in this chapter.\\
With the above introduction, we now make precise, our goal of this section.\\
We now seek to find $(\bu, \widehat{\psi}, \eta):=(\bu^\varrho, \widehat{\psi}^\varrho, \eta^\varrho)$ that solves the following system
\begin{align}
\divx \bu=0,
\label{contEqDCnonlinear}
\\
\partial_t^2
\eta + K'(\eta)+\varrho\mathcal L'(\eta) =g + \mathbf{F}\cdot \bm{\nu},
\label{sheeEqDCnonlinear}
\\
\partial_t \bu  + (\mathcal{R}^\varrho\bu \cdot \nabx)\mathbf{u} + \nabx   p
= 
\mu\Delta_\bx   \vu +\mathcal{R}^\varrho\divx   \mathbb{T}(M\widehat{\psi}) + \mathbf{f},
\label{momEqDCnonlinear}
\\
\partial_t (M\widehat{\psi})+ (\mathcal{R}^\varrho\bu \cdot \nabx) M\widehat{\psi}
=
\varepsilon \Delx(M\widehat{\psi})
- \sum_{i=1}^K
 \divqi  \big( M(\nabx   \mathcal{R}^\varrho\mathbf{u}) \bq_i\widehat{\psi}\big)
\nonumber
\\+
\frac{1}{4\lambda} \sum_{i=1}^K \sum_{j=1}^KA_{ij}\,  \divqi  \big( M \nabqj  \widehat{\psi} 
\big)&
\label{fokkerPlankDC}
\end{align}
in $I \times \Omega_{\mathfrak{r}^\varrho\eta(t)} \times B$ subject to the following initial and boundary conditions
\begin{align}
&\mathbf{u}(0, \cdot) = \mathbf{u}_0
&\quad \text{in }  \Omega_{\mathfrak{r}^\varrho\eta_0},
\label{initialDensitynonlinear}
\\
&\eta(0,\cdot) =\eta_0, \quad \partial_t \eta(0,\cdot)  =\eta_1
&\quad \text{in } {\omega}, \label{etaTzerononlinear}
\\
& \bu(t, \bm{\varphi}(\by)+\mathfrak{r}^\varrho\eta(t,\by) \bm{\nu}(\by)) =\partial_t \eta(t,\by) \bm{\nu}(\by)
&\quad \text{on } I \times {\omega} \label{inthesenseoftrace}
\end{align}
and  
\begin{align}
&\bigg[\frac{1}{4\lambda}  \sum_{j=1}^K A_{ij}\, M \nabqj  \widehat{\psi} - M(\nabx   \mathcal{R}^\varrho\bu ) \bq_i \widehat{\psi}
 \bigg] \cdot \bn_i =0
&\quad \text{on }I \times \Omega_{\mathfrak{r}^\varrho\eta(t)} \times \partial \overline{B}_i, \quad\text{for } i=1, \ldots, K,
\label{fokkerPlankBoundarynonlinear}
\\
& \naby \widehat{\psi} \cdot \bm{\nu}_{\mathfrak{r}^\varrho\eta} = \widehat{\psi}\big(\mathcal{R}^\varrho\bfu- (\partial_t\mathfrak{r}^\varrho\eta \bm{\nu})\circ \bm{\varphi}_{\mathfrak{r}^\varrho\eta}^{-1}\big) \cdot \bm{\nu}_{\mathfrak{r}^\varrho\eta}
&\quad \text{on }  I \times \partial \Omega_{\mathfrak{r}^\varrho\eta(t)} \times B, \label{fokkerPlankBoundaryNeumannnonlinear}
\\
&\widehat{\psi}(0, \cdot, \cdot) =\widehat{\psi}_0 \geq 0
& \quad \text{in }\Omega_{\mathfrak{r}^\varrho\eta_0} \times B,
\label{fokkerPlankIintialnonlinear}
\end{align}
where $\eta_0, \eta_1 : {\omega} \rightarrow \mathbb{R}$ and  $g: I \times {\omega} \rightarrow \mathbb{R}$ are given functions  and 
\begin{align}
\mathbf{F}(t,\by):= \Big(-2\mu\, \mathbb{D}_\by \bu(t,\by)
-
\mathcal{R}^\varrho\mathbb{T}(M\widehat{\psi})(t,\by)
+ p(t, \by) \mathbb{I}
 \Big)\bm{\nu} \circ \bm{\varphi}_{\mathfrak{r}^\varrho\eta(t)} \vert \mathrm{det} D  \bm{\varphi}_{\mathfrak{r}^\varrho\eta(t)} \vert.
\end{align}
Let us start with a precise definition of the solution. Note that as in \cite{lengeler2014weak} we have to rewrite the convective term in an uncommon way to preserve the a priori estimates.
\begin{definition}[Finite energy weak solution] \label{def:solreg}
Let $(\bff, g, \eta_0, \widehat{\psi}_0, \bu_0, \eta_1)$ be a dataset such that
\begin{equation}
\begin{aligned}
\label{datasetnonlinear}
&\bff \in L^2\big(I; L^2_{\mathrm{loc}}(\mathbb{R}^3; \mathbb{R}^{3})\big),\quad
g \in L^2\big(I; L^2(\omega)\big), \quad
\eta_0 \in W^{5,2}(\omega) \text{ with } \Vert \eta_0 \Vert_{L^\infty( \omega)} < L, 
\\
&\widehat{\psi}_0\in L^2\big( \Omega_{\eta_0}; L^2_M(B) \big), \quad
\vu_0\in L^2_{\mathrm{\divx}}(\Omega_{\eta_0}; \mathbb{R}^{3}) \text{ is such that }\mathrm{tr}_{\eta_0} \bu_0 =\eta_1 \gamma(\eta_0), \quad
\eta_1 \in L^2(\omega).
\end{aligned}
\end{equation} 
In addition, we assume
\begin{align}\label{datasetnonlinear2}
\Xi_0\in L^\infty(\Omega_{\eta_{0}}) \quad \text{where} \quad \Xi_0=\int_B M(\bq)\widehat{\psi}_0(\cdot,\bq) \dq \quad\text{in}\quad \Omega_{\eta_{0}}.
\end{align}
We call the triple
$(\bu,\widehat{\psi}, \eta )$
a \textit{finite energy weak solution} to the system \eqref{contEqDCnonlinear}--\eqref{fokkerPlankIintialnonlinear} with data\\ $(\bff, g, \eta_0, \widehat{\psi}_0, \bu_0, \eta_1)$ provided that the following holds:
\begin{itemize}
\item[(a)] the velocity $\vu$ satisfies
\begin{align*}
 \vu \in L^\infty \big(I; L^2(\Omega_{\mathfrak{r}^\varrho\eta(t)} ;\mathbb{R}^3) \big)\cap  L^2 \big(I; W^{1,2}_{\divx}(\Omega_{\mathfrak{r}^\varrho\eta(t)};\mathbb{R}^3) \big) 
 \quad \text{and} \quad \eqref{inthesenseoftrace}
\end{align*}
in the sense of traces and $\eta$ satisfies
\begin{align*}
\eta \in W^{1,\infty} \big(I; L^2(\omega) \big)\cap  L^\infty \big(I; W^{5,2}(\omega) \big) \quad \text{with} \quad \Vert \eta \Vert_{L^\infty(I \times \omega)} <L
\end{align*}
and for all  $(\phi, \bm{\varphi}) \in C^\infty(\overline{I}\times\omega) \times C^\infty(\overline{I}\times \R^3; \R^3)$ with $\phi(T,\cdot)=0$, $\bm{\varphi}(T,\cdot)=0$, $\divx \bm{\varphi}=0$ and $\mathrm{tr}_{\mathfrak{r}^\varrho\eta}\bm{\varphi}= \phi\bm{\nu}$, we have
\begin{align}
\int_I  \frac{\mathrm{d}}{\dt}\bigg(\int_{\Omega_{\mathfrak{r}^\varrho\eta(t)}}\vu  \cdot \bm{\varphi}\dx
&+\int_{\omega} \partial_t \eta \, \phi \dy
\bigg)\dt 
=\int_I  \int_{\Omega_{\mathfrak{r}^\varrho\eta(t)}}\big(  \bu\cdot \partial_t  \bm{\varphi} -\tfrac{1}{2}(\mathcal{R}^\varrho\bfu\cdot\nabla)\bu \cdot  \bm{\varphi}  +\tfrac{1}{2}(\mathcal{R}^\varrho\bfu\cdot\nabla)\bm{\varphi} \cdot\bu  \big) \dx\dt\nonumber
\\&
-  \int_I  \int_{\Omega_{\mathfrak{r}^\varrho\eta(t)}}
\big(\mu \nabla_\bx \bu :\nabx \bm{\varphi}+ \mathcal{R}^\varrho\mathbb{T}(M\widehat{\psi}) :\nabx\bm{\varphi}- \bff\cdot\bm{\varphi} \big) \dx\dt\label{distrho}\\
&+
\int_I \int_{\omega} \big(\partial_t \eta\, \partial_t\phi+\tfrac{1}{2}\partial_t \eta\, \partial_t\mathfrak r^\varrho\eta\,\phi
+g\, \phi \big)\dy\dt- \int_I \big(
\langle K'(\eta), \phi\rangle+
 \varrho\langle\mathcal L'(\eta), \phi \rangle \big)\dt
\nonumber
\end{align}
\item[(b)] the probability density function $ \widehat{\psi}$ satisfies:
\begin{align*}
&\widehat{\psi}\geq 0 \text{ a.e. in }  I \times \Omega_{\mathfrak{r}^\varrho\eta(t)}\times B,
\\ 
&
\widehat{\psi} \in L^\infty \big( I \times \Omega_{\mathfrak{r}^\varrho\eta(t)} ; L^1_M(B) \big),
\\&\mathcal{F}(\widehat{\psi} ) \in L^\infty \big( I; L^1(\Omega_{\mathfrak{r}^\varrho\eta(t)}; L^1_M(B))\big),
\\&
\sqrt{\widehat{\psi}} \in  L^2 \big( I; L^2(\Omega_{\mathfrak{r}^\varrho\eta(t)} ; W^{1,2}_M(B))\big) \cap L^2 \big( I; D^{1,2}(\Omega_{\mathfrak{r}^\varrho\eta(t)} ; L^2_M(B))\big),
\\&
 \Xi(t,\bx) = \int_BM \widehat{\psi}(t, \bx, \bq) \dq \in L^\infty\big(I \times \Omega_{\mathfrak{r}^\varrho\eta(t)}\big) \cap L^2\big(I; W^{1,2}(\mathfrak{r}^\varrho\eta(t)) \big);
\end{align*}
and for all  $\varphi \in C^\infty (\overline{I}\times \R^3 \times \overline{B} )$, we have
\[
\begin{split}
\int_I  \frac{\mathrm{d}}{\dt} \int_{\Omega_{\mathfrak{r}^\varrho\eta(t)} \times B}M \widehat{\psi} \, \varphi \dq \dx \dt &= \int_{I \times \Omega_{\mathfrak{r}^\varrho\eta(t)} \times B}\big(M \widehat{\psi} \,\partial_t \varphi 
+
M\mathcal{R}^\varrho\bu  \widehat{\psi} \cdot \nabx \varphi -
\varepsilon M\nabx \widehat{\psi} \cdot \nabx \varphi \big) \dq \dx \dt
\\&
+\sum_{i=1}^K \int_{I \times \Omega_{\mathfrak{r}^\varrho\eta(t)} \times B}
 \bigg( M(\nabx \mathcal{R}^\varrho  \bu ) \mathbf{q}_i\widehat{\psi}-
\sum_{j=1}^K \frac{A_{ij}}{4\lambda}   M \nabqj  \widehat{\psi} \bigg) \cdot \nabqi\varphi \dq \dx \dt;
\end{split}
\]
\item[(c)] for all $t\in I$, we have
\begin{equation}
\begin{aligned}
\label{energyEstnonlinear}
&\int_{\Omega_{\mathfrak{r}^\varrho\eta^\varrho(t)}}
\frac{1}{2} \vert \bu(t) \vert^2  \dx
+\mu\int_0^t
 \int_{ \Omega_{\mathfrak{r}^\varrho\eta^\varrho(\sigma)}}\vert \nabx \bu \vert^2 \dx\ds
 +
\int_{\omega}
\frac{1}{2}
 \vert \partial_t \eta (t) \vert^2  \dy
+ K(\eta(t))
\\&
+\varrho\int_{\omega}|\nabla_\by^5\eta(t)|^2\dy
+\Vert \Xi(t,\cdot) \Vert^2_{L^\infty(\Omega_{\mathfrak{r}^\varrho\eta^\varrho(t)})}
+
\varepsilon \int_0^t
 \int_{\Omega_{\mathfrak{r}^\varrho\eta^\varrho(\sigma)}}\vert \nabx \Xi \vert^2 \dx\ds
\\&
+
k\int_{\Omega_{\mathfrak{r}^\varrho\eta^\varrho(t)} \times B} 
M\, \mathcal{F} \big( \widehat{\psi}(t) \big)
\dq \dx
+
4k\,\varepsilon \int_0^t
 \int_{\Omega_{\mathfrak{r}^\varrho\eta^\varrho(\sigma)} \times B}
 M\Big\vert \nabx \sqrt{ \widehat{\psi}} \Big\vert^2
 \dq \dx\ds
  \\
&+
 \frac{kA_0}{\lambda}
 \int_0^t
 \int_{\Omega_{\mathfrak{r}^\varrho\eta^\varrho(\sigma)} \times B}
 M\Big\vert \nabq \sqrt{ \widehat{\psi}} \Big\vert^2
 \dq \dx\ds
 \\
&\lesssim \int_{\Omega_{\mathfrak{r}^\varrho\eta^\varrho_0}}
\frac{1}{2} \vert \bu_0 \vert^2  \dx
+
\int_{\omega}
\frac{1}{2}
 \vert \eta_1 \vert^2  \dy
+ K(\eta_{0})
+\varrho\int_{\omega}|\nabla_\by^5\eta_0|^2\dy
+
\Vert \Xi_0  \Vert^2_{L^\infty(\Omega_{\mathfrak{r}^\varrho\eta^\varrho_0})}
\\&
 +k
\int_{\Omega_{\mathfrak{r}^\varrho\eta_0^\varrho} \times B} 
M\, \mathcal{F} \big( \widehat{\psi}_0 \big) 
+
\frac{1}{2}
\int_0^t\int_{\Omega_{\mathfrak{r}^\varrho\eta^\varrho(\sigma)}}\mathbf{f}\cdot \bu \dx\ds 
+
\frac{1}{2}
\int_0^t\int_{\omega}g\,\partial_t \eta \dy\ds.
\end{aligned}
\end{equation}
\end{itemize}
\end{definition}
The following result establishes the existence of a solution, in the sense of Definition \ref{def:solreg} to the regularized system \eqref{contEqDCnonlinear}--\eqref{fokkerPlankIintialnonlinear}, see Remark \ref{rem:2212} for the case $\lim_{t\rightarrow T}\|\eta(t)\|_{L^\infty(\omega)}=L$.
\begin{theorem}
\label{thm:nonlinearSystem} Let $(\bff, g, \eta_0, \widehat{\psi}_0, \bu_0, \eta_1)$ be a dataset satisfying \eqref{datasetnonlinear} and \eqref{datasetnonlinear2}.
Then there is a \textit{finite energy weak solution} $(\bu,\widehat{\psi}, \eta )$ of the system \eqref{contEqDCnonlinear}--\eqref{fokkerPlankIintialnonlinear} on the interval $I=(0,T)$ in the sense of Definition \ref{def:solreg}. The number $T$ is restricted only if $\lim_{t\rightarrow T}\|\eta(t)\|_{L^\infty(\omega)}=L$.
\end{theorem}
The proof of Theorem \ref{thm:nonlinearSystem} will be given in Section \ref{subsec:fix}. It as based on a fixed point argument applied to a linearized system. The latter one will be introduced and analyzed in the next subsection.

\subsection{The linearised system}
\label{sec:decoupled}
Having constructed a solution $\widehat{\psi}$ to the Fokker--Planck equation in Section \ref{sec:SolveFokkerPlanck}, we now aim to derive a triple $(\bu, \widehat{\psi}, \eta)$ that solves the  linearized fluid-structure-kinetic system transported by a given velocity $\mathbf{v} \in  L^2_{\Div}(I \times \mathbb{R}^3 ;\mathbb{R}^3 )$ on a given shell function  $\xi \in C(\overline{I} \times  \omega )$, i.e., for a given pair $(\bv, \xi)$, we seek to find $(\bu, \widehat{\psi}, \eta)$ that solve the following system
\begin{align}
\divx \bu=0,
\label{contEqDC}
\\
\partial_t^2
\eta + K'(\xi)+\varrho\mathcal L'(\eta) =g + \mathbf{F}\cdot \bm{\nu}\label{sheeEqDC}
\\
\partial_t \bu  + (\mathcal{R}^\varrho\bv \cdot \nabx)\mathbf{u} + \nabx   p
= 
\mu\Delta_\bx   \vu +\mathcal{R}^\varrho\divx   \mathbb{T}(M\widehat{\psi}) + \mathbf{f},
\label{momEqDC}
\\
\partial_t (M\widehat{\psi}) + (\mathcal{R}^\varrho\bv \cdot \nabx) M\widehat{\psi}
=
\varepsilon \Delx (M \widehat{\psi})
- \sum_{i=1}^K
 \divqi  \big( M(\nabx   \mathcal{R}^\varrho\mathbf{v}) \bq_i\widehat{\psi} \big) 
 \nonumber
 \\+
\frac{1}{4\lambda} \sum_{i=1}^K \sum_{j=1}^K A_{ij}\,  \divqi  \big( M \nabqj  \widehat{\psi} 
\big)
\label{fokkerPlankDC}
\end{align}
in $I \times \Omega_{\mathfrak{r}^\varrho\xi(t)} \times B$ subject to the following initial and boundary conditions 
\begin{align}
&\mathbf{u}(0, \cdot) = \mathbf{u}_0
&\quad \text{in }  \Omega_{\mathfrak{r}^\varrho\xi_0},
\label{initialDensityVeloGivenT}
\\
&\eta(0,\cdot) =\eta_0, \quad \partial_t \eta(0,\cdot)  =\eta_1
&\quad \text{in }  \omega , \label{etaTzeroGivenT}
\\
& \bu(t, \bm{\varphi}(\by)+\mathfrak{r}^\varrho\xi(t,\by) \bm{\nu}(\by)) =\partial_t \eta(t,\by) \bm{\nu}(\by)
&\quad \text{on } I \times  \omega \label{uEtaNu}
\end{align}
and
\begin{align}
&\bigg[\frac{1}{4\lambda}  \sum_{j=1}^K A_{ij}\, M \nabqj  \widehat{\psi} - M (\nabx   \mathcal{R}^\varrho\bv ) \bq_i \widehat{\psi}
 \bigg] \cdot \bn_i =0
&\quad \text{on }I \times \Omega_{\mathfrak{r}^\varrho\xi(t)} \times \partial \overline{B}_i, \quad\text{for } i=1, \ldots, K,
\label{fokkerPlankBoundaryGivenFluid}
\\
& \varepsilon\naby\widehat{\psi} \cdot \bm{\nu}_{\mathfrak{r}^\varrho\xi} = \widehat{\psi}\big(\mathcal{R}^\varrho\bfv- (\partial_t\mathfrak{r}^\varrho\xi \bm{\nu})\circ \bm{\varphi}_{\mathfrak{r}^\varrho\xi}^{-1}\big) \cdot \bm{\nu}_{\mathfrak{r}^\varrho\xi}
&\quad \text{on }  I \times \partial \Omega_{\mathfrak{r}^\varrho\xi(t)} \times B, \label{fokkerPlankBoundaryNeumannGivenFluid}
\\
&\widehat{\psi}(0, \cdot, \cdot) =\widehat{\psi}_0 \geq 0
& \quad \text{in }\Omega_{\mathfrak{r}^\varrho\xi_0} \times B,
\label{fokkerPlankIintialGivenFluid}
\end{align}
where $\eta_0, \eta_1 :  \omega  \rightarrow \mathbb{R}$ and  $g: I \times  \omega  \rightarrow \mathbb{R}$ are given functions  and
\begin{align}
\mathbf{F}(t,\by):= \Big(-2\mu\, \mathbb{D}_\by \bu(t,\by)
-
\mathcal{R}^\varrho\mathbb{T}(M\widehat{\psi})(t,\by)
+ p(t, \by) \mathbb{I}
 \Big)\bm{\nu} \circ \bm{\varphi}_{\mathfrak{r}^\varrho\xi(t)} \vert \mathrm{det} D  \bm{\varphi}_{\mathfrak{r}^\varrho\xi(t)} \vert.
\end{align}
Let us start with a precise definition of the solution.
\begin{definition}[Finite energy weak solution] \label{def:solDecoupled}
Let $(\bff, g, \eta_0, \widehat{\psi}_0, \bu_0, \eta_1, \bv,\xi)$ be a dataset such that
\begin{equation}
\begin{aligned}
\label{datasetLinear}
&\bff \in L^2\big(\overline{I}; L^2_{\mathrm{loc}}(\mathbb{R}^3; \mathbb{R}^{3})\big),\quad
g \in L^2\big(\overline{I}; L^2({ \omega })\big), \quad
\eta_0 \in W^{5,2}({ \omega }) \text{ with } \Vert \eta_0 \Vert_{L^\infty( { \omega })} < L, 
\\
&\widehat{\psi}_0\in L^2\big( \Omega_{\eta_0}; L^2_M(B) \big), \quad
\vu_0\in L^2_{\mathrm{\divx}}(\Omega_{\eta_0}; \mathbb{R}^{3}) \text{ is such that }\mathrm{tr}_{\eta_0} \bu_0 =\eta_1 \gamma(\eta_0), \quad
\eta_1 \in L^2({ \omega }),
\\
&\mathbf{v} \in  L^2_{\Div}(I \times \mathbb{R}^3 ;\mathbb{R}^3 ), \quad \xi \in L^p(I;W^{4,p}(\omega))\quad \text{with} \quad \Vert \xi \Vert_{L^\infty(I \times { \omega })} <L,
\end{aligned}
\end{equation} 
where $p>2$.
In addition, we assume
\begin{align}\label{datasetLinear2}
\Xi_0\in L^\infty(\Omega_{\eta_{0}}) \quad \text{where} \quad \Xi_0=\int_B M(\bq)\widehat{\psi}_0(\cdot,\bq) \dq \quad\text{in}\quad \Omega_{\eta_{0}}.
\end{align}
We call the triple
$(\bu,\widehat{\psi} , \eta )$
a \textit{finite energy weak solution} to the system \eqref{contEqDC}--\eqref{fokkerPlankIintialGivenFluid} with data $(\bff, g, \eta_0,\widehat{\psi}_0, \bu_0, \eta_1, \bv,\xi)$ provided that the following holds:
\begin{itemize}
\item[(a)] the velocity $\vu$ satisfies
\begin{align*}
 \vu \in L^\infty \big(I; L^2(\Omega_{\mathfrak{r}^\varrho\xi(t)} ;\mathbb{R}^3) \big)\cap  L^2 \big(I; W^{1,2}_{\divx}(\Omega_{\mathfrak{r}^\varrho\xi(t)};\mathbb{R}^3) \big) 
 \quad \text{and} \quad \eqref{uEtaNu}
\end{align*}
in the sense of traces  and $\eta$ satisfies
\begin{align*}
\eta \in W^{1,\infty} \big(I; L^2({ \omega }) \big)\cap  L^\infty \big(I; W^{5,2}({ \omega }) \big) \quad \text{with} \quad \Vert \eta \Vert_{L^\infty(I \times { \omega })} <L
\end{align*}
and for all  $(\phi, \bm{\varphi}) \in C^\infty(\overline{I}\times\omega) \times C^\infty(\overline{I}\times \R^3; \R^3)$ with $\phi(T,\cdot)=0$, $\bm{\varphi}(T,\cdot)=0$, $\divx \bm{\varphi}=0$ and $\mathrm{tr}_{\mathfrak{r}^\varrho\xi}\bm{\varphi}= \phi\bm{\nu}$, we have
\begin{align}\begin{aligned}
\int_I  \frac{\mathrm{d}}{\dt}\bigg(&\int_{\Omega_{\mathfrak{r}^\varrho\xi(t)}}\vu  \cdot \bm{\varphi}\dx
+\int_{ \omega } \partial_t \eta \, \phi \dy
\bigg)\dt 
\\&=\int_I  \int_{\Omega_{\mathfrak{r}^\varrho\xi(t)}}\big(  \vu\cdot \partial_t  \bm{\varphi}  -\tfrac{1}{2}(\mathcal{R}^\varrho\bfv\cdot\nabla)\bu \cdot  \bm{\varphi}  +\tfrac{1}{2}(\mathcal{R}^\varrho\bfv\cdot\nabla)\bm{\varphi} \cdot\bu \big) \dx\dt
\\&
- \int_I  \int_{\Omega_{\mathcal{R}^\varrho\xi(t)}}\big(\mu \nabla_\bx\vu:\nabx \bm{\varphi} +\mathbb{T}(M\widehat{\psi}) :\nabx \mathcal{R}^\varrho\bm{\varphi}-\bff\cdot\bm{\varphi} \big) \dx\dt
\\&+
\int_I \int_{ \omega } \big(\partial_t \eta\, \partial_t\phi+\tfrac{1}{2}\partial_t \eta\, \partial_t\mathfrak r^\varrho\eta\,\phi +g\, \phi \big)\dy\dt-\int_I \big(\langle K'(\xi), \phi\rangle
 +\varrho\langle\mathcal L'(\eta), \phi \rangle\big)\dt.
\end{aligned}\label{eq:mometumreg}
\end{align}
\item[(b)] the probability density function $ \widehat{\psi}$ satisfies:
\begin{align*}
&\widehat{\psi}\geq 0 \text{ a.e. in }  I \times \Omega_{\mathfrak{r}^\varrho\xi(t)}\times B,
\\ 
&
\widehat{\psi} \in L^\infty \big( I ;L^1(\Omega_{\mathfrak{r}^\varrho\xi(t)} ; L^1_M(B)) \big),
\\&\mathcal{F}(\widehat{\psi} ) \in L^\infty \big( I; L^1(\Omega_{\mathfrak{r}^\varrho\xi(t)}; L^1_M(B))\big),
\\&
\sqrt{\widehat{\psi}} \in  L^2 \big( I; L^2(\Omega_{\mathfrak{r}^\varrho\xi(t)} ; W^{1,2}_M(B))\big) \cap L^2 \big( I; D^{1,2}(\Omega_{\mathfrak{r}^\varrho\xi(t)} ; L^2_M(B))\big),
\\&
 \Xi(t,\bx) = \int_BM \widehat{\psi}(t, \bx, \bq) \dq \in L^\infty\big(I \times \Omega_{\mathfrak{r}^\varrho\xi(t)}\big) \cap L^2\big(I; W^{1,2}(\Omega_{\mathfrak{r}^\varrho\xi(t)}) \big);
\end{align*}
and for all  $\varphi \in C^\infty (\overline{I}\times \R^3 \times \overline{B} )$, we have
\[
\begin{split}
\int_I  \frac{\mathrm{d}}{\dt} \int_{\Omega_{\mathfrak{r}^\varrho\xi(t)} \times B}M \widehat{\psi} \, \varphi \dq \dx \dt &= \int_{I \times \Omega_{\mathfrak{r}^\varrho\xi(t)} \times B}\big(M \widehat{\psi} \,\partial_t \varphi 
+
M\mathcal{R}^\varrho\bv  \widehat{\psi} \cdot \nabx \varphi -
\varepsilon M\nabx \widehat{\psi} \cdot \nabx \varphi \big) \dq \dx \dt
\\&
+\sum_{i=1}^K \int_{I \times \Omega_{\mathfrak{r}^\varrho\xi(t)} \times B}
 \bigg( M(\nabx \mathcal{R}^\varrho  \bv ) \mathbf{q}_i\widehat{\psi}-
\sum_{j=1}^K \frac{A_{ij}}{4\lambda}   M \nabqj  \widehat{\psi} \bigg) \cdot \nabqi\varphi \dq \dx \dt;
\end{split}
\]
\item[(c)] for all $t\in I$, we have 
\begin{align}
\int_{\Omega_{\mathfrak{r}^\varrho\xi(t)}}&
\frac{1}{2} \vert \bu(t) \vert^2  \dx
+
\int_{ \omega }
\frac{1}{2}\vert \partial_t \eta(t) \vert^2  \dy
+\varrho\int_{\omega}|\nabla_\by^5\eta(t)|^2\dy
+
\mu\int_0^t
 \int_{\Omega_{\mathfrak{r}^\varrho\xi(\sigma)}}\vert \nabx \bu \vert^2 \dx\ds
 \nonumber
\\&
 \leq 
\int_{\Omega_{\mathfrak{r}^\varrho\xi_0}}
\frac{1}{2} \vert \bu_0 \vert^2  \dx
+
\int_{ \omega }
\frac{1}{2}\vert  \eta_1 \vert^2  \dy
+\varrho\int_{\omega}|\nabla_\by^5\eta_0|^2\dy
+\int_0^t
\int_{ \omega } g\,\partial_\sigma\eta\dy\ds+\int_0^t
K'(\xi)\,\partial_\sigma\eta\ds
\nonumber
\\
&+
 \int_0^t\int_{\Omega_{\mathfrak{r}^\varrho\xi(\sigma)}}
\mathbf{f}\cdot\bu\dx\ds 
-\int_0^t\int_{\Omega_{\mathfrak{r}^\varrho\xi(\sigma)}}\mathbb T(M\widehat{\psi}):\nabx\mathcal{R}^\varrho\bfu\dxs.
\label{energyEstY}
\end{align}
\item[(d)] we have the estimate
\begin{equation}
\begin{aligned}
\label{energyEstKappa}
&
\Vert \Xi(t,\cdot) \Vert^2_{L^\infty(\Omega_{\mathfrak{r}^\varrho\xi(t)})}
 +
k \int_{\Omega_{\mathfrak{r}^\varrho\xi(t)} \times B} 
M\, \mathcal{F} \big( \widehat{\psi}(t, \cdot, \cdot) \big)
\dq \dx
+
\varepsilon \int_0^t
 \int_{\Omega_{\mathfrak{r}^\varrho\xi(\sigma)}}\vert \nabx \Xi \vert^2 \dx\ds
\\&
+
 4k\,\varepsilon \int_0^t
 \int_{ \Omega_{\mathfrak{r}^\varrho\xi(\sigma)} \times B}
 M\Big\vert \nabx \sqrt{ \widehat{\psi}} \Big\vert^2
 \dq \dx\ds
+
 \frac{kA_0}{\lambda} \int_0^t
 \int_{\Omega_{\mathfrak{r}^\varrho\xi(\sigma)} \times B}
 M\Big\vert \nabq \sqrt{ \widehat{\psi} } \Big\vert^2
 \dq \dx\ds
 \\&
 \lesssim 
\Vert \Xi_0  \Vert^2_{L^\infty(\Omega_{\mathfrak{r}^\varrho\xi(0)})}
+
k \int_{\Omega_{\mathfrak{r}^\varrho\xi(0)}\times B} 
M\, \mathcal{F} \big( \widehat{\psi}_0 \big)
\dq \dx  
+
\int_0^t \ \int_{\Omega_{\mathfrak{r}^\varrho\xi(\sigma)} }
\mathbb{T}(M\widehat{\psi}) :\nabx\mathcal{R}^\varrho \bv
  \dx\ds .
\end{aligned}
\end{equation}
for all $t\in I$.
\end{itemize}
\end{definition}
The following result establishes the existence of a solution, in the sense of Definition \ref{def:solDecoupled} to the linearized system \eqref{contEqDC}--\eqref{fokkerPlankIintialGivenFluid}, see Remark \ref{rem:2212} for the case $\lim_{t\rightarrow T}\|\eta(t)\|_{L^\infty(\omega)}=L$.
\begin{theorem}
\label{thm:linearSystem}
Let $(\bff, g, \eta_0, \widehat{\psi}_0, \bu_0, \eta_1, \bv,\xi)$ be a dataset satisfying \eqref{datasetLinear} and \eqref{datasetLinear2}.
Then there is a \textit{finite energy weak solution} $(\bu,\widehat{\psi}, \eta )$ of the system  \eqref{contEqDC}--\eqref{fokkerPlankIintialGivenFluid} on the interval $I=(0,T)$ in the sense of Definition \ref{def:solDecoupled}. The number $T$ is restricted only if $\lim_{t\rightarrow T}\|\eta(t)\|_{L^\infty(\omega)}=L$.
\end{theorem}
\begin{proof}
In the linearized problem, the Fokker--Planck equation is decoupled from the fluid-structure problem. Consequently, both can be solved independently. On the one hand, given $(\widehat{\psi}_0, \eta_1, \bv,\xi)$, let $\widehat{\psi}$ be the solution to the Fokker--Planck equation from Theorem \ref{thm:kappaReg}.  On the other hand, given $\widehat{\psi}$, the analysis of the fluid-structure problem (with right-hand side $\mathbf{f}+\mathcal{R}^\varrho\Div\mathbb T(M\widehat{\psi})$ in the fluid equation) is very similar to that of \cite[Prop. 3.27]{lengeler2014weak}
which is based on a finite-dimensional Galerkin approximation. 
The distribution $K'(\xi)$ can be represented by a function belonging to $L^2(I\times \omega)$ as $\xi\in L^p(I;W^{4,p})$ (at least if we choose $p$ large enough)
and can be put together with the forcing $g$. The higher order operator $\mathcal L'$ does not change the analysis compared to $\Delta^2_{\bx}$ (apart from the need for higher order Sobolev spaces).
 As a consequence, the problem can be treated as in \cite{lengeler2014weak}. In particular, we obtain the energy inequality \eqref{energyEstY} (in fact, the energy inequality in \cite[Prop. 3.27]{lengeler2014weak} slightly differs from \eqref{energyEstY} but a trivial modification leads to the required form).
\end{proof}

\subsection{A fixed point argument.}
\label{subsec:fix}
We now seek a fixed point of the solution map $(\bfv,\xi)\mapsto (\bfu,\eta)$ on $L^2(I,L^2_{\Div}(\mathbb R^3))\times L^p(I;W^{4,p}(\omega))$ from Theorem \ref{thm:linearSystem}. Note that we extend $\bfu$ by zero to $\R^3$. The resulting function is not weekly differentibale anymore but still divergence-free (in the sense of distributions).
Since we do not know about uniqueness of the solutions constructed in Theorem~\ref{thm:linearSystem}, we will use the following fixed point theorem for set-valued mappings.
\begin{theorem}[\cite{granis2003fixed}]\label{lem:fix}
Let $C$ be a convex subset of a normed
vector space $Z$, let $\mathfrak P (C)$ be the power set of $C$ and let $F:C\rightarrow \mathfrak P (C)$ be an upper-semicontinuous set-valued
mapping, that is, for every open set $W\subset C$ the set $\{c\in C:\,\, F(c)\in W\}\subset C$ is
open. Moreover, let $F(C)$ be contained in a compact subset of $C$, and let $F(c)$ be
non-empty, convex and compact for all $c\in C$. Then F possesses a fixed point, that
is, there exists some $c_0\in C$ with $c_0\in F(c_0)$.
\end{theorem}
We consider the interval $I_*=(0,T_*)$ with $T_*$ sufficiently small to be chosen later. We consider the set 
\begin{align*}
\mathscr D:=\Bigg\{(\xi,\bfv)\in L^p(I_*;W^{4,p}(\omega))\cap C(\overline I_*\times\omega)\times L^2(I_*,L^2_{\Div}(\R^3)):\,\,\begin{array}{c}\|\xi\|_{L^p(I_*;W^{2,p}(\omega))}\leq M_1,\\\|\xi\|_{L^\infty(I_*\times\omega)}\leq M_1,\,\xi(0)=\eta_0,\,\,\\\|\bfv\|_{L^2(I_\ast \times \R^3)}\leq M_2\end{array}\Bigg\}
\end{align*}
for $M_1=(\|\eta_0\|_{L^\infty( \omega)}+L)/2$ and $M_2>0$ to be chosen later. Note that the coupling at the boundary between velocity and shell is not contained in the definition of $\mathscr D$. This is a feature which one only gains via the fixed point and not before. Let
\begin{align*}
F:\mathscr{D}\rightarrow \mathfrak P(\mathscr{D})
\end{align*}
with 
\[
F:(\bfv, \xi)\mapsto \Big\{(\bfu,\eta):\,(\bfu,\eta)\text{ solves }\eqref{eq:mometumreg}\text{ with $(\bfv, \xi)$ and satisfies \eqref{energyEstY}}\Big\}.
\]
Note that when we say $(\bfu,\eta)$ solves \eqref{eq:mometumreg} with $(\bfv, \xi)$ this has to be understood in the sense that there is a solution $\psi$ to the Fokker-Planck equation (with data $(\bfv, \xi)$) which satisfies the corresponding energy inequality (existence of $\psi$ is guaranteed by Theorem \ref{thm:kappaReg}).
First of all, we note that the image of $F$ is nonempty due to Theorem \ref{thm:linearSystem}.
Next, we have to check that $F(\mathscr D)\subset \mathscr D$. Given $(\bfv,\xi)$, we can solve the Fokker-Planck equation with data $(\mathcal R^\varrho\bv,\mathfrak{r}^\varrho \xi, \widehat{\psi}_0)$ by means of Theorem \ref{thm:kappaReg}. 
In particular, we have 
\begin{align*}
\int_{\Omega_{\mathfrak{r}^\varrho\xi(t)} \times B} M\big\vert \widehat{\psi} \big\vert^2\dq \dx\leq\,c(M_1,M_2,\varrho)
\end{align*}
uniformly in $t$.
This yields
\begin{align*}
\int_0^{T^\ast}\int_{\Omega_{\mathfrak{r}^\varrho\xi(t)} \times B} M\big\vert \widehat{\psi} \big\vert^2\dq \dxt\leq\,T^\ast c(M_1,M_2,\varrho)\leq 1,
\end{align*}
provided we choose $T^\ast$ small enough (depending on $M_1$ and $M_2$ and $\varrho$) and hence
\begin{align}
\label{eq:0506first}
\int_0^{T^\ast}\int_{\Omega_{\mathfrak{r}^\varrho\xi(t)} \times B} |\mathcal R^\varrho\Div\mathbb T(M\widehat{\psi}) |^2\dq \dxt\leq \,c(\varrho)\int_0^{T^\ast}\int_{\Omega_{\mathfrak{r}^\varrho\xi(t)} \times B} |\mathbb T(M\widehat{\psi}) |^2\dq \dxt\leq\,c(\varrho).
\end{align}
Now we use the a priori estimate from Theorem \ref{thm:linearSystem} to conclude
\begin{equation}
 \begin{aligned}\label{eq:0506}
\sup_{I_\ast}\int_{\Omega_{\mathfrak{r}^\varrho\xi(t)}}|\bfu|^2\dx
&+\int_{I_\ast}\int_{\Omega_{\mathfrak{r}^\varrho\xi(t)}}|\nabx\bfu|^2\dxs+\sup_{I_\ast}\int_{ \omega }\big(|\partial_t\eta|^2
+
|\nabla_\by^5\eta|^2\big)\dy
\\&\leq\,c(\mathbf{f},g,K'(\xi),\bu_0,\eta_0,\eta_1,\varrho)\leq\,c(M_1)
\end{aligned}
\end{equation}
independently of $M_2$ and the size of $I_\ast$. This implies that $\eta\in C^\alpha(\overline I\times  \omega )$, by Sobolev embedding for some $\alpha>0$, with H\"older norm independent of $M_1$ and $M_2$. We obtain
\begin{align}
\label{constarint-eta}
|\eta(t,x)|\leq |\eta(t,x)-\eta_0(x)|+|\eta_0(x)|\leq c(M_1) (T^*)^\alpha+\|\eta_0\|_{L^\infty(\omega)}.
\end{align}
Therefore, we find for $T^*$ small enough (but independent of $\bfv$ and $\xi$) such that
\[
\|\eta\|_{L^\infty(I_*\times \omega)}\leq M_1.
\]
Hence we gain $F(\mathscr D)\subset \mathscr D$ for an appropriate choice of $M_2\in \mathbb R_+$.
Similarly, we obtain
\begin{align*}
\bigg(\int_{I_\ast}\|\eta\|_{W^{4,p}(\omega)}^p\dt\bigg)^{\frac{1}{p}}\leq\, (T^*)^{\frac{1}{p}}\|\eta\|_{L^\infty(I_\ast;W^{4,p}(\omega))}\leq\,c\, (T^*)^{\frac{1}{p}}\|\eta\|_{L^\infty(I_\ast;W^{5,2}(\omega))}\leq\,M_1
\end{align*}
for $T^*$ small enough using Sobolev's embedding and \eqref{eq:0506}.
\\
Next, since the problem is linear and the left-hand side of the energy inequality is convex, we find that $F(\xi,\bfv)$ is a convex and closed subset of $\mathscr D$. Also, we obtain upper-semicontinuity of the set-valued mapping. It remains to show that $F(\mathscr D)$ is relatively compact.
 Consider $(\eta_n,\bu_n)_{n\in \mathbb{N}}\subset F(\mathscr D)$. Then there exists a corresponding sequence $(\xi_n,\bv_n)_{n\in \mathbb{N}}\subset \mathscr D$, such that $(\eta_n,\bu_n)$ solves \eqref{eq:mometumreg}, with respect to $(\bv_n,\xi_n)$. The corresponding solution to the Fokker-Planck equation will be denoted by $\widehat{\psi}_n$ .
Due to the estimates above, we may choose subsequences such that
\begin{align}
\label{conv1}
\eta_n&\rightharpoonup^\ast\eta\quad\text{in}\quad L^\infty(I_*,W^{5,2}( \omega )),
\\
\label{conv2}
\partial_t\eta_n&\rightharpoonup^\ast\partial_t\eta\quad\text{in}\quad L^\infty(I_*,L^{2}( \omega ))),
\\
\label{conv3}
\eta_n&\rightarrow\eta\quad\text{in}\quad L^2(I_*,W^{2,2}( \omega )),
\\
\label{conv3orig}
\bfu_n&\rightharpoonup^{\ast,\eta}\bfu\quad\text{in}\quad L^\infty(I_*;L^2(\Omega_{\mathfrak{r}^\varrho\xi(t)}; \R^3)),
\\
\label{conv4orig} 
\nabx \bu_n&\rightharpoonup^\eta \nabx \bu\quad\text{in}\quad L^2(I_*;L^2(\Omega_{\mathfrak{r}^\varrho\xi(t)} ; \R^{3\times3})),
\\
\mathcal R^\varrho\Div\mathbb T(M\widehat{\psi}_n)&\rightharpoonup\mathcal R^\varrho\Div\mathbb T(M\widehat{\psi})\quad\text{in}\quad L^2(I_*\times \Omega_{\mathfrak{r}^\varrho\xi(t)}; \R^3),
\end{align}
where we also used linearity of $\mathbb T$.
The compactness of $\eta_n$ in $C(\overline I_*\times \omega)$ follows immediately by Arcela-Ascoli's theorem, since we know that $\eta_n$ is uniformly H\"older continuous. Compactness in $L^p(I_*;W^{4,p}(\omega))$ then follows from Aubin-Lions' Lemma (recall \eqref{conv1} and \eqref{conv2}) and interpolation.
The proof of the compactness of $\bfu_n$ is much more sophisticated.
Fortunately, we can follow \cite[Proposition 3.34]{lengeler2014weak} which is based on the compactness arguments \cite[Section 3.1]{lengeler2014weak}. The only difference is that we have a sequence of forcing terms $\mathbf{f}+\mathcal R^\varrho\Div\mathbb T(M\widehat{\psi}_n)$. But the term coming from the extra stress is bounded in $L^2$ due to \eqref{eq:0506first}  so the argument remains unchanged and we conclude.
Note also that we can extend $\bfu_n$ and $\bfu$ by zero to the whole space and gain
\begin{align}
\label{conv3origneu}\bfu_n&\rightarrow\bfu\quad\text{in}\quad L^2(I_*;L^2_{\Div}(\R^3; \R^3)).
\end{align} 
Finally, we find for all $k,l\in\mathbb N$ that $\|\partial_t^l\nabla^k\mathfrak{r}^\varrho\xi_n\|_{L^\infty(I_*\times\omega)}\leq c(\kappa,k,l)$. Hence, there is a subsequence  (not relabelled) such that
 \begin{align}\label{eq:2410}
\mathfrak{r}^\varrho\xi_n\to \mathfrak{r}^\varrho\xi \quad\text{in}\quad C^2(\overline I_*\times\omega).
\end{align}
The proof of the compactness is complete and the existence of a fixed point follows by Theorem~\ref{lem:fix}.\\
Finally, we add \eqref{energyEstY} and \eqref{energyEstKappa} (with $(\xi,\bv)=(\eta,\bu)$) to obtain the energy inequality \eqref{energyEstnonlinear} which completes the proof of Theorem \ref{thm:nonlinearSystem}. Note that with $\xi=\eta$ we can rewrite
\begin{align*}
\int_0^t \langle K'(\xi),\partial_\sigma\eta\rangle\ds=\int_0^t \langle K'(\eta),\partial_\sigma\eta\rangle\ds=\int_0^t\frac{\dd}{\dd\sigma}K(\eta)\ds
=K(\eta(t))-K(\eta_0).
\end{align*}

\section{The limit $\varrho\rightarrow0$}
\label{sec:main}
We consider a sequence of smooth functions $(\eta_0^\varrho,\eta_1^\varrho,\bfu_0^\varrho)$ such that
\begin{align}\label{KLeta0}
\eta_0^\varrho\rightarrow \eta_0\quad\text{in}\quad W^{2,2}( \omega ),\,\, \eta_1^\varrho\rightarrow \eta_1\quad\text{in}\quad L^{2}( \omega ),\,\
\sqrt{\varrho}\,\eta_0^\varrho\rightarrow 0\quad\text{in}\quad W^{5,2}( \omega ),\,\, \bfu_0^\varrho\rightarrow \bfu_0\quad\text{in}\quad L^{2}( \R^3 ),
\end{align}
as $\varrho\rightarrow0$. Here $\bfu_0$ has been extended by zero to the whole space and $\bfu_0^\varrho$ is smooth and divergence-free approximation which satisfies $\tr_{\mathfrak{r}^\varrho\eta}\bfu_0^\varrho=\eta_1^\varrho\gamma(\eta_0^\varrho)$, cf. \cite[Sec. 3.2]{lengeler2014weak}.
 For a fixed $\varrho>0$, we apply Theorem \ref{thm:nonlinearSystem}
to obtain a sequence of solutions $(\bu^\varrho,\widehat{\psi}^\varrho,\eta^\varrho)$ to the regularized problem \eqref{contEqDCnonlinear}--\eqref{fokkerPlankIintialnonlinear} with data $(\bff, g, \eta^\varrho_{0}, \widehat{\psi}_0, \bu_0, \eta_1^\varrho)$. The forthcoming effort is to pass with $\varrho\rightarrow0$ which will prove Theorem \ref{thm:main}. We split this prove into two subsections. In the first part, we establish uniform a priori estimates. In particular, we prove fractional differentiability of the shell which eventually allows to pass to the limit in the nonlinear term $K(\eta^\varrho)$. Finally, we have to pass to the limit in the nonlinear terms in the momentum equations as well as the Fokker--Planck equations. In the case of fixed domains, this is done by classical compactness tools like the classical Aubin-Lions lemma. While one can still use the equations to gain information on the time-regularity of $\bu^\varrho$ and $\widehat{\psi}^\varrho$, the analysis is significantly more involved.
\subsection{A priori estimate}
\label{sec:6.1}
Using \eqref{energyEstnonlinear} and applying Young's inequality, we obtain 
\begin{equation}
\begin{aligned}
\label{energyvarrho}
&\int_{\Omega_{\mathfrak{r}^\varrho\eta^\varrho(t)}}
 \vert \bu^\varrho(t) \vert^2  \dx
+\mu\int_0^t
 \int_{ \Omega_{\mathfrak{r}\varrho\eta^\varrho(\sigma)}}\vert \nabx \bu^\varrho \vert^2 \dx\ds
 +
\int_{\omega}
 \vert \partial_t \eta^\varrho(t) \vert^2  \dy
+\varrho\int_{\omega}|\nabla_\by^5\eta^\varrho(t)|^2\dy
\\&
+ K(\eta^\varrho(t))
+\Vert \Xi^\varrho(t,\cdot) \Vert^2_{L^\infty(\Omega_{\mathfrak{r}^\varrho\eta^\varrho(t)})}
+
\varepsilon \int_0^t
 \int_{\Omega_{\mathfrak{r}^\varrho\eta^\varrho(\sigma)}}\vert \nabx \Xi^\varrho \vert^2 \dx\ds
+k
\int_{\Omega_{\mathfrak{r}^\varrho\eta^\varrho(t)} \times B} 
M\, \mathcal{F} \big( \widehat{\psi}^\varrho (t)\big)
\dq \dx
\\
  &+
 4k\,\varepsilon \int_0^t
 \int_{\Omega_{\mathfrak{r}^\varrho\eta^\varrho(\sigma)} \times B}
 M\Big\vert \nabx \sqrt{ \widehat{\psi}^\varrho } \Big\vert^2
 \dq \dx\ds
 +
  \frac{kA_0}{\lambda}
 \int_0^t
 \int_{\Omega_{\mathfrak{r}^\varrho\eta^\varrho(\sigma)} \times B}
 M\Big\vert \nabq \sqrt{ \widehat{\psi}^\varrho } \Big\vert^2
 \dq \dx\ds
 \\
&\lesssim \int_{\Omega_{\mathfrak{r}^\varrho\eta^\varrho_0}}
\vert \bu_0 \vert^2  \dx
+
\int_{\omega}
 \vert \eta_1^\varrho \vert^2  \dy
+ K(\eta^\varrho_{0})
+\varrho\int_{\omega}|\nabla_\by^5\eta^\varrho_0|^2\dy
+
\Vert \Xi^\varrho_0  \Vert^2_{L^\infty(\Omega_{\mathfrak{r}^\varrho\eta^\varrho_0})}
\\&
 +k
\int_{\Omega_{\mathfrak{r}^\varrho\eta_0^\varrho} \times B} 
M\, \mathcal{F} \big( \widehat{\psi}_0 \big) 
+
\int_0^t\int_{\Omega_{\mathfrak{r}^\varrho\eta^\varrho(\sigma)}}|\mathbf{f}|^2\dx\ds 
\int_0^t\int_{\omega}|g|^2\dy\ds
\end{aligned}
\end{equation}
for all $t\in I$ where the right-hand side is uniformly bounded due to \eqref{KLeta0} and \eqref{datasetnonlinear}. Passing to a non-relabelled subsequence, we conclude
\begin{align}\label{varrho1}
\eta^\varrho&\rightharpoonup^\ast\eta\quad\text{in}\quad L^\infty(I,W^{2,2}(\omega)),\\
\label{varrho2}\partial_t\eta^\varrho&\rightharpoonup^\ast\partial_t\eta\quad\text{in}\quad L^\infty(I,L^{2}(\omega))),\\
\label{varrho3}\varrho\,\eta^\varrho&\rightarrow 0\quad\text{in}\quad L^\infty(I,W^{5,2}(\omega)),\\
\label{varrho4}\bfu^\varrho&\rightharpoonup^{\ast,\eta}\bfu\quad\text{in}\quad L^\infty(I;L^2(\Omega_{\mathfrak{r}^\varrho\eta^\varrho(t)}; \R^{3})),\\
\label{varrho5}\nabx\bfu^\varrho&\rightharpoonup^\eta \nabx\bfu\quad\text{in}\quad L^2(I;L^2(\Omega_{\mathfrak{r}^\varrho\eta^\varrho(t)}; \R^{3\times3})),\\
\label{varrho5x}\Xi^\varrho&\rightharpoonup^{\ast,\eta}\Xi \quad\text{in}\quad L^\infty \big( I \times \Omega_{\regkap\eta^\varrho(t)} \big) ,\\
\label{varrho5y}\Xi^\varrho&\rightharpoonup^{\eta}\Xi \quad\text{in}\quad L^2 \big( I; W^{1,2} \big(\Omega_{\mathfrak{r}^\varrho\eta^\varrho(t)} \big) \big) ,\\
\label{varrho6}\widehat{\psi}^\varrho&\rightharpoonup^{\ast,\eta}\widehat{\psi} \quad\text{in}\quad L^\infty \big( I; L^1 \big(\Omega_{\mathfrak{r}^\varrho\eta^\varrho(t)} ; L^1_M(B) \big) \big) ,\\
\label{varrho7}\sqrt{\widehat{\psi}^\varrho} &\rightharpoonup^{\eta}  \sqrt{\widehat{\psi}} \quad\text{in}\quad L^2 \big( I; L^2 \big(\Omega_{\mathfrak{r}^\varrho\eta^\varrho(t)} ; W^{1,2}_M(B) \big) \big) ,\\
\label{varrho8}\sqrt{\widehat{\psi}^\varrho} &\rightharpoonup^{\eta}  \sqrt{\widehat{\psi}} \quad\text{in} \quad L^2 \big( I; D^{1,2} \big(\Omega_{\mathfrak{r}^\varrho\eta^\varrho(t)} ; L^2_M(B) \big) \big) ,
\end{align}
for some limit functions $(\bu,\widehat{\psi},\eta)$ and $\Xi=\int_BM\widehat{\psi}\dq$.
From \eqref{varrho5x}--\eqref{varrho8} we obtain in particular
\begin{align}
\label{varrho10}\mathbb T(M\widehat{\psi}^\varrho)&\rightharpoonup\mathbb T(M\widehat{\psi})\quad\text{in}\quad L^2(I \times \Omega_{\mathfrak{r}^\varrho\eta^\varrho(t)}; \R^{3\times3}).
\end{align}
In the following we are going to prove that
\begin{align}\label{eq:0806}
\int_I\|\eta^\varrho\|^2_{W^{2+s,2}(\omega)}\dt
\end{align}
is uniformly bounded for some $s>0$. This will be achieved by using an appropriate test-function in the shell equation. As the shell equation is coupled with the fluid equation, we need a suitable test-function for the fluid equation as well.
As shown in \cite[Sec. 3]{muha2019existence}, for a given $\eta\in L^\infty(I;W^{1,2}( \omega ))$ with $\|\eta\|_{L^\infty(I\times\omega)}<\ell <L$, there are linear operators
\begin{align*}
\mathscr K_\eta:L^1( \omega )\rightarrow\mathbb R,\quad
\Testzeta:\{\xi\in L^1(I;W^{1,1}( \omega )):\,\mathscr K_\eta(\xi)=0\}\rightarrow L^1(I;W^{1,1}_{\Div}(S_{\ell} \cup \Omega)),
\end{align*}
such that the tuple $(\Testzeta(\xi-\mathscr K_\eta(\xi)),\xi-\mathscr K_\eta(\xi))$ satisfies
\begin{align*}
\Testzeta(\xi-\mathscr K_\eta(\xi))&\in L^\infty(I;L^2(\Omega_\eta))\cap L^2(I;W^{1,2}_{\Div}(\Omega_\eta)),\\
\xi-\mathscr K_\eta(\xi)&\in L^\infty(I;W^{2,2}( \omega ))\cap  W^{1,\infty}(I;L^{2}( \omega )),\\
\mathrm{tr}_\eta (\Testzeta&(\xi-\mathscr K_\eta(\xi))=\xi-\mathscr K_\eta(\xi),
\end{align*}
provided we have $\eta\in L^\infty(I;W^{2,2}( \omega ))\cap  W^{1,\infty}(I;L^{2}(\omega))$.
In particular, we have the estimates
\begin{align}\label{musc1}
\|\mathscr{F}_\eta(\xi-\mathscr K_\eta(\xi))\|_{L^q(I;W^{1,p}( S_{\ell} \cup \Omega))}\lesssim \|\xi\|_{L^q(I;W^{1,p}( \omega ))}+\|\xi\nabx \eta\|_{L^q(I;L^{p}( \omega ))},\\
\|\partial_t(\mathscr{F}_\eta(\xi-\mathscr K_\eta(\xi)))\|_{L^q(I;L^{p}( S_{\ell} \cup \Omega))}\lesssim \|\partial_t\xi\|_{L^q(I;L^{p}( \omega ))}+\|\xi\partial_t \eta\|_{L^q(I;L^{p}( \omega ))},\label{musc2}
\end{align}
for any $q\in[1,\infty]$ and $p\in (1,\infty)$ as proved in \cite[Prop. 3.3]{muha2019existence}.
We now use the test-function
\begin{align*}
(\bm{\varphi}^\varrho,\phi^\varrho)=\big(\Testetar(\Delta_{-h}^s\Delta_h^s \eta^\varrho-\mathscr K_{\eta^\varrho}(\Delta_{-h}^s\Delta_h^s \eta^\varrho)),\Delta_{-h}^s\Delta_h^s \eta^\varrho-\mathscr K_{\eta^\varrho}(\Delta_{-h}^s\Delta_h^s \eta^\varrho)\big)
\end{align*}
in equation \eqref{distrho}. Here $\Delta_s^hv(\by)=h^{-s}(v(\by+h\bm{e}_\alpha)-v(\by))$ is the fractional difference quotient in direction $\bm{e}_\alpha$ for $\alpha\in\{1,2\}$. 
We obtain  
\begin{align}
\nonumber
\int_I & \big(\langle K'(\eta^\varrho),\phi^\varrho\rangle+\varrho\langle \mathcal L'(\eta^\varrho),\phi^\varrho\rangle\big)\dy\dt\\
&=\int_I  \int_{\Omega_{\mathfrak{r}^\varrho\eta^\varrho(t)}}\big(  \bu^\varrho \cdot \partial_t  \bm{\varphi} ^\varrho+ \tfrac{1}{2}(\mathcal{R}^\varrho\bu^\varrho \cdot\nabla) \bu^\varrho\cdot \bm{\varphi}^\varrho -\tfrac{1}{2}(\mathcal{R}^\varrho\bu^\varrho \cdot\nabx) \bm{\varphi}^\varrho \cdot \bu^\varrho
 -\mu \nabx\bu^\varrho:\nabx \bm{\varphi}^\varrho  
+\bff\cdot\bm{\varphi}^\varrho \big) \dx\dt\nonumber
\\
&-\int_I  \frac{\mathrm{d}}{\dt}\bigg(\int_{\Omega_{\mathfrak{r}^\varrho\eta^\varrho(t)}}\vu^\varrho  \cdot \bm{\varphi}^\varrho\dx
+\int_{ \omega } \partial_t \eta^\varrho \, \phi^\varrho \dy
\bigg)\dt 
+
\int_I \int_{ \omega } \big(\partial_t \eta^\varrho\, \partial_t\phi^\varrho +g\, \phi^\varrho \big)\dy\dt\label{mom:reg}
\\
&+\int_I \int_{\omega} \tfrac{1}{2}\partial_t \eta^\varrho\, \partial_t\mathfrak r^\varrho\eta^\varrho\,\phi^\varrho\dxt
- \int_I  \int_{\Omega_{\mathfrak{r}^\varrho\eta^\varrho(t)}}\mathbb{T}(M\widehat{\psi}^\varrho) :\nabx \mathcal{R}^\varrho\bm{\varphi}^\varrho \dx\dt\nonumber\\
&=:(I)^\varrho+(II)^\varrho+(III)^\varrho+(IV)^\varrho+(V)^\varrho.\nonumber
\end{align}
Since $\eta^\varrho\in L^\infty(I,W^{2,2}(\omega))$ uniformly, due to \cite[Lemma 4.5]{muha2019existence}, we have
\begin{align*} 
\int_I\|\Delta_h^s \nabla^2\eta^\varrho\|_{L^2( \omega )}^2\dt\lesssim 1+ \int_I \langle K'(\eta^\varrho), \phi^\varrho\rangle\dt\lesssim 1+ \int_I \big(\langle K'(\eta^\varrho)\phi^\rho\rangle+\varrho\langle\mathcal L'(\eta^\varrho), \phi^\varrho\rangle\big)\dt
\end{align*} 
for every $h>0$ and $s\in (0,\frac{1}{2})$
so that our task consists of establishing uniform estimates for the 
terms $(I)^\varrho,\dots,(V)^\varrho$. As in \cite[Sec. 4.2]{muha2019existence}, we obtain the required bounds for $(I)^\varrho,(II)^\varrho$ and $(III)^\varrho$. A key ingredient in these estimates is the observation that 
\begin{align}\label{eq:0806b}
\partial_t\eta^\varrho\in  L^2(I,W^{s,2}( \omega )))
\end{align}
for all $s\in (0,\frac{1}{2})$ by the use of \eqref{varrho5} and the trace theorem, cf. Lemma \ref{lem:2.28}. In fact, we can transfer the regularity from
$\bfu^\varrho$ to $\partial_t\eta^\varrho$ by means of the boundary condition $\tr_{\mathcal R^\varrho\eta^\varrho}\bfu^\varrho=\partial_t\eta^\varrho\gamma(\eta^\varrho)$. Using
\eqref{varrho1} and \eqref{varrho2},
we obtain
\begin{align*}
(IV)^\varrho&\leq \|\partial_t\eta^\varrho\|_{L^\infty(I;L^2(\omega))}\|\partial_t\mathfrak r^\varrho\eta^\varrho\|_{L^\infty(I;L^2(\omega))}\|\phi^\varrho\|_{L^\infty(I;L^\infty(\omega))}\\
&\lesssim \|\partial_t\eta^\varrho\|_{L^\infty(I;L^2(\omega))}^2\|\Delta_{-h}^s\Delta_h^s \eta^\varrho\|_{L^\infty(I;W^{1,p}(\omega))}\\
&\lesssim\|\eta^\varrho\|_{L^\infty(I;W^{1+2s,p}(\omega))}\lesssim\|\eta^\varrho\|_{L^\infty(I;W^{2,2}(\omega))}\leq\,c
\end{align*}
for any $s<\frac{1}{2}$.
Here, we have chosen $p\in(2,\frac{1}{s}]$ such that the Sobolev embeddings $W^{1,p}(\omega)\hookrightarrow L^\infty(\omega)$ and $ W^{2,2}(\omega)\hookrightarrow W^{1+2s,p}(\omega)$ holds.
We are left with the term
$(V)^\varrho$ for the mesoscopic coupling. We obtain 
\begin{align*}
(V)^\varrho&\leq \|\mathbb{T}(M\widehat{\psi}^\varrho)\|_{L^2(I\times \Omega_{\mathfrak{r}^\varrho\eta^\varrho(t)})}\|\nabx \mathcal{R}^\varrho\bm{\varphi}^\varrho\|_{L^{2}(I\times \Omega_{\mathfrak{r}^\varrho\eta^\varrho(t)})}
\\&
\lesssim \|\Delta_{-h}^s\Delta_h^s \eta^\varrho\|_{L^{2}(I;W^{1,2}( \omega ))}+\| (\Delta_{-h}^s\Delta_h^s \eta^\varrho)\nabx\eta^\varrho\|_{L^{2}(I;L^{2}( \omega ))}
\\
&\lesssim \|\eta^\varrho\|_{L^{\infty}(I;W^{1+2s,2}( \omega ))}+\|\eta^\varrho\|_{L^\infty(I;W^{2s,4}( \omega )}\|\eta^\varrho\|_{L^{\infty}(I;W^{1,4}( \omega ))}
\end{align*}
using \eqref{varrho10}  and \eqref{musc1}.  Using \eqref{varrho1}, the right-hand side is clearly bounded for all $s<\frac{1}{2}$.
Combining all the previous estimates and passing with $h\rightarrow0$ yields \eqref{eq:0806}.

\subsection{Compactness}
\label{subsec:comp}
To show compactness of the velocity, we are going to apply the method from \cite{lengeler2014weak}.
We aim to prove
\begin{align}\label{eq:convrhou}\begin{aligned}
\int_{I}\int_{\Omega_{\mathfrak{r}^\varrho\eta^\varrho(t)}}&|\bfu^\varrho|^2\dxt+\int_{I}\int_\omega|\partial_t\eta^\varrho|^2\dy\dt\\
&\longrightarrow \int_{I}\int_{\Omega_{\eta(t)}}|\bfu|^2\dxt+\int_{I}\int_\omega|\partial_t\eta|^2\dy\dt,
\end{aligned}
\end{align}
which will be a consequence of 
\begin{align}\label{eq:312}
\begin{aligned}
\int_{I}\int_{\Omega_{\mathfrak{r}^\varrho\eta^\varrho(t)}}&\bfu_n\cdot\mathscr F_{\mathfrak{r}^\varrho\eta^\varrho}(\partial_t\eta^\varrho)\dxt+\int_{I}\int_\omega|\partial_t\eta^\varrho|^2\dy\dt\\
&\longrightarrow \int_{I}\int_{\Omega_{\eta(t)}}\bfu\cdot\mathscr F_{\eta}(\partial_t\eta)\dxt+\int_{I}\int_\omega|\partial_t\eta|^2\dy\dt
\end{aligned}
\end{align}
and 
\begin{align}\label{eq:313}
\begin{aligned}
\int_{I}\int_{\Omega_{\mathfrak{r}^\varrho\eta^\varrho(t)}}&\bfu^\varrho\cdot(\bfu^\varrho-\mathscr F_{\mathfrak{r}^\varrho\eta^\varrho}(\partial_t\eta^\varrho))\dxt
\longrightarrow \int_{I}\int_{\Omega_{\eta(t)}}\bfu\cdot(\bfu-\mathscr F_\eta(\partial_t\eta_n))\dxt.
\end{aligned}
\end{align}
In order to prove \eqref{eq:312}, we take the test-function $(b,\mathscr F_{\mathcal R^\varrho\eta^\varrho}(b-\mathscr K_{\mathfrak{r}^\varrho\eta^\varrho}(b)))$ with $b\in W^{5,2}(\omega)$ and $\|b\|_{W^{5,2}(\omega)}\leq 1$. 
It  satisfies for all $q<\infty$, 
\begin{align}\label{eq:1601a}
\begin{aligned}
\|\partial_t(\mathscr F_{\mathfrak{r}^\varrho\eta^\varrho}(b-\mathscr K_{\mathfrak{r}^\varrho\eta^\varrho}(b))\|_{L^\infty(I;L^{2}(\Omega_{\mathfrak{r}^\varrho\eta^\varrho(t)}))}&\leq\,c,\\
\|\mathscr F_{\mathfrak{r}^\varrho\eta^\varrho}(b-\mathscr K_{\mathfrak{r}^\varrho\eta^\varrho}(b))\|_{L^\infty(I;W^{1,q}(\Omega_{\mathfrak{r}^\varrho\eta^\varrho(t)}))}&\leq \,c,
\end{aligned}
\end{align}
due to \eqref{musc1} and \eqref{musc2} in combination with \eqref{varrho1} and \eqref{varrho2}. We consider
the functions
\begin{align*}
c_{b,\varrho}(t)&=\int_{\Omega_{\mathfrak{r}^\varrho\eta^\varrho(t)}}\bfu^\varrho\cdot\mathscr F_{\mathfrak{r}^\varrho\eta^\varrho}(b-\mathscr K_{\mathfrak{r}^\varrho\eta^\varrho}(b))\dx+\int_\omega\partial_t\eta^\varrho\,b\dy,\\
c_{b}(t)&=\int_{\Omega_{\eta(t)}}\bfu\cdot\mathscr F_{\eta}(b-\mathscr K_{\eta}(b))\dx+\int_\omega\partial_t\eta\,b\dy,
\end{align*}
and aim to show
\begin{align}\label{eq:1601}
\big\|c_{b,\varrho}\big\|_{C^{0,1/\chi'}(\overline I)}\lesssim 1
\end{align}
for some $\chi>1$ uniformly in $b$ and $\varrho$. This follows by inserting $\mathbb I_{(0,t)}\mathscr F_{\mathfrak{r}^\varrho\eta^\varrho}(b-\mathscr K_{\mathfrak{r}^\varrho\eta^\varrho}(b))$
as a test-function
in \eqref{distrho}
 provided all terms in
in \eqref{distrho} have integrability $\chi$. This is indeed the case as a consequence of
\eqref{varrho1}--\eqref{varrho10}.
Applying Arcela-Ascoli's theorem we can assume that $c_{b,\varrho}$ converges (for a fixed $b$) uniformly in $\overline I$, at least for a subsequence.
Together with \eqref{eq:1601},
this implies uniform convergence of the function
\begin{align*}
t\mapsto\sup_{\|b\|_{W^{5,2}(\omega)}\leq 1}\big(c_{b,\varrho_n}(t)-c_{\varphi}(t)),
\end{align*}
for a subsequence $(\varrho_n)_{n\in\mathbb{N}}$,
cf. \cite[pages 229, 230]{lengeler2014weak} for more details.
 In particular, we obtain
\begin{align}\label{eq:2101}
\int_{I}\sup_{\|b\|_{W^{5,2}(\omega)}\leq 1}\big(c_{b,\varrho_n}(t)-c_b(t)\big)\dt\rightarrow0
\end{align}
as $n\rightarrow\infty$.
By interpolation, we have
\begin{align*}
\int_{I}\sup_{\|b\|_{L^{2}(\omega)}\leq 1}\big(c_{b,\varrho_n}(t)-c_b(t)\big)\dt&\leq\varepsilon\int_I\big(\|\bfu^{\varrho_n}\|^2_{W^{1,2}(\mathfrak{r}^{\varrho_n}\eta^{\varrho_n})}+\|\bfu\|_{W^{1,2}(\mathfrak{r}^{\varrho_n}\eta^{\varrho_n})}\big)\\
&+\varepsilon\int_I\big(\|\partial_t\eta^{\varrho_n}\|_{W^{s,2}(\omega)}+\|\partial_t\eta\|_{W^{s,2}(\omega)}\big)\\
&+c_\varepsilon \int_{I}\sup_{\|b\|_{W^{5,2}(\omega)}\leq 1}\big(c_{b,\varrho_n}(t)-c_b(t)\big)\dt
\end{align*}
for all $\varepsilon>0$ and all $s\in(0,1/2)$. Consequently, we obtain
\begin{align}\label{eq:2101B}
\int_{I}\sup_{\|b\|_{L^{2}(\omega)}\leq 1}\big(c_{b,\varrho_n}(t)-c_b(t)\big)\dt\rightarrow0
\end{align}
as $n\rightarrow\infty$
on account of \eqref{varrho5}, \eqref{eq:0806} and \eqref{eq:2101}. Finally, by \eqref{varrho2} and  \eqref{eq:2101B} we have
\begin{align*}
\int_I \big(c_{\partial_t\eta^{\varrho_n},{\varrho_n}}(t)-c_{\partial_t\eta^{\varrho_n}}(t)\big)\dt
&\rightarrow0
\end{align*}
as $n\rightarrow \infty$
which implies \eqref{eq:312}. For the convergence in \eqref{eq:313},
one uses the fact that the function $\bfu^{\varrho_n}-\mathscr F_{\mathfrak{r}^{\varrho_n}\eta^{\varrho_n}}(\partial_t\eta^{\varrho_n})$ is zero on $\partial\Omega_{\mathfrak{r}^{\varrho_n}\eta^{\varrho_n}}$ by construction. Hence \eqref{eq:313} is not affected by the shell equation
and can be proved as in \cite[pages 213, 232]{lengeler2014weak}.
We conclude that \eqref{eq:convrhou} holds and obtain
\begin{align*}
\eta^{\varrho_n}&\rightarrow\eta\quad\text{in}\quad L^2(I,W^{2,2}( \omega )),\\
\bu^{\varrho_n}&\rightarrow^{\eta}\bfu\quad\text{in}\quad L^2(I;L^2(\Omega_{\mathfrak{r}^{\varrho_n}\eta^{\varrho_n}(t)};\R^3)),
\end{align*}
 by
convexity of the $L^2$-norm.
This is enough to pass to the limit in the nonlinearities of the fluid-structure system. Note that
with the above and \eqref{eq:0806} we also obtain compactness in $L^2(I,W^{2,p}( \omega ))$ for some $p>2$ and recall from Section \ref{sec:koiterEnergy} that the energy $K$ is continuous on $W^{2,p}( \omega )$.\\
In order to complete the proof of Theorem \ref{thm:main}, we have to pass to the limit 
in the term $M\nabx\bu^{\varrho_n}\widehat{\psi}^{\varrho_n}$ appearing in the Fokker--Planck equation.
For this purpose we have to prove compactness of $\widehat{\psi}^{\varrho_n}$.
For this, we first note that \eqref{varrho1} and \eqref{varrho2} imply
\begin{align}
\label{calpha}
\eta^{\varrho_n}\rightarrow\eta\quad\text{in}\quad C^\alpha(\overline{I}\times \omega)
\end{align}
as $n\rightarrow \infty$ for some $\alpha\in(0,1)$. Therefore, by using
\begin{align*}
\big\Vert \mathfrak{r}^{\varrho_n}\eta^{\varrho_n} -\eta
 \big\Vert_{ C^\alpha(\overline{I}\times \omega)} 
 \leq
 \big\Vert \mathfrak{r}^{\varrho_n}(\eta^{\varrho_n} -\eta )\big\Vert_{ C^\alpha(\overline{I}\times \omega)} 
 + \big\Vert \mathfrak{r}^{\varrho_n}\eta -\eta \big\Vert_{ C^\alpha(\overline{I}\times \omega)},
\end{align*}
we also get that
\begin{align}
\label{calpha1}
\mathfrak{r}^{\varrho_n}\eta^{\varrho_n}\rightarrow\eta\quad\text{in}\quad C^\alpha(\overline{I}\times \omega)
\end{align}
as $n\rightarrow \infty$. Next, similarly to \eqref{logEPlus1}, we obtain from \eqref{energyvarrho} the bound
\begin{align}
\label{logEPlus3}
\sup_{t\in I}\int_{\R^3 \times B}& \chi_{\Omega_{\mathfrak{r}^{\varrho_n}\eta^{\varrho_n}(t)}} M\widehat{\psi}^{\varrho_n} \ln(\mathrm{e}+ \widehat{\psi}^{\varrho_n})\dq\dx
\lesssim 1
\end{align}
which holds uniformly in $n\in \mathbb{N}$ and thus, $\widehat{\psi}^{\varrho_n}$ is equi-integrable. 
Similarly to \eqref{3101x}, we obtain
\begin{align}\label{3101z1}
\chi_{\Omega_{\mathfrak{r}^{\varrho_n}\eta^{\varrho_n}(t)}}\widehat{\psi}^{\varrho_n} \rightarrow  \chi_{\Omega_{\eta(t)}}\psi \quad \text{a.e. in }\quad I\times\R^3\times B.
\end{align}
for a possible subsequence.
If we now combine \eqref{3101z1} with \eqref{logEPlus3}, then we can conclude from Vitali's convergence theorem that
\begin{align*}
\widehat{\psi}^{\varrho_n} \rightarrow^\eta  \psi \quad \text{ in }\quad L^1\big(I \times \Omega_{\mathfrak{r}^{\varrho_n}\eta^{\varrho_n}(t)};L^1_M( B)\big)
\end{align*}
and because of \eqref{varrho6}, we obtain by interpolation,
\begin{align}
\label{psimq1}
\widehat{\psi}^{\varrho_n} \rightarrow^\eta  \psi  \quad \text{ in }\quad L^q\big(I \times \Omega_{\mathfrak{r}{\varrho_n}\eta^{\varrho_n}(t)} ;L^1_M( B)\big)\quad \text{for all} \quad q\in[1,\infty)
\end{align}
and hence 
\begin{align*}
\nabx\bu^{\varrho_n}\widehat{\psi}^{\varrho_n} \rightharpoonup^\eta \nabx \bu\,\widehat{\psi}\quad\text{in}\quad L^1\big(I \times \Omega_{\mathfrak{r}^{\varrho_n}\eta^{\varrho_n}(t)};L^1_M( B;\R^{3\times 3})\big)
\end{align*}
due to \eqref{varrho5}. The proof of Theorem \ref{thm:main} is complete.


\begin{thebibliography}{10}
\providecommand{\url}[1]{{#1}}
\providecommand{\urlprefix}{URL }
\expandafter\ifx\csname urlstyle\endcsname\relax
  \providecommand{\doi}[1]{DOI~\discretionary{}{}{}#1}\else
  \providecommand{\doi}{DOI~\discretionary{}{}{}\begingroup
  \urlstyle{rm}\Url}\fi

\bibitem{barrett2005existence}
Barrett, J.W., Schwab, C., S\"{u}li, E.: Existence of global weak solutions for
  some polymeric flow models.
\newblock Math. Models Methods Appl. Sci. {15}(6), 939--983 (2005).


\bibitem{barrett2007existence}
Barrett, J.W., S\"{u}li, E.: Existence of global weak solutions to some
  regularized kinetic models for dilute polymers.
\newblock Multiscale Model. Simul. {6}(2), 506--546 (2007).



\bibitem{barrett2008existence}
Barrett, J.W., S\"{u}li, E.: Existence of global weak solutions to dumbbell
  models for dilute polymers with microscopic cut-off.
\newblock Math. Models Methods Appl. Sci. {18}(6), 935--971 (2008).


\bibitem{barrettSuli2011existence}
Barrett, J.W., S\"{u}li, E.: Existence and equilibration of global weak
  solutions to kinetic models for dilute polymers {I}: {F}initely extensible
  nonlinear bead-spring chains.
\newblock Math. Models Methods Appl. Sci. {21}(6), 1211--1289 (2011).


\bibitem{barrett2012existenceMMMAS}
Barrett, J.W., S\"{u}li, E.: Existence and equilibration of global weak
  solutions to kinetic models for dilute polymers {II}: {H}ookean-type models.
\newblock Math. Models Methods Appl. Sci. {22}(5), 1150,024, 84 (2012).


\bibitem{barrett2012existenceJDE}
Barrett, J.W., S\"{u}li, E.: Existence of global weak solutions to finitely
  extensible nonlinear bead-spring chain models for dilute polymers with
  variable density and viscosity.
\newblock J. Differential Equations {253}(12), 3610--3677 (2012).

\bibitem{barrett2012finite}
Barrett, J.W., S\"{u}li, E.: Finite element approximation of finitely
  extensible nonlinear elastic dumbbell models for dilute polymers.
\newblock ESAIM Math. Model. Numer. Anal. {46}(4), 949--978 (2012).

\bibitem{bird1987dynamics}
Bird, R.B., Curtiss, C.F., Armstrong, R.C., Hassager, O.: Dynamics of polymeric
  liquids. vol. 2: Kinetic theory.
\newblock Wiley, New York  (1987).

\bibitem{fluid2014bodnar}
Bodn\'{a}r, T., Galdi, G.P., Ne\v{c}asov\'{a}, \v{S}
.: Fluid-structure
  interaction and biomedical applications.
\newblock Advances in Mathematical Fluid Mechanics. Birkh\"{a}user/Springer,
  Basel (2014).
  

\bibitem{boulakia2019well}
Boulakia, M., Guerrero, S., Takahashi, T.:
Well-posedness for the coupling between a viscous incompressible fluid and an elastic structure.
\newblock  Nonlinearity. {32}(10), 3548--3592 (2019).



\bibitem{breit2019local}
Breit, D., Mensah, P.R.: Local well-posedness of the compressible FENE dumbbell model of Warner-type.
\newblock arXiv preprint arXiv:1911.02465v2 (2018).

\bibitem{breitSchw2018compressible}
Breit, D., Schwarzacher, S.: Compressible fluids interacting with a
  linear-elastic shell.
\newblock Arch. Ration. Mech. Anal. {228}(2), 495--562 (2018).


\bibitem{bulicek2013existence}
Bul\'{\i}\v{c}ek, M., M\'{a}lek, J., S\"{u}li, E.: Existence of global weak
  solutions to implicitly constituted kinetic models of incompressible
  homogeneous dilute polymers.
\newblock Comm. Partial Differential Equations {38}(5), 882--924 (2013).


\bibitem{chakrabarti2002theory}
Chakrabarti, S.K.: The theory and practice of hydrodynamics and vibration,
  vol.~20.
\newblock World scientific (2002).


\bibitem{chambolle2005existence}
Chambolle, A., Desjardins, B., Esteban, M.J., Grandmont, C.: Existence of weak
  solutions for the unsteady interaction of a viscous fluid with an elastic
  plate.
\newblock J. Math. Fluid Mech. {7}(3), 368--404 (2005).


\bibitem{ciarlet2005Introduction}
Ciarlet, P.G.: An introduction to differential geometry with applications to
  elasticity.
\newblock Springer, Dordrecht (2005).
\newblock Reprinted from J. Elasticity {{{7}}8/79} (2005).


\bibitem{ciarlet2001justification}
Ciarlet, P.G., Roquefort, A.: Justification of a two-dimensional nonlinear
  shell model of {K}oiter's type.
\newblock Chinese Ann. Math. Ser. B {22}(2), 129--144 (2001).


\bibitem{constantin2005nonlinear}
Constantin, P.: Nonlinear {F}okker-{P}lanck {N}avier-{S}tokes systems.
\newblock Commun. Math. Sci. {3}(4), 531--544 (2005).


\bibitem{constantin2007regularity}
Constantin, P., Fefferman, C., Titi, E.S., Zarnescu, A.: Regularity of coupled
  two-dimensional nonlinear {F}okker-{P}lanck and {N}avier-{S}tokes systems.
\newblock Comm. Math. Phys. {270}(3), 789--811 (2007).


\bibitem{dowell2015modern}
Dowell, E.H.: A modern course in aeroelasticity, \emph{Solid Mechanics and its
  Applications}, vol. 217, enlarged edn.
\newblock Springer, Cham (2015).


\bibitem{e2004well}
E, W., Li, T., Zhang, P.: Well-posedness for the dumbbell model of polymeric
  fluids.
\newblock Comm. Math. Phys. {248}(2), 409--427 (2004).


\bibitem{el1989existence}
El-Kareh, A.W., Leal, L.G.: Existence of solutions for all deborah numbers for
  a non-newtonian model modified to include diffusion.
\newblock Journal of Non-Newtonian Fluid Mechanics {33}(3), 257--287
  (1989).

\bibitem{granis2003fixed}
Granas, A., Dugundji, J.: Fixed point theory.
\newblock Springer Monographs in Mathematics. Springer-Verlag, New York (2003).


\bibitem{grandmont2008existence}
Grandmont, C.: Existence of weak solutions for the unsteady interaction of a
  viscous fluid with an elastic plate.
\newblock SIAM J. Math. Anal. {40}(2), 716--737 (2008).


\bibitem{gwiazda2018existence}
Gwiazda, P., Luk\'{a}\v{c}ov\'{a}-Medvidov\'{a}, M., Mizerov\'{a}, H.,  \'{S}wierczewska-Gwiazda, A.: Existence of global weak solutions to the kinetic {P}eterlin model.
\newblock Nonlinear Anal. Real World Appl.  {44}, 465--478 (2018).


\bibitem{hundertmark2016existence}
Hundertmark-Zau\v{s}kov\'{a}, A., Luk\'{a}\v{c}ov\'{a}-Medvi\v{d}ov\'{a}, M.,
  Ne\v{c}asov\'{a}, \v{S}.: On the existence of weak solution to the coupled
  fluid-structure interaction problem for non-{N}ewtonian shear-dependent
  fluid.
\newblock J. Math. Soc. Japan {68}(1), 193--243 (2016).


\bibitem{ignatova2017small}
Ignatova, M., Kukavica, I., Lasiecka, I., Tuffaha, A.:
Small data global existence for a fluid-structure model.
\newblock Nonlinearity {30}(2), 848--898 (2017).


\bibitem{jourdain2004existence}
Jourdain, B., Leli\`evre, T., Le~Bris, C.: Existence of solution for a
  micro-macro model of polymeric fluid: the {FENE} model.
\newblock J. Funct. Anal. {209}(1), 162--193 (2004).


\bibitem{klainerman1981singular}
Klainerman, S., Majda, A.: Singular limits of quasilinear hyperbolic systems
  with large parameters and the incompressible limit of compressible fluids.
\newblock Comm. Pure Appl. Math. {34}(4), 481--524 (1981).


\bibitem{lee2013Introduction}
Lee, J.M.: Introduction to smooth manifolds, \emph{Graduate Texts in
  Mathematics}, vol. 218, second edn.
\newblock Springer, New York (2013).

\bibitem{lengeler2011globale}
Lengeler, D.: Globale existenz f{\"u}r die interaktion eines
  navier-stokes-fluids mit einer linear elastischen schale.
\newblock Ph.D. thesis, Universit{\"a}t (2011).

\bibitem{lengeler2014weak}
Lengeler, D., R\ocirc{u}\v{z}i\v{c}ka, M.: Weak solutions for an incompressible
  {N}ewtonian fluid interacting with a {K}oiter type shell.
\newblock Arch. Ration. Mech. Anal.  {211}(1), 205--255 (2014).


\bibitem{li2004local}
Li, T., Zhang, H., Zhang, P.: Local existence for the dumbbell model of
  polymeric fluids.
\newblock Comm. Partial Differential Equations {29}(5-6), 903--923
  (2004).
  

\bibitem{lions2000global}
Lions, P.L., Masmoudi, N.: Global solutions for some {O}ldroyd models of
  non-{N}ewtonian flows.
\newblock Chinese Ann. Math. Ser. B {21}(2), 131--146 (2000).


\bibitem{lions2007global}
Lions, P.L., Masmoudi, N.: Global existence of weak solutions to some
  micro-macro models.
\newblock C. R. Math. Acad. Sci. Paris  {345}(1), 15--20 (2007).


\bibitem{lukacova2017global}
Luk\'{a}\v{c}ov\'{a}-Medvi\v{d}ov\'{a}, M., Mizerov\'{a}, H., Ne\v{c}asov\'{a} \v{S}., Renardy, M.:
Global existence result for the generalized {P}eterlin viscoelastic model.
\newblock SIAM J. Math. Anal. {49}(4),  2950--2964 (2017).


\bibitem{luo2017global}
Luo, W., Yin, Z.: Global existence and well-posedness for the {FENE} dumbbell
  model of polymeric flows.
\newblock Nonlinear Anal. Real World Appl. {37}, 457--488 (2017).


\bibitem{masmoudi2008well}
Masmoudi, N.: Well-posedness for the {FENE} dumbbell model of polymeric flows.
\newblock Comm. Pure Appl. Math. {61}(12), 1685--1714 (2008).


\bibitem{masmoudi2013global}
Masmoudi, N.: Global existence of weak solutions to the {FENE} dumbbell model
  of polymeric flows.
\newblock Invent. Math. {191}(2), 427--500 (2013).


\bibitem{muha2014note}
Muha, B.: A note on the trace theorem for domains which are locally subgraph of a {H}\"{o}lder continuous function.
\newblock Netw. Heterog. Media. {9}(1), 191--196 (2014).


\bibitem{muha2014existence}
Muha, B., \v{C}ani\'{c}, S.: Existence of a solution to a
  fluid-multi-layered-structure interaction problem.
\newblock J. Differential Equations {256}(2), 658--706 (2014).


\bibitem{muha2013existence}
Muha, B., \v{C}ani\'{c}, S.: Existence of a weak solution to a nonlinear
  fluid-structure interaction problem modeling the flow of an incompressible,
  viscous fluid in a cylinder with deformable walls.
\newblock Arch. Ration. Mech. Anal. {207}(3), 919--968 (2013).


\bibitem{muha2019existence}
Muha, B., Schwarzacher, S.: Existence and regularity for weak solutions for a
  fluid interacting with a non-linear shell in 3d.
\newblock arXiv preprint arXiv:1906.01962  (2019).


\bibitem{MR3575853}
N\"{a}gele, P., R\ocirc{u}\v{z}i\v{c}ka, M., Lengeler, D.: Functional setting
  for unsteady problems in moving domains and applications.
\newblock Complex Var. Elliptic Equ. {62}(1), 66--97 (2017).


\bibitem{otto2008continuity}
Otto, F., Tzavaras, A.E.: Continuity of velocity gradients in suspensions of
  rod-like molecules.
\newblock Comm. Math. Phys. {277}(3), 729--758 (2008).


\bibitem{peralta2019analysis}
Peralta, G., Kunisch, K.: Analysis of a nonlinear fluid-structure interaction model with  mechanical dissipation and delay.
\newblock Nonlinearity. {32}(12), 5110--5149 (2019).


\bibitem{renardy1991existence}
Renardy, M.: An existence theorem for model equations resulting from kinetic
  theories of polymer solutions.
\newblock SIAM J. Math. Anal. {22}(2), 313--327 (1991).


\bibitem{roquefort2001quelques}
Roquefort, A.: Sur quelques questions li{\'e}es aux mod{\'e}les non
  lin{\'e}aires de coques minces.
\newblock Ph.D. thesis, Paris 6 (2001).

\bibitem{zhang2006local}
Zhang, H., Zhang, P.: Local existence for the {FENE}-dumbbell model of
  polymeric fluids.
\newblock Arch. Ration. Mech. Anal. {181}(2), 373--400 (2006).


\end{thebibliography}

\end{document}